\documentclass[12pt,reqno]{amsart}
\usepackage{amssymb,amscd,url,amsmath}
\usepackage{xcolor}    %% allows colors, used for edit comments

\RequirePackage[pdfa, psdextra]{hyperref}  
\hypersetup{
     colorlinks=true,
     linkcolor=blue,
     urlcolor=blue,
     }
\providecommand{\texorpdfstring}[2]{#1}  %% Prevent error if hyperref not loaded.
\begin{document}

\allowdisplaybreaks

%%%%%%%%%%%%%%%%%%%%%%%%%%%%%%%%%%%%%%%%%%%%%%%%%%%%%%%%%%%%%%%%%%%%%%
%% Title and Author Information

\title{Orbits on K3 Surfaces of Markoff Type}      
      
\date{\today}

\author[E. Fuchs]{Elena Fuchs}
\email{efuchs@math.ucdavis.edu}
\address{Department of Mathematics,
  University of California Davis, One Shields Ave, Davis, CA 95616 USA.
   ORCID: https://orcid.org/0000-0002-0978-5137}

\author[M. Litman]{Matthew Litman}
\email{mclitman@ucdavis.edu}
\address{Department of Mathematics,
  University of California Davis, One Shields Ave, Davis, CA 95616 USA.
  ORCID: https://orcid.org/0000-0002-0908-9369}

\author[J.H. Silverman]{Joseph H. Silverman}
\email{jhs@math.brown.edu}
\address{Department of Mathematics, Box 1917
  Brown University, Providence, RI 02912 USA.
  MR Author ID: 162205.
  ORCID: https://orcid.org\allowbreak/0000-0003-3887-3248}

\author[A. Tran]{Austin Tran}
\email{austran@ucdavis.edu}
\address{Department of Mathematics,
  University of California Davis, One Shields Ave, Davis, CA 95616 USA.
    ORCID: https://orcid.org/0000-0003-3725-7822}

\subjclass[2010]{Primary: 37P55; Secondary: 14J28, 37F80, 37P25, 37P35 }
\keywords{K3 surface, arithmetic dynamics, finite orbits, orbits over finite fields}
\thanks{Silverman's research supported by Simons Collaboration Grant \#712332.}

%%%%%%%%%%%%%%%%%%%%%%%%%%%%%%%%%%%%%%%%%%%%%%%%%%%%%%%%%%%%%%%%%%%%%%
\hyphenation{ca-non-i-cal semi-abel-ian Mar-koff}
%%%%%%%%%%%%%%%%%%%%%%%%%%%%%%%%%%%%%%%%%%%%%%%%%%%%%%%%%%%%%%%%%%%%%%
% Theorem environments
% * surpresses numbering

\newtheorem{theorem}{Theorem}[section]
\newtheorem{lemma}[theorem]{Lemma}
\newtheorem{conjecture}[theorem]{Conjecture}
\newtheorem{proposition}[theorem]{Proposition}
\newtheorem{corollary}[theorem]{Corollary}

\theoremstyle{definition}
\newtheorem*{claim}{Claim}
\newtheorem{definition}[theorem]{Definition}
\newtheorem*{intuition}{Intuition}
\newtheorem{example}[theorem]{Example}
\newtheorem{remark}[theorem]{Remark}
\newtheorem{question}[theorem]{Question}

\theoremstyle{remark}
\newtheorem*{acknowledgement}{Acknowledgements}

%%%%%%%%%%%%%%%%%%%%%%%%%%%%%%%%%%%%%%%%%%%%%%%%%%%%%%%%%%%%%%%%%%%%%%

%%%%%%%% Set Up Environment for Notation %%%%%%%%%%%%%%
% This is currently set to allow quite wide items to be defined
\newenvironment{notation}[0]{%
  \begin{list}%
    {}%
    {\setlength{\itemindent}{0pt}
     \setlength{\labelwidth}{4\parindent}
     \setlength{\labelsep}{\parindent}
     \setlength{\leftmargin}{5\parindent}
     \setlength{\itemsep}{0pt}
     }%
   }%
  {\end{list}}

%%%%%%%% Set Up Environment for Parts in Theorems %%%%%%%%%%%%%%
\newenvironment{parts}[0]{%
  \begin{list}{}%
    {\setlength{\itemindent}{0pt}
     \setlength{\labelwidth}{1.5\parindent}
     \setlength{\labelsep}{.5\parindent}
     \setlength{\leftmargin}{2\parindent}
     \setlength{\itemsep}{0pt}
     }%
   }%
  {\end{list}}
% Use \Part{(a)}, instead of \item[(a)], to ensure upright font
\newcommand{\Part}[1]{\item[\upshape#1]}

%%%%%%%%%%%%%%%%%%
% Greek Alphabet %
%%%%%%%%%%%%%%%%%%
\renewcommand{\a}{\alpha}
\newcommand{\bfalpha}{{\boldsymbol{\alpha}}}
\renewcommand{\b}{\beta}
\newcommand{\bfbeta}{{\boldsymbol{\beta}}}
\newcommand{\g}{\gamma}
\renewcommand{\d}{\delta}
\newcommand{\e}{\epsilon}
\newcommand{\f}{\varphi}
\newcommand{\fhat}{\hat\varphi}
\newcommand{\bfphi}{{\boldsymbol{\f}}}
\renewcommand{\l}{\lambda}
\renewcommand{\k}{\kappa}
\newcommand{\lhat}{\hat\lambda}
\newcommand{\m}{\mu}
\newcommand{\bfmu}{{\boldsymbol{\mu}}}
\renewcommand{\o}{\omega}
\newcommand{\bfpi}{{\boldsymbol{\pi}}}
\renewcommand{\r}{\rho}
\newcommand{\bfrho}{{\boldsymbol{\rho}}}
\newcommand{\rbar}{{\bar\rho}}
\newcommand{\s}{\sigma}
\newcommand{\sbar}{{\bar\sigma}}
\renewcommand{\t}{\tau}
\newcommand{\z}{\zeta}

\newcommand{\D}{\Delta}
\newcommand{\G}{\Gamma}
\newcommand{\F}{\Phi}
\renewcommand{\L}{\Lambda}

%%%%%%%%%%%%%%%%%%%%
% Fraktur Alphabet %
%%%%%%%%%%%%%%%%%%%%
\newcommand{\ga}{{\mathfrak{a}}}
\newcommand{\gA}{{\mathfrak{A}}}
\newcommand{\gb}{{\mathfrak{b}}}
\newcommand{\gB}{{\mathfrak{B}}}
\newcommand{\gc}{{\mathfrak{c}}}
\newcommand{\gM}{{\mathfrak{M}}}
\newcommand{\gn}{{\mathfrak{n}}}
\newcommand{\gp}{{\mathfrak{p}}}
\newcommand{\gP}{{\mathfrak{P}}}
\newcommand{\gS}{{\mathfrak{S}}}
\newcommand{\gq}{{\mathfrak{q}}}

%%%%%%%%%%%%%%%%%%%
% Barred Alphabet %
%%%%%%%%%%%%%%%%%%%
\newcommand{\Abar}{{\bar A}}
\newcommand{\Ebar}{{\bar E}}
\newcommand{\kbar}{{\bar k}}
\newcommand{\Kbar}{{\bar K}}
\newcommand{\Pbar}{{\bar P}}
\newcommand{\Sbar}{{\bar S}}
\newcommand{\Tbar}{{\bar T}}
\newcommand{\gbar}{{\bar\gamma}}
\newcommand{\lbar}{{\bar\lambda}}
\newcommand{\ybar}{{\bar y}}
\newcommand{\phibar}{{\bar\f}}
\newcommand{\nubar}{{\overline\nu}}

%%%%%%%%%%%%%%%%%%%%%%%%%
% Calligraphic Alphabet %
%%%%%%%%%%%%%%%%%%%%%%%%%
\newcommand{\Acal}{{\mathcal A}}
\newcommand{\Bcal}{{\mathcal B}}
\newcommand{\Ccal}{{\mathcal C}}
\newcommand{\Dcal}{{\mathcal D}}
\newcommand{\Ecal}{{\mathcal E}}
\newcommand{\Fcal}{{\mathcal F}}
\newcommand{\Gcal}{{\mathcal G}}
\newcommand{\Hcal}{{\mathcal H}}
\newcommand{\Ical}{{\mathcal I}}
\newcommand{\Jcal}{{\mathcal J}}
\newcommand{\Kcal}{{\mathcal K}}
\newcommand{\Lcal}{{\mathcal L}}
\newcommand{\Mcal}{{\mathcal M}}
\newcommand{\Ncal}{{\mathcal N}}
\newcommand{\Ocal}{{\mathcal O}}
\newcommand{\Pcal}{{\mathcal P}}
\newcommand{\Qcal}{{\mathcal Q}}
\newcommand{\Rcal}{{\mathcal R}}
\newcommand{\Scal}{{\mathcal S}}
\newcommand{\Tcal}{{\mathcal T}}
\newcommand{\Ucal}{{\mathcal U}}
\newcommand{\Vcal}{{\mathcal V}}
\newcommand{\Wcal}{{\mathcal W}}
\newcommand{\Xcal}{{\mathcal X}}
\newcommand{\Ycal}{{\mathcal Y}}
\newcommand{\Zcal}{{\mathcal Z}}

%%%%%%%%%%%%%%%%%%%%%%%%%%%%
% Blackboard Bold Alphabet %
%%%%%%%%%%%%%%%%%%%%%%%%%%%%
\renewcommand{\AA}{\mathbb{A}}
\newcommand{\BB}{\mathbb{B}}
\newcommand{\CC}{\mathbb{C}}
\newcommand{\FF}{\mathbb{F}}
\newcommand{\GG}{\mathbb{G}}
\newcommand{\NN}{\mathbb{N}}
\newcommand{\PP}{\mathbb{P}}
\newcommand{\QQ}{\mathbb{Q}}
\newcommand{\RR}{\mathbb{R}}
\newcommand{\TT}{\mathbb{T}}
\newcommand{\ZZ}{\mathbb{Z}}

%%%%%%%%%%%%%%%%%%%%%%%%%%
% Boldface Math Alphabet %
%%%%%%%%%%%%%%%%%%%%%%%%%%
\newcommand{\bfa}{{\boldsymbol a}}
\newcommand{\bfb}{{\boldsymbol b}}
\newcommand{\bfc}{{\boldsymbol c}}
\newcommand{\bfd}{{\boldsymbol d}}
\newcommand{\bfe}{{\boldsymbol e}}
\newcommand{\ee}{{\boldsymbol{e}}} %% exp(2 \pi i .)
\newcommand{\bff}{{\boldsymbol f}}
\newcommand{\bfg}{{\boldsymbol g}}
\newcommand{\bfi}{{\boldsymbol i}}
\newcommand{\bfj}{{\boldsymbol j}}
\newcommand{\bfk}{{\boldsymbol k}}
\newcommand{\bfm}{{\boldsymbol m}}
\newcommand{\bfn}{{\boldsymbol n}}
\newcommand{\bfp}{{\boldsymbol p}}
\newcommand{\bfr}{{\boldsymbol r}}
\newcommand{\bfs}{{\boldsymbol s}}
\newcommand{\bft}{{\boldsymbol t}}
\newcommand{\bfu}{{\boldsymbol u}}
\newcommand{\bfv}{{\boldsymbol v}}
\newcommand{\bfw}{{\boldsymbol w}}
\newcommand{\bfx}{{\boldsymbol x}}
\newcommand{\bfy}{{\boldsymbol y}}
\newcommand{\bfz}{{\boldsymbol z}}
\newcommand{\bfA}{{\boldsymbol A}}
\newcommand{\bfF}{{\boldsymbol F}}
\newcommand{\bfB}{{\boldsymbol B}}
\newcommand{\bfD}{{\boldsymbol D}}
\newcommand{\bfG}{{\boldsymbol G}}
\newcommand{\bfI}{{\boldsymbol I}}
\newcommand{\bfM}{{\boldsymbol M}}
\newcommand{\bfP}{{\boldsymbol P}}
\newcommand{\bfQ}{{\boldsymbol Q}}
\newcommand{\bfT}{{\boldsymbol T}}
\newcommand{\bfU}{{\boldsymbol U}}
\newcommand{\bfX}{{\boldsymbol X}}
\newcommand{\bfY}{{\boldsymbol Y}}
\newcommand{\bfzero}{{\boldsymbol{0}}}
\newcommand{\bfone}{{\boldsymbol{1}}}

%%%%%%%%%%%%%%%%%%%%%%%%%%%%%%
% Miscellaneous New Commands %
%%%%%%%%%%%%%%%%%%%%%%%%%%%%%%
\newcommand{\Aut}{\operatorname{Aut}}
\newcommand{\Berk}{{\textup{Berk}}}
\newcommand{\Birat}{\operatorname{Birat}}
\newcommand{\Cage}{\operatorname{\textsf{Cage}}}
\newcommand{\characteristic}{\operatorname{char}}
\newcommand{\CircleNum}[1]{\raisebox{.5pt}{\textcircled{\raisebox{-.9pt} {\small#1}}}}
\newcommand{\codim}{\operatorname{codim}}
\newcommand{\ConnFib}{\operatorname{\textsf{ConnFib}}}
\newcommand{\Crit}{\operatorname{Crit}}
\newcommand{\crit}{{\textup{crit}}}
\newcommand{\critwt}{\operatorname{critwt}} % valency of a portrait
\newcommand{\Cycle}{\operatorname{Cycles}}
\newcommand{\diag}{\operatorname{diag}}
\newcommand{\dimEnd}{{M}}  % Dimension of End_d^N, i.e., End_d^N \subset \PP^\dimEnd
\newcommand{\Disc}{\operatorname{Disc}}
\newcommand{\Div}{\operatorname{Div}}
\newcommand{\Df}{{Df}}  % adjust spacing?
\newcommand{\Dom}{\operatorname{Dom}}
\newcommand{\dyn}{{\textup{dyn}}}
\newcommand{\End}{\operatorname{End}}
\newcommand{\PortEndPt}{{\textup{endpt}}} 
\newcommand{\END}{\smash[t]{\overline{\operatorname{End}}}\vphantom{E}}
\newcommand{\EndPoint}{E}  % endpoint of a component with no cycle
\newcommand{\ExtOrbit}{\mathcal{EO}} %% Extended orbit
\newcommand{\Fbar}{{\bar{F}}}
\newcommand{\fib}{{\textup{fib}}}
\newcommand{\Fix}{\operatorname{Fix}}
\newcommand{\Fiber}{\operatorname{Fiber}}
\newcommand{\Flatten}{\operatorname{\textsf{Flatten}}}
\newcommand{\FOD}{\operatorname{FOD}}
\newcommand{\FOM}{\operatorname{FOM}}
\newcommand{\Frame}{\operatorname{Fr}}
\newcommand{\Gal}{\operatorname{Gal}}
\newcommand{\genus}{\operatorname{genus}}
\newcommand{\GITQuot}{/\!/}
\newcommand{\GL}{\operatorname{GL}}
\newcommand{\Gp}{\Gcal}  %% This is the group of automorphisms. [BGS] use \Gamma, which we could use
\newcommand{\Gpplus}{\hat\Gcal}  %% This is the group of automorphisms that includes inversions
\newcommand{\GR}{\operatorname{\mathcal{G\!R}}}
\newcommand{\hhat}{{\hat h}}
\newcommand{\Hom}{\operatorname{Hom}}
\newcommand{\Index}{\operatorname{Index}}
\newcommand{\Image}{\operatorname{Image}}
\newcommand{\Isom}{\operatorname{Isom}}
\newcommand{\Jac}{\operatorname{Jac}}
\newcommand{\Ker}{{\operatorname{Ker}}}
\newcommand{\Ksep}{K^{\text{sep}}}  %% separable closure of K
\newcommand{\Length}{\operatorname{Length}}
\newcommand{\Lift}{\operatorname{Lift}}
\newcommand{\limstar}{\lim\nolimits^*}
\newcommand{\limstarn}{\lim_{\hidewidth n\to\infty\hidewidth}{\!}^*{\,}}
\def\LS#1#2{{\genfrac{(}{)}{}{}{#1}{#2}}} % Legendre symbol
\newcommand{\Mat}{\operatorname{Mat}}
\newcommand{\maxplus}{\operatornamewithlimits{\textup{max}^{\scriptscriptstyle+}}}
\newcommand{\MOD}[1]{~(\textup{mod}~#1)}
\newcommand{\Model}{\operatorname{Model}}
\newcommand{\Mor}{\operatorname{Mor}}
\newcommand{\Moduli}{\mathcal{M}}
\newcommand{\MODULI}{\overline{\mathcal{M}}}
\newcommand{\Mult}{\operatorname{\textup{\textsf{Mult}}}}
\newcommand{\Norm}{{\operatorname{\mathsf{N}}}}
\newcommand{\notdivide}{\nmid}
\newcommand{\normalsubgroup}{\triangleleft}
\newcommand{\NS}{\operatorname{NS}}
\newcommand{\onto}{\twoheadrightarrow}
\newcommand{\ord}{\operatorname{ord}}
\newcommand{\Orbit}{\mathcal{O}}
\newcommand{\Pcase}[3]{\par\noindent\framebox{$\boldsymbol{\Pcal_{#1,#2}}$}\enspace\ignorespaces}
\newcommand{\Per}{\operatorname{Per}}
\newcommand{\Perp}{\operatorname{Perp}}
\newcommand{\PrePer}{\operatorname{PrePer}}
\newcommand{\PGL}{\operatorname{PGL}}
\newcommand{\Pic}{\operatorname{Pic}}
\newcommand{\prim}{\textup{prim}}
\newcommand{\Prob}{\operatorname{Prob}}
\newcommand{\Proj}{\operatorname{Proj}}
\newcommand{\Qbar}{{\overline{\QQ}}}
\newcommand{\QR}[2]{\left( \dfrac{#1}{#2} \right) }
\newcommand{\rank}{\operatorname{rank}}
\newcommand{\Rat}{\operatorname{Rat}}
\newcommand{\Resultant}{\operatorname{Res}}
\newcommand{\Residue}{\operatorname{Residue}} %% residue
\renewcommand{\setminus}{\smallsetminus}
\newcommand{\sgn}{\operatorname{sgn}}
\newcommand{\Sing}{\operatorname{Sing}}
\newcommand{\SL}{\operatorname{SL}}
\newcommand{\Span}{\operatorname{Span}}
\newcommand{\Spec}{\operatorname{Spec}}
\renewcommand{\ss}{{\textup{ss}}}
\newcommand{\stab}{{\textup{stab}}}
\newcommand{\Stab}{\operatorname{Stab}}
\newcommand{\Support}{\operatorname{Supp}}
\newcommand{\Sym}{\operatorname{Sym}}  %% Symmetric group
\newcommand{\tors}{{\textup{tors}}}
\newcommand{\Trace}{\operatorname{Trace}}
\newcommand{\trianglebin}{\mathbin{\triangle}} % symmetric set difference
\newcommand{\tr}{{\textup{tr}}} % for K/k trace
\newcommand{\UHP}{{\mathfrak{h}}}    % Upper half plane
\newcommand{\val}{\operatorname{val}} % valency of a portrait
\newcommand{\wt}{\operatorname{wt}} %% weight of a portrait
\newcommand{\<}{\langle}
\renewcommand{\>}{\rangle}

\newcommand{\pmodintext}[1]{~\textup{(mod}~#1\textup{)}}
\newcommand{\ds}{\displaystyle}
\newcommand{\longhookrightarrow}{\lhook\joinrel\longrightarrow}
\newcommand{\longonto}{\relbar\joinrel\twoheadrightarrow}
\newcommand{\SmallMatrix}[1]{%
  \left(\begin{smallmatrix} #1 \end{smallmatrix}\right)}

\newcommand\Wcalf[1]{\Wcal^{(#1)}}  %% fiber of \Wcal
\newcommand\Ecalf[1]{\Ecal^{(#1)}}  %% fiber of \Ecal
\newcommand\Pcalf[1]{\Pcal^{(#1)}}  %% point in fiber of \Ecal
\newcommand\Ccalf[1]{\Ccal^{(#1)}}  %% important curve in (P^1)^3
\newcommand\Lcalf[1]{\Lcal^{(#1)}}  %% fibral linking set
\newcommand\Scalf[1]{\Scal^{(#1)}}  %% surface used in proof(s)

\begin{abstract}
Let $\mathcal{W}\subset\mathbb{P}^1\times\mathbb{P}^1\times\mathbb{P}^1$ be a surface given by the vanishing of a $(2,2,2)$-form. These surfaces admit three involutions coming from the three projections $\mathcal{W}\to\mathbb{P}^1\times\mathbb{P}^1$, so we call them \emph{tri-involutive K3} (TIK3) \emph{surfaces}. By analogy with the classical Markoff equation, we say that $\mathcal{W}$ is of \emph{Markoff type} (MK3) if it is symmetric in its three coordinates and invariant under double sign changes. An MK3 surface admits a group of automorphisms $\mathcal{G}$ generated by the three involutions, coordinate permutations, and sign changes. In this paper we study the $\mathcal{G}$-orbit structure of points on TIK3 and MK3 surfaces. Over finite fields, we study fibral connectivity and the existence of large orbits, analogous to work of Bourgain, Gamburd, Sarnak and others for the classical Markoff equation. For a particular $1$-parameter family of MK3 surfaces $\mathcal{W}_k$, we compute the full $\mathcal{G}$-orbit structure of $\mathcal{W}_k(\mathbb{F}_p)$ for all primes $p\le113$, and we use this data as a guide to find many finite $\mathcal{G}$-orbits in $\mathcal{W}_k(\mathbb{C})$, including a family of orbits of size $288$  parameterized by a curve of genus $9$.
\end{abstract}

%% Non-LaTex abstract
%% Let W be a surface of type (2,2,2). Such surfaces admit three involutions coming from the three projections onto two of the coordinates, so we call them tri-involutive K3 (TIK3) surfaces.  By analogy with the classical Markoff equation, we say that W is of Markoff type (MK3) if it is symmetric in its three coordinates and invariant under double sign changes. An MK3 surface admits a group of automorphisms G generated by the three involutions, coordinate permutations, and sign changes. In this paper we study the G-orbit structure of points on TIK3 and MK3 surfaces. Over finite fields, we study fibral connectivity and the existence of large orbits, analogous to work of Bourgain, Gamburd, Sarnak and others for the classical Markoff equation. For a particular 1-parameter family of MK3 surfaces W_k, we compute the full G-orbit structure of W(F_p) for all primes p up to 113, and we use this data as a guide to find many finite G-orbits in W_k(C), including a family of orbits of size 288 parameterized by a curve of genus 9.

\maketitle

\tableofcontents

%%%%%%%%%%%%%%%%%%%%%%%%%%%%%%%%%%%%%%%%%%%%%%%%%%%%%%%%%%%%%%%%%%%%%%
\section{Introduction}
\label{section:introduction}
%%%%%%%%%%%%%%%%%%%%%%%%%%%%%%%%%%%%%%%%%%%%%%%%%%%%%%%%%%%%%%%%%%%%%%
The classical Markoff equation is the affine surface
\begin{equation}
  \label{eqn:markoffeqnintro}
  \Mcal : x^2+y^2+z^2=3xyz.
\end{equation}
It admits three involutions coming from the three
projections~$\Mcal\to\AA^2$, and these three involutions, together
with double sign changes and coordinate permutations, generate the
automorphism group~$\Gp_\Mcal:=\Aut(\Mcal)$ of~$\Mcal$. A classical theorem of
Markoff~\cite{MR1510073}  says that the set of integer
solutions~$\Mcal(\ZZ)$ consists of two orbits, one ``small'' $\Gp_\Mcal$-orbit
containing the single point~$(0,0,0)$, and one ``large''
$\Gp_\Mcal$-orbit containing~$(1,1,1)$. 
\par
The orbit structure structure of~$\Mcal(\FF_p)$ under the
action of~$\Gp_\Mcal$ has been studied by a number of authors.
Baragar \cite{MR2686830}  conjectured that for every prime~$p$, there is only one large orbit in~$\Mcal(\FF_p)$,
and this was proved for almost all~$p$ by Bourgain--Gambard--Sarnak~\cite{MR3456887} and
subsequently for all sufficiently large~$p$ by Chen~\cite{arxiv2011.12940}. 
The proofs rely on an ingenious algorithm that jumps between differently
oriented fibers, using the Hasse--Weil estimate to say that if a point
on a ``vertical'' fiber has a large enough orbit, then one of the
``horizontal'' orbits consists of an entire ``horizontal'' fiber. The proof implicitly
relies on the fact that each fiber of~$\Mcal$ is a torus and that the
fibral automorphisms are toral translations (i.e.,~$\GG_m$-translations), which
in~\cite{MR3456887} are called rotations.  See
Section~\ref{section:markoffresults} for more details.
\par
The first goal of this paper is to study similar questions on an analogous
family of projective surfaces that admit three involutions. We define
the family of \emph{tri-involutive~K3 \textup{(TIK3)} surfaces} to be
the hypersurfaces
\begin{equation}
  \label{eqn:WinP1P1P1intro}
  \Wcal \subset \PP^1\times\PP^1\times\PP^1
\end{equation}
given by the vanishing of a~$(2,2,2)$-form. These surfaces have three involutions
\[
\s_1,\s_2,\s_3 : \Wcal \longrightarrow \Wcal
\]
coming from switching the sheets of the three double covers coming from the projections
\[
\pi_{12},\pi_{13},\pi_{23} : \Wcal \longrightarrow \PP^1\times\PP^1.
\]
The study of the geometry and arithmetic of these surfaces is of course not new; see
Section~\ref{section:TIK3resultssurvey} for a brief history. 

The first goal of this paper is to study the orbit structure of~$\Wcal(\FF_p)$ under the action of~$\Aut(\Wcal)$.
To do this, we start by analyzing the connectivity of the fibers of~$\Wcal(\FF_p)$ for the three projections
\[
\pi_{1},\pi_{2},\pi_{3} : \Wcal(\FF_p) \longrightarrow \PP^1(\FF_p).
\]
We prove the following fibral linking result, which is a TIK3 analogue of~\cite[Proposition~6]{MR3456887} for the Markoff equation. See Theorem~\ref{theorem:fiberconnectviathirdfiber} for further details and a proof.

\begin{theorem}
\label{proposition:fiberlinkintro}
Assume that $p>100$, and let~$\Wcal/\FF_p$ be a TIK3 surface. Let~$\Fcal_1$ and~$\Fcal_2$ be fibers of $\Wcal(\FF_p)$ for any two (possibly identical) of the three projections $\pi_1,\pi_2,\pi_3:\Wcal(\FF_p)\to\PP^1(\FF_p)$. Then there is a fiber~$\Fcal_3$ for one of the projections satisfying
\[
\Fcal_1\cap \Fcal_3 \ne \emptyset
\quad\text{and}\quad 
\Fcal_2\cap \Fcal_3 \ne \emptyset.
\]
\end{theorem}

Our second goal is inspired by the classification of finite orbits on Markoff-type surfaces over~$\CC$. For example, the papers~\cite{MR2254812,MR2649343,MR1767271,MR3253555} contain a detailed description of the~$(a,b,c,d)\in\CC$ for which the surface
\begin{equation}
\label{eqn:markoffabcd}
    x^2 + y^2 + z^2 + ax + by + cz + dxyz = 0.
\end{equation}
has one or more finite orbits. The existence of such orbits turns out to be related to algebraic solutions to Painlev\'{e} differential equations. It is likewise true~\cite{arxiv2012.01762} that a (non-degenerate) TIK3 surface~$\Wcal(\CC)$ has only finitely many finite orbits, but the methods used to classify the orbits for Markoff-type equations do not seem easily applicable to the TIK3 situation. 
\par
Generically, the automorphism group of~$\Wcal$ is generated by the three
automorphisms. Since the Markoff equation~\eqref{eqn:markoffeqnintro}
admits additional automorphisms, we consider an analogous
family of~TIK3 surfaces, which we call 
\emph{Markoff-type K3 \textup{(MK3)} surfaces}.
These are the TIK3 surfaces~\eqref{eqn:WinP1P1P1intro} that are invariant
under coordinate permutations and double sign changes. See Proposition~\ref{proposition:nondegK3Markofftype}
for a description of the full $4$-dimensional family of~MK3 surfaces.
\par
A typical example, which we use as a prototype, is the following
one-parameter family of MK3-surfaces~$\Wcal_k$. For
non-zero~$k$, we define~$\Wcal_k$ to be the projective closure
in~$(\PP^1)^3$ of the affine surface
\begin{equation}
\label{eqn:Wkintro}
\Wcal_k :  x^2 + y^2 + z^2 + x^2 y^2 z^2 + k x y z = 0.
\end{equation}
\par
In order to understand the orbit structure in~$\Wcal_k(\FF_p)$, we computed all orbits for~$p\le113$ and all~$k\in\FF_p^*$; see~Section~\ref{section:totalorbits} and Appendix~\ref{appendix:finitefieldtables}. We use these computations for two purposes. 
\par
First, by studying small orbit sizes that appear in~$\Wcal_k(\FF_p)$ for many different~$p$ and~$k$, we find patterns which we use to construct finite orbits in~$\Wcal_k(\CC)$. A full description of our findings is contained in Section~\ref{section:finiteorbitsoverC}; see especially Table~\ref{table:finiteorbschar0}. We illustrate by stating a few results, including some fairly large finite orbits that occur in $1$-parameter families:

\begin{proposition}
Let $\Wcal_k$ be the projective closure in~$(\PP^1)^3$ of the affine surface~\eqref{eqn:Wkintro}.
\begin{itemize}
    \item
    $\Wcal_{-4}(\QQ)$ contains an orbit of size~$4$, and~$\Wcal_4(\QQ)$ contains an orbit of size~$12$.
    \item
    $\Wcal_k\bigl(\QQ(i)\bigr)$ contains an orbit of size~$48$ for every~$k\in\QQ(i)$.
    \item
    There is a field~$K/\QQ$ of degree~$8$ and an element~$k\in{K}$ so that~$\Wcal_k(K)$ has an orbit of size~$144$.
    \item
    There is a field~$K/\QQ$ of degree~$8$ and an element~$k\in{K}$ so that~$\Wcal_k(K)$ has an orbit of size~$160$.
    \item
    There is a $k(t)\in\QQ(t)$ so that~$\Wcal_{k(t)}\bigl(\QQ(t)\bigr)$ has an orbit of size~$24$.
    \item
    There is a $k(t)\in\QQ(i,t)$ so that~$\Wcal_{k(t)}\bigl(\QQ(i,t)\bigr)$ has an orbit of size~$96$.
    \item
    There is an irreducible curve~$C/\QQ$ of genus~$9$ and an element~$k\in\QQ(C)$ in the function field of~$C$ so that~$\Wcal_k\bigl(\QQ(C)\bigr)$ has an orbit of size~$288$.
\end{itemize}
\end{proposition}

In the spirit of the many uniform boundedness theorems and conjectures in arithmetic geometry and arithmetic dynamics, we pose the following question:
\begin{question}
\label{question:NTIK3bd}
Does there exist a constant $N$ so that
\[
\#\{ P\in\Wcal_k(\CC) : \text{the orbit of $P$ is finite}\} \le N
\quad\text{for all $k\in\CC^*$?}
\]
More generally, does there exist a constant $N$ so that for every non-degenerate\footnote{See Definition~\ref{definition:nondegenTIK3}, but briefly, non-degeneracy means that the three involutions are well-defined.} TIK3 surface~$\Wcal$ we have
\[
\#\{ P\in\Wcal(\CC) : \text{the $\langle\s_1,\s_2,\s_3\rangle$-orbit of $P$ is finite}\} \le N?
\]
See Question~\ref{question:uniformbdness} for  a further discussion of uniform boundedness of finite orbits.
\end{question}
\par

Second, we investigate large orbits in~$\Wcal_k(\FF_p)$ to see if the methods employed in~\cite{MR3456887}  for the Markoff equation are potentially applicable to the MK3 setting.
The fiber-to-fiber jumping strategy employed by~\cite{MR3456887} uses the
fact, which they prove for~\eqref{eqn:markoffabcd} with~$(a,b,c,d)=(0,0,0,-3)$, that if a vertical fibral orbit is
sufficiently large, then at least one of the points in that vertical
orbit has a horizontal orbit that consists of the entire horizontal
fiber. (See Section~\ref{section:connectivitystrategy} and Remark~\ref{remark:BGSmethod} for further details.)  We are interested in the question of whether such a
fiber-to-fiber jumping strategy will work on the~MK3-surface~$\Wcal_k(\FF_p)$.  
In Section~\ref{section:fibralorbits} we show that the surface~$\Wcal_{1}(\FF_{53})$ has an orbit of size~$3456$, but that the fiber-to-fiber jumping strategy cannot be used to prove that this orbit consists of a single orbit. This suggests that additional ideas may be needed to prove the existence of a large orbit in~$\Wcal_k(\FF_p)$.

%%%%%%%%%%%%%%%%%%%%%%%%%%%%%%%%%%%%%%%%%%%%%%%%%%%%%%%%%%%%%%
\begin{acknowledgement}
The authors would like to thank Philip Boalch, Wei Ho, Ram Murty, and Igor Shparlinski for their helpful advice and Peter Sarnak for his encouragement. Calculations in this article were done using Magma~\cite{MR1484478} and GP-PARI~\cite{PARI2}.
\end{acknowledgement}

%%%%%%%%%%%%%%%%%%%%%%%%%%%%%%%%%%%%%%%%%%%%%%%%%%%%%%%%%%%%%%%%%%%%%%
\section{A brief survey of related work on the Markoff equation}
\label{section:markoffresults}
%%%%%%%%%%%%%%%%%%%%%%%%%%%%%%%%%%%%%%%%%%%%%%%%%%%%%%%%%%%%%%%%%%%%%%

\begin{definition}
Let~$a\in{K^*}$ and~$k\in{K}$. The associated \emph{Markoff equation} is
\begin{equation}
  \label{eqn:Makx2y2z2axyzkdef}
  \Mcal_{a,k} : x^2 + y^2 + z^2 = axyz + k,
\end{equation}
and $\Gp_\Mcal$ denotes the group of automorphisms of~$\Mcal_{a,k}$ generated by the involutions~$\s_1,\s_2,\s_3$, double sign changes, and permutations of the coordinates.
\end{definition}

\begin{theorem}
\label{theorem:Markoff}
\begin{parts}
\Part{(a)}
\textup{(Markoff \cite{MR1510073})}
\[
\Mcal_{3,0}(\ZZ) = \bigl\{(0,0,0)\bigr\} \cup \Gp_\Mcal\cdot(1,1,1).
\]
\Part{(b)}
More generally, for all~$a,k\in\ZZ$ with~$a\ne0$, there is a finite
set of points~$P_1,\ldots,P_r\in\Mcal_{a,k}(\ZZ)$ such that
\[
\Mcal_{a,k}(\ZZ) = \bigcup_{i=1}^r \Gp_\Mcal\cdot{P_i}.
\]
\end{parts}
\end{theorem}

\begin{conjecture}
\label{conjecture:M1Fpconnected}
\textup{(Baragar \cite[Section~V.3]{MR2686830}, Bourgain--Gambard--Sarnak \cite{arxiv1607.01530,MR3456887})}
For all primes~$p\ge5$ we have
\[
\Mcal_{3,0}(\FF_p) = \bigl\{(0,0,0)\bigr\} \cup \bigl(\Gp_\Mcal\cdot(1,1,1)\bigr).
\]
\end{conjecture}

As noted in Theorem~\ref{theorem:Markoff}(b), 
the set~$\Mcal_{a,k}(\ZZ)$ generally consists of finitely many orbits. However,  we may still ask to what extent the points in~$\Mcal_{a,k}(\FF_p)$ lift to points in~$\Mcal_{a,k}(\ZZ)$, or alternatively, to what extent~$\Mcal_{a,k}(\FF_p)$ is essentially a single~$\Gcal_\Mcal$-orbit. One difficulty that occurs comes from  finite orbits in in~$\Mcal_{a,k}(\Qbar)$, since their mod~$p$ reduction leads to (small) finite orbits in various~$\Mcal_{a,k}(\FF_p)$. This leads to the following conjectures.

\begin{conjecture}
\label{conjecture:MakFpmostlyconnected}
Let $a,k\in\ZZ$. 
\begin{parts}
\Part{(a)}
There is a constant $M_1(a,k)$ such that for all primes~$p\nmid{a}$ we have
\[
\#\Mcal_{a,k}(\FF_p) \le \#\Bigl(\text{largest $\Gp_\Mcal$-orbit in $\Mcal_{a,k}(\FF_p)$}\Bigr) + M_1(a,k).
\]
\Part{(b)}
If $\#\Mcal_{a,k}(\ZZ)=\infty$, then there is a
constant~$M_2(a,k)$ such that for all primes~$p\nmid{a}$ we have
\[
\#\Mcal_{a,k}(\FF_p) \le \#\bigl( \Mcal_{a,k}(\ZZ)\bmod p \bigr) + M_2(a,k).
\]
\end{parts}
\textup(One might further ask whether~$M_1(a,k)$ and~$M_2(a,k)$ may be chosen independently of~$a$ and~$k$.\textup)
\end{conjecture}

Bourgain--Gambard--Sarnak and Chen have a number of results related to 
Conjectures~\ref{conjecture:M1Fpconnected} and~\ref{conjecture:MakFpmostlyconnected}, including the following:

\begin{theorem}
\label{theorem:BGSthms12}
\begin{parts}
  \Part{(a)}
  \cite[Theorem~1]{MR3456887}
  \[
  \#\Mcal_{3,0}(\FF_p) - \#\bigl(\Gp_\Mcal\cdot(1,1,1)\bigr) = p^{o(1)},
  \quad\text{as $p\to\infty$.}
  \]
  \Part{(b)}
  \cite[Theorem~2]{MR3456887}
  Conjecture~$\ref{conjecture:M1Fpconnected}$ holds for all but
  possibly $X^{o(1)}$ primes $p\le{X}$, as~$X\to\infty$.
  \Part{(c)}
  \cite{arxiv2011.12940}
  Conjecture~$\ref{conjecture:M1Fpconnected}$ holds for all but finitely many primes~$p$.
\end{parts}
\end{theorem}

\begin{remark}
Chen's result (Theorem~\ref{theorem:BGSthms12}(c)) supersedes the results
of Bourgain--Gambard--Sarnak (Theorem~\ref{theorem:BGSthms12}(a,b)), but Chen's proof depends strongly on the particular form of the equation~$\Mcal_{3,0}$. More precisely, Chen proves that the orbit of~$(1,1,1)$ in~$\Mcal_{3,0}(\FF_p)$ has cardinality divisible by~$p$. This combined with~\cite[Theorem~1]{MR3456887} and the fact that~$\#\Mcal_{3,0}(\FF_p)\equiv1\pmodintext{p}$ yields the desired result. However, we note that the methods used to prove the results in~\cite{MR3456887} should extend to give versions of Conjecture~\ref{conjecture:MakFpmostlyconnected} analogous to Theorem~\ref{theorem:BGSthms12}(a,b).
\end{remark}

Other recent notable results include the following: 
\begin{itemize}
\item
  Konyagin--Makarychev--Shparlinski--Vyugin \cite{MR4112680}
  improve Theorem~\ref{theorem:BGSthms12}, and their methods should extend to more general Markoff equations:
\[
\#\Mcal_{3,0}(\FF_p)\setminus\bigl(\Gp_\Mcal\cdot(1,1,1)\bigr)
\le \exp\left( (\log p)^{2/3+o(1)} \right). 
\]
\item
  Given a pseudo-Anosov element $g\in \mathrm{Out}(\mathbf{F}_2)$, $g$ induces a permutation $g_p$ on $\Mcal_{1,k}(\FF_p)$ for each prime $p$. Cerbu--Gunther--Magee--Peilen
  \cite{MR4074547} prove that asymptotically, the action of $g_p$ on $\Mcal_{1,k}(\FF_p)$ has an orbit of size at least $\frac{\log(p)}{\log|\lambda|}+O_g(1)$, where $\lambda$ is the eigenvalue of largest modulus of $g$ when viewed as an element of $\mathrm{GL}_2(\ZZ)$.
\item
  M. de Courcy-Ireland and S.  Lee \cite{arxiv1812.07275} verify strong approximation for the Markoff surface for all primes $p<3000$. Additionally, they completely characterize the orbit structure of the degenerate Cayley cubic, $\Mcal_{1,4}(\FF_p)$, providing both the number of orbits as well as their sizes, given in terms of divisors of $p^2-1$.
\item
  M. de Courcy-Ireland and M.  Magee \cite{arxiv1811.00113} demonstrate that the eigenvalues of the family of Markoff graphs modulo $p$ converge to the Kesten-McCay measure, which is a heuristic indicator that Markoff graphs are suitably ``random". This also provides a (very) weak bound on the spectral gap of such graphs.
\item
  M. de Courcy-Ireland \cite{arxiv2105.12411} shows that if $p \equiv 1 \pmod{4}$ or if $p \equiv 1,2$ or $4 \pmod{7}$, then the Markoff graph mod $p$ is not planar.
\item
  A. Gamburd , M.  Magee and R.  Ronan \cite{MR4024562} prove that the counting function for the
  number of integer solutions on
  $x_1^2+\cdots+x_n^2=ax_1\cdots{x_n}+k$, excluding potential
  exceptional sets, is asymptotic to a constant multiple
  of~$(\log{R})^\beta$.
\end{itemize}

%%%%%%%%%%%%%%%%%%%%%%%%%%%%%%%%%%%%%%%%%%%%%%%%%%%%%%%%%%%%%%%%%%%%%%
\section{Tri-Involutive K3 (TIK3) Surfaces}
\label{section:TIK3surfaces}
%%%%%%%%%%%%%%%%%%%%%%%%%%%%%%%%%%%%%%%%%%%%%%%%%%%%%%%%%%%%%%%%%%%%%% 

\begin{definition}
\label{definition:nondegenTIK3}
A \emph{Tri-Involutive K3 \textup{(TIK3)} Surface} is a~K3 surface
\[
\Wcal = \{\overline{F}=0\} \subset\PP^1\times\PP^1\times\PP^1
\]
defined by a~$(2,2,2)$-form
\begin{equation}
  \label{eqn:222form}
  \overline{F}(X_1,X_2;Y_1,Y_2;Z_1,Z_2) \in K[X_1,X_2;Y_1,Y_2;Z_1,Z_2].
\end{equation}
For distinct~$i,j\in\{1,2,3\}$, we denote the various projections
of~$\Wcal$ onto one or two copies of~$\PP^1$ by
\[
\pi_i : \Wcal\longrightarrow\PP^1
\quad\text{and}\quad
\pi_{ij} : \Wcal\longrightarrow\PP^1\times\PP^1.
\]
We say that the TIK3 is \emph{non-degenerate} if
it satisfies the following two conditions\textup:
\begin{parts}
\Part{(i)}
The projection maps~$\pi_{12},\pi_{13},\pi_{23}$ are
finite.\footnote{We note that~$\pi_{12},\pi_{13},\pi_{23}$ are finite if and only if
  their fibers are $0$-dimensional, in which case they are maps of
  degree~$2$.}
\Part{(ii)}
The generic fibers of the projection maps~$\pi_1,\pi_2,\pi_3$ are
smooth curves, in which case the smooth fibers are necessarily curves
of genus~$1$, since they are~$(2,2)$ curves in~$\PP^1\times\PP^1$.
\end{parts}
\end{definition}

To ease notation, we write~$\PP^1=\AA^1\cup\{\infty\}$, and we let
\[
F(x,y,z) = \overline{F}(x,1;y,1;z,1).
\]
Then~$\Wcal$ is the closure in~$(\PP^1)^3$ of the affine surface, which by abuse of notation we also denote by~$\Wcal$,
\[
\Wcal : F(x,y,z)=0.
\]

\begin{definition}
Let~$\Wcal$ be a TIK3 surface with projections~$\pi_1,\pi_2,\pi_3:\Wcal\to\PP^1$
We define a \emph{fiber of~$\Wcal$} to be a set of the form
\[
\pi_i^{-1}(t)\quad\text{for some $i\in\{1,2,3\}$ and some $t\in\PP^1$.}
\]
Thus fibers may lie in any of three different directions, and we may view~$\Wcal$ as being triply cross-hatched by the various fibers. We denote the set of fibers by
\[
\Fiber(\Wcal) = \{\text{fibers of $\Wcal$}\}.
\]
If we need to refer to fibers over a particular point and corresponding to a particular projection, we use the following more
precise notation. We denote the fibers of~$\pi_1,\pi_2,\pi_3:\Wcal\to\PP^1$ over points~$x_0,y_0,z_0\in\PP^1$ by, respectively,
\[
\Wcalf1_{x_0} = \pi_1^{-1}(x_0),\qquad
\Wcalf2_{y_0} = \pi_2^{-1}(y_0),\qquad
\Wcalf3_{z_0} = \pi_3^{-1}(z_0).
\]
For~$P=(x_P,y_P,z_P)\in\Wcal$, we let
\[
\Wcalf1_{P} = \Wcalf1_{x_P},\quad
\Wcalf2_{P} = \Wcalf2_{y_P},\quad
\Wcalf3_{P} = \Wcalf3_{z_P}.
\]
\end{definition}

\begin{definition}
\label{definition:sigmakswapsheets}
Let~$\Wcal$ be a non-degenerate TIK3 surface.  For distinct~$i,j,k\in\{1,2,3\}$, we write
\begin{equation}
  \label{eqn:skWtoWinvolTIK3}
  \s_{k} : \Wcal \longrightarrow \Wcal
\end{equation}
for the involution that swaps the sheets of~$\pi_{ij}$,
i.e.,~$\s_{k}\in\Aut(\Wcal)$ is the unique non-identity automorphism
satisfying
\[
\pi_{ij}\circ\s_{k} = \pi_{ij}.
\]
\end{definition}

The automorphism group of a TIK3 surface~$\Wcal$ contains the non-commuting involutions~$\s_1,\s_2,\s_3$, and depending on the symmetries of~$\Wcal$'s defining polynomial~$F$, the automorphism group may contain additional automorphisms. Typical examples include symmetry in~$x,y,z$ that allows permutation of the coordinates, and power symmetry that allows the signs of two of~$x,y,z$ to be reversed. For example, the Markoff equation~\eqref{eqn:markoffeqnintro} permits these extra automorphisms; and in Section~\ref{section:mk3surfaces} we consider analogous TIK3 surfaces. In any case, we will be interested in subgroups of the automorphism group that move points around individual fibers. 

\begin{definition}
Let~$\Wcal$ be a non-degenerate TIK3 surface, let $\Gp\subseteq\Aut(\Wcal)$ be a group of automorphisms of~$\Wcal$, and let~$\Fcal\in\Fiber(\Wcal)$ be a fiber of~$\Wcal$. We denote the stabilzer of~$\Fcal$ by
\[
\Gp_\Fcal = \bigl\{ \f\in\Gp : \f(\Fcal)=\Fcal\bigr\}.
\]
We further define \emph{fibral automorphism groups} in each of the three directions by
\begin{align*}
\Gp^{(1)} &= \bigl\{ \f\in\Gp : \f(\Wcalf1_x)=\Wcalf1_x~\text{for all $x\in\PP^1$} \bigr\},\\
\Gp^{(2)} &= \bigl\{ \f\in\Gp : \f(\Wcalf2_y)=\Wcalf2_y~\text{for all $y\in\PP^1$} \bigr\},\\
\Gp^{(3)} &= \bigl\{ \f\in\Gp : \f(\Wcalf3_z)=\Wcalf3_z~\text{for all $z\in\PP^1$} \bigr\}.
\end{align*}
\end{definition}

For example, if~$\{i,j,k\}=\{1,2,3\}$, then~$\s_i,\s_j\in\Gp^{(k)}$, since~$\s_i$ and~$\s_j$ map the~$k$-direction fibers to themselves.

\begin{definition}
Let~$\Wcal$ be a non-degenerate TIK3 surface, let $\Gp\subseteq\Aut(\Wcal)$ be a group of automorphisms of~$\Wcal$, and let~$P_0=(x_0,y_0,z_0)\in\Wcal(K)$. The \emph{$\Gp$-orbit of~$P$} is
\[
\Gp \cdot P = \bigl\{ \f(P) : \f\in\Gp \bigr\}.
\]
The \emph{fibral $\Gp$-orbits of~$P$} are
\[
\Gp^{(k)} \cdot P = \bigl\{ \f(P) : \f\in\Gp^{(k)} \bigr\}\quad\text{for $k=1,2,3$.}
\]
\end{definition}

%%%%%%%%%%%%%%%%%%%%%%%%%%%%%%%%%%%%%%%%%%%%%%%%%%%%%%%%%%%%%%%%%%%%%%
\section{A strategy for proving that \texorpdfstring{$\Wcal(\FF_q)$}{WFq} has a large \texorpdfstring{$\Gp$}{G}-connected component}
\label{section:connectivitystrategy}
%%%%%%%%%%%%%%%%%%%%%%%%%%%%%%%%%%%%%%%%%%%%%%%%%%%%%%%%%%%%%%%%%%%%%%
In this section we consider a non-degenerate TIK3-surface~$\Wcal$ defined over a finite field~$\FF_q$,
and a group of automorphisms~$\Gp\subseteq\Aut(\Wcal)$.

\begin{definition}
Let~$t\in\PP^1(\FF_q)$, and let~$i\in\{1,2,3\}$. We say that the fiber~$\Wcalf{i}_t(\FF_q)$ is \emph{$\Gp$-fiber connected} if $\Gp^{(i)}$ acts transitively on~$\Wcalf{i}_t(\FF_q)$. Following terminology from~\cite{arxiv1607.01530}, we define the \emph{$\Gp$-cage of~$\Wcal(\FF_q)$} to be the set
\begin{multline*}
\Cage_\Gp\bigl(\Wcal(\FF_q)\bigr) \\
=
\left\{ P\in\Wcal(\FF_q) :
\begin{tabular}{@{}l@{}}
at least one of
$\Wcalf1_{P}(\FF_q)$, $\Wcalf2_{P}(\FF_q)$, \\and $\Wcalf3_{P}(\FF_q)$ 
is $\Gp$-fiber connected\\
\end{tabular}
\right\}.
\end{multline*}
We denote the set of $\Gp$-connected fibers by
\[
\ConnFib_\Gp\bigl(\Wcal(\FF_q)\bigr)
=
\left\{ \Wcalf{i}_t(\FF_q) : 
\begin{array}{@{}l@{}}
i\in\{1,2,3\},\, t\in\PP^1(\FF_q),\\
\text{$\Wcalf{i}_t(\FF_q)$ is $\Gp$-fiber connected} \\
\end{array}
\right\}.
\]
With this notation, an alternative description of the cage is as the union of the points in the fibers in~$\ConnFib_\Gp\bigl(\Wcal(\FF_q)\bigr)$.
\end{definition}

The starting point used in~\cite{arxiv1607.01530} to prove that the Markoff graph $\Mcal_{3,0}(\FF_q)$ is connected is to show that the associated cage is connected. This is done via a process that jumps from one connected fiber to another using a version of the following property:

\begin{definition}
We say that~$\Wcal(\FF_q)$ has the \emph{fiber-jumping property} if for all fibers~$\Fcal_1$ and~$\Fcal_2$ of~$\Wcal(\FF_q)$ there exists a
$\Gp$-connected fiber~$\Fcal_3\in\ConnFib\bigl(\Wcal(\FF_q)\bigr)$ satisfying
\[
\Fcal_1\cap\Fcal_3\ne\emptyset
\quad\text{and}\quad
\Fcal_2\cap\Fcal_3\ne\emptyset.
\]
\end{definition}

As described in~\cite{arxiv1607.01530}, the fiber-jumping property implies that the cage is connected. For the convenience of the reader, we recall the short proof.

\begin{proposition}
\label{proposition:fibjumimpliesconnected}
Suppose that $\Wcal(\FF_q)$ has the fiber-jumping property.
Then for all $P,Q\in\Cage_\Gp\bigl(\Wcal(\FF_q)\bigr)$ there exists an automorphism~$\g\in\Gp$ such that $\g(Q)=P$.
\end{proposition}
\begin{proof}
The fact that~$P$ and~$Q$ are in the $\Gp$-cage means that they lie on connected fibers, so we can find indices~$i$ and~$j$ so that
\begin{equation}
\label{eqn:GiPGjQ} 
\Gp^{(i)}\cdot P = \Wcalf{i}_P(\FF_q)
\quad\text{and}\quad
\Gp^{(j)}\cdot Q = \Wcalf{j}_Q(\FF_q).
\end{equation}
We apply the assumption that $\Wcal(\FF_q)$ has the fiber-jumping property to the fibers~$\Wcalf{i}_P(\FF_q)$ and~$\Wcalf{j}_Q(\FF_q)$. This allows us to find a connected fiber~$\Fcal\in\ConnFib\bigl(\Wcal(\FF_q)\bigr)$ satisfying
\begin{equation}
\label{eqn:WfiPWfjQ}
\Wcalf{i}_P(\FF_q) \cap \Fcal \ne \emptyset
\quad\text{and}\quad
\Wcalf{j}_Q(\FF_q)\cap \Fcal \ne \emptyset.
\end{equation}
We choose any point~$R\in\Fcal$. The connectivity of~$\Fcal$ tells us that~$\Fcal=\Wcalf{k}_R(\FF_q)=\Gp^{(k)}\cdot{R}$ for some index~$k$. Then~\eqref{eqn:GiPGjQ} and~\eqref{eqn:WfiPWfjQ} say that we can find points
\[
S \in \Gp^{(i)}\cdot P \cap \Gp^{(k)}\cdot{R}
\quad\text{and}\quad 
T \in  \Gp^{(j)}\cdot Q \cap \Gp^{(k)}\cdot{R}.
\]
In particular, there are automorphisms~$\g_1,\g_2,\g_3,\g_4\in\Gp$ satisfying
\[
S = \g_1 P = \g_2 R \quad\text{and}\quad T = \g_3 Q = \g_4 R.
\]
This yields
\[
P = \g_1^{-1} \g_2 R = \g_1^{-1} \g_2 \g_4^{-1} \g_3 Q,
\]
which completes the proof that~$P\in\Gp\cdot{Q}$.
\end{proof}

The strategy that is employed in~\cite{arxiv1607.01530} to prove that the large component of the Markoff graph $\Mcal_{3,0}(\FF_q)$ is connected has several steps. We reformulate these steps for TIK3-surfaces, retaining (and expanding on) their chess terminology. 
\begin{description}
    \item[Setting the board (Cage connectivity)] \hfill\newline
    The cage $\Cage_\Gp\bigl(\Wcal(\FF_q)\bigr)$ is $\Gp$-connected. 
    \item[End game (Large fibral orbits)]\hfill\newline
    Let $P\in\Wcalf{i}_t(\FF_q)$ be a point whose fibral orbit~$\Gp^{(i)}\cdot{P}$ is moderately large. Then~$\Gp^{(i)}\cdot{P}$ contains a point of the cage, i.e., it intersects a $\Gp$-connected fiber.
    \item[Middle game (Small fibral orbits)]\hfill\newline
    Let $P\in\Wcalf{i}_t(\FF_q)$ be a point whose fibral orbit~$\Gp^{(i)}\cdot{P}$ is of small, but non-negligible, size. Then~$\Gp^{(i)}\cdot{P}$ contains a point lying in a fibral orbit of strictly larger size.
    \item[Opening (Tiny fibral orbits)]\hfill\newline
    There are no non-trivial points $P\in\Wcalf{i}_t(\FF_q)$ whose fibral orbit~$\Gp^{(i)}\cdot{P}$ is tiny.
 \end{description}

 \begin{remark}[The Bourgain--Gamburd--Sarnak Connectivity Proof for the Markoff Equation]
 \label{remark:BGSmethod}
 We briefly sketch the connectivity proof for
 \[
 \Mcal^*(\FF_p) = \Mcal_{3,0}(\FF_p) \setminus (0,0,0)
 \]
 in~\cite{arxiv1607.01530}. They prove connectivity using the subgroup~$\Gp\subset\Aut(\Mcal_{3,0})$ generated by the compositions
 \[
 \rho^{(i)} = \f_i\circ\t_{jk},\quad\text{where~$\{i,j,k\}=\{1,2,3\}$.}
 \]
 They call~$\rho^{(i)}$ a \emph{rotation}, since it acts on the fibers~$(\Mcal_{3,0})_t^{(i)}$ via a $2$-by-$2$ (rotation) matrix acting on the~$jk$-coordinates. Writing~$\rho^{(i)}_t$ for the restriction of~$\rho^{(i)}$ to this fiber, they note that the order of~$\rho^{(i)}_t$ divides one of~$p-1$,~$p$, or~$p+1$, with the exact order depending on the eigenvalues of the matrix~$\rho^{(i)}_t$. It follows that
 \[
 (\Mcal_{3,0})_t^{(i)}(\FF_p) \subset \Cage\bigl(\Mcal_{3,0}(\FF_p)\bigr)
 \quad\Longleftrightarrow\quad
 \text{$\rho^{(i)}_t$ has maximal order.}
 \]
 \par
 The first step in proving that~$\Mcal^*(\FF_p)$ is $\Gp$-connected is an argument that uses curve coverings, point counting, and inclusion/exclusion to show that~$\Mcal_{3,0}(\FF_p)$ has the fiber jumping property for~$\Gp$. It follows that $\Cage_\Gp\bigl(\Mcal_{3,0}(\FF_p)\bigr)$ is connected, cf.\ Proposition~\ref{proposition:fibjumimpliesconnected}. They then use a similar argument for the endgame, where a fiber is deemed large if it has~$p^{1/2+\epsilon}$ points. Next they consider the middle game, which consists of points whose (small) fibral orbit has at least~$p^\epsilon$ points. This comes down to showing that certain equations have few solutions whose coordinates are elements of~$\FF_p^*$ of small order. They provide three proofs of the required statement, one via Stepanov's auxiliary polynomial proof of Weil's conjecture for curves over~$\FF_p$, one using directly a sharp estimate due to Corvaja and Zannier~\cite{MR3082249} for the gcd of polynomials over finite fields, and one using a projective Szemeredi-Trotter theorem due to Bourgain~\cite{MR2989378}. Indeed, they can handle the middle game for even smaller fibral components provided that~$p^2-1$ does not have too many prime divisors. Finally, for the opening,  they first observe that finite orbits in~$\Mcal_{a,k}(\Qbar)$ will cause tiny orbits in~$\Mcal_{a,k}(\FF_p)$ for infinitely many~$p$. However, in their case~$\Mcal_{3,0}(\Qbar)$ contains no finite orbits other than~$\bigl\{(0,0,0)\bigr\}$, so this is not a problem. They next show that every point $P\in\Mcal^*(\FF_p)$ lies in a fibral component containing at least~$(\log_{20}p)^{1/3}$ points. This and some further calculations suffice to prove that~$\Mcal^*(\FF_p)$ is $\Gp$-connected unless~$p^2-1$ is very smooth, i.e., is a product of a large number of small primes. (Conjecturally, there are only finitely many such primes.)
 \end{remark}
    
\begin{remark}[Fiber Jumping and Cage Connectivity for TIK3-Sur\-faces]
As explained in Remark~\ref{remark:BGSmethod}, Bourgain, Gamburd, and Sarnak~\cite{arxiv1607.01530} prove that the Markoff equation~$\Mcal_{3,0}(\FF_p)\setminus\bigl\{(0,0,0\bigr\}$ is $\Gp$-connected by first verifying the fiber-jumping property, which sets the board by implying that the cage is $\Gp$-connected. Later we will give an example showing that the analogous statement need not be true for TIK3 surfaces. More precisely, in Example~\ref{example:W1overF53cagedisconnected} we describe a TIK3-surface~$\Wcal$ such that~$\Wcal(\FF_{53})$ has one large $\Gp$-connected component~$\Wcal^*(\FF_{53})$ containing~$3456$ points, but~$\Wcal^*(\FF_{53})$ does not have the $\Gp$-fiber-jumping property. More precisely, the~$\Gp$-connected fibers in~$\Wcal(\FF_{53})$ form two connected components, so any proof that~$\Wcal^*(\FF_{53})$ is $\Gp$-connected must find a way to connect points in~$\ConnFib\bigl(\Wcal(\FF_{53})\bigr)$ that uses points that do not lie on a $\Gp$-connected fiber, i.e., using points that are not in the cage.
Of course, the prime~$p=53$ is not huge, so our example may simply be a small number phenomenon. However, other examples suggest that the number of fibral components in a TIK3 cage tends to be smaller than the number of fibral components in a Markoff surface cage. So a proof that TIK3 surfaces over finite fields have large $\Gp$-connected components may need to  find a way to expand the cage in order to fit it into a $\Gp$-connected set that can be used for the ``setting the board'' step.

In addition, the issue concerning smoothness of fibral group orders that arises in the method of BGS will be  exacerbated for TIK3 surfaces. The analogous rotations (translations) on a TIK3 surface come from the actions of elliptic curves on homogeneous spaces. These actions are translations by a point whose order can range from $p+1-2\sqrt{p}$ to $p+1+2\sqrt{p}$. So now we are not concerned with smoothness of only $p\pm1$, but instead with the smoothness of all numbers within this range. Ideally, we would like to restrict to values of~$p$ for which this range of numbers contains no smooth numbers, but there are unlikely to be infinitely many such~$p$.
\end{remark}

%%%%%%%%%%%%%%%%%%%%%%%%%%%%%%%%%%%%%%%%%%%%%%%%%%%%%%%%%%%%%%%%%%%%%%
\section{A brief survey of related work on TI K3 surfaces}
\label{section:TIK3resultssurvey}
%%%%%%%%%%%%%%%%%%%%%%%%%%%%%%%%%%%%%%%%%%%%%%%%%%%%%%%%%%%%%%%%%%%%%%

We briefly describe some earlier work on the geometry and arithmetic of~TIK3 surfaces.  Wang~\cite{MR1352278} explicitly constructed
canonical heights on~TIK3 surfaces defined over number fields associated to the infinite order automorphisms~$\s_i\circ\s_j$,
similar to those constructed in~\cite{MR1115546} for~K3
surfaces having two
involutions. Baragar~\cite{MR1459107,MR2011389,MR2059748} further
studied these height functions and asked, in particular, whether they
fit together to form a vector canonical height.
Kawaguchi~\cite{MR3073239} answered this in the negative for
certain~K3 surfaces, and Cantat and Dujardin~\cite{arxiv2012.01762}
completely characterized the surfaces on which vector
canonical heights exist.

We next state a recent result regarding finite orbits on~TIK3
surfaces in charateristic~$0$.

\begin{theorem}[{\cite[Cantat--Dujardin]{arxiv2012.01762}}]
Let~$\Wcal/\CC$ be a TIK3 surface, and
let~$\langle\s_1,\s_2,\s_3\rangle\subseteq\Aut(\Wcal)$ be the subgroup
of~$\Wcal$ generated by the three involutions~$\s_1,\s_2,\s_3$. Then
\[
\bigl\{P\in \Wcal(\CC) : \text{the $\langle\s_1,\s_2,\s_3\rangle$-orbit of $P$ is finite}\bigr\}
\]
is a finite set.
\end{theorem}
\begin{proof}
This is a special case of the results in~\cite{arxiv2012.01762}, since
in the language of~\cite{arxiv2012.01762}, the~TIK3-surface~$\Wcal$
and its group of automorphisms $\langle\s_1,\s_2,\s_3\rangle$ do not
form a Kummer group, and~$\Wcal$ contains no
$\langle\s_1,\s_2,\s_3\rangle$-invariant curves.
\end{proof}

Finally, we mention Cantat's fundamental paper~\cite{MR1864630}, although it is
not specifically about~TIK3 surfaces. 
Let $\f:\Xcal\to\Xcal$ be an automorphism of positive entropy of a~K3 surface
defined over~$\CC$, e.g., $\s_i\circ\s_j$ for a~TIK3 surface. Then Cantat proves that there
exists a unique invariant probability measure~$\mu$ with maximal
entropy, that~$(\f,\mu)$ is measurably conjugate to a Bernoulli shift,
and that~$\mu$ gives the asymptotic distribution of periodic points.

%%%%%%%%%%%%%%%%%%%%%%%%%%%%%%%%%%%%%%%%%%%%%%%%%%%%%%%%%%%%%%%%%%%%%%
\section{The incidence graph of the fibers of a TIK3 surface}
\label{section:incidencegraph}
%%%%%%%%%%%%%%%%%%%%%%%%%%%%%%%%%%%%%%%%%%%%%%%%%%%%%%%%%%%%%%%%%%%%%%

\begin{definition}
A TIK3 surface has three fibral directions associated to the three
projections onto~$\PP^1$.  For expositional convenience, we will say
that fibers corresponding to different projections are
(\emph{pairwise}) \emph{orthogonal} to one another, while fibers
corresponding to the same projection are \emph{parallel}.  So for
example, the fibers~$\Wcalf1_{{x}_0}$ and~$\Wcalf2_{{y}_0}$ are
orthogonal, while the fibers~$\Wcalf1_{{x}_0}$ and~$\Wcalf1_{{x}_1}$
are parallel.
\end{definition}

\begin{remark}
Distinct parallel fibers clearly do not intersect, while orthogonal
fibers in~$\Wcal(\FF_q)$ may intersect in~$0$,~$1$, or~$2$ points. For
example, if~${x}_0,{y}_0\in\PP^1(\FF_q)$, then
\[
\Bigl( \Wcalf1_{{x}_0}(\FF_q) \cap \Wcalf2_{{y}_0}(\FF_q) \Bigr)
=
\bigl\{ ({x}_0,{y}_0,{z}) : F({x}_0,{y}_0,{z}) = 0 \bigr\}.
\]
Thus the intersection is non-empty if and only if a certain quadratic
form has a solution in~$\PP^1(\FF_q)$.
\end{remark}

Our goal in this section is to give an easily verifiable condition which
ensures that, given two orthogonal fibers~$\Fcal_1$ and~$\Fcal_2$
in~$\Wcal(\FF_q)$, there is a third fiber~$\Fcal_3\subset\Wcal(\FF_q)$
satisfying
\[
\Fcal_1\cap\Fcal_3\ne\emptyset
\quad\text{and}\quad
\Fcal_2\cap\Fcal_3\ne\emptyset.
\]
In more evocative terms, although the
union~\text{$\Fcal_1\cup\Fcal_2$} of two orthogonal fibers may be
``disconnected,'' there is a third fiber so
that~\text{$\Fcal_1\cup\Fcal_2\cup\Fcal_3$} is a ``connected'' set of
orthogonal fibers.  See Figure~\ref{figure:interconnectingfibers}.

\begin{figure}
  \begin{picture}(300,150)(-10,-20)
    \put(0,0){\circle*5}
    \put(5,-5){\makebox(0,0)[l]{$({x}_0,{y}_1,{z}_1)$}}
    \put(200,100){\circle*5}
    \put(205,95){\makebox(0,0)[lt]{$({x}_1,{y}_0,{z}_1)$}}      
    \put(0,-20){\line(0,1){140}}
    \put(-20,-10){\line(2,1){240}}
    \put(100,100){\line(1,0){200}}
    \put(-5,50){\makebox(0,0)[r]{$\Wcalf1_{{x}_0}$}}
    \put(250,105){\makebox(0,0)[lb]{$\Wcalf2_{{y}_0}$}}
    \put(100,55){\makebox(0,0)[b]{$\Wcalf3_{{z}_1}$}}
    \put(165,30){\makebox(0,0)[l]{\small
        \framebox{
          \begin{tabular}{@{}l@{}}
          Given ${x}_0$ and ${y}_0$, find ${z}_1$\\
          so that there exist ${x}_1$ and ${y}_1$\\
          satisfying $({x}_0,{y}_1,{z}_1)\in\Wcalf1_{{x}_0}$\\
          and  $({x}_1,{y}_0,{z}_1)\in\Wcalf2_{{y}_0}$.\\
          \end{tabular}
          }
        }}
  \end{picture}
\caption{Finding a fiber $\Wcalf3_{{z}_1}$ that intersects two given fibers $\Wcalf1_{{x}_0}$ and $\Wcalf2_{{y}_0}$}
\label{figure:interconnectingfibers}
\end{figure}

\begin{definition}
For ${x}_0,{y}_0,{z}_0\in\PP^1$, we define linking sets that
describe how to link two given fibers via a third fiber.
\begin{align*}
  \Lcalf1_{{y}_0,{z}_0}
  &= \bigl\{ {x}\in\PP^1 :
  \text{$\Wcalf2_{{y}_0}\cap\Wcalf1_{{x}}\ne\emptyset$ and 
  $\Wcalf3_{{z}_0}\cap\Wcalf1_{{x}}\ne\emptyset$}
  \bigr\} ,\\
  \Lcalf2_{{x}_0,{z}_0}
  &= \bigl\{ {y}\in\PP^1 :
  \text{$\Wcalf1_{{x}_0}\cap\Wcalf2_{{y}}\ne\emptyset$ and 
  $\Wcalf3_{{z}_0}\cap\Wcalf2_{{y}}\ne\emptyset$}
  \bigr\}, \\
  \Lcalf3_{{x}_0,{y}_0}
  &= \bigl\{ {z}\in\PP^1 :
  \text{$\Wcalf1_{{x}_0}\cap\Wcalf3_{{z}}\ne\emptyset$ and 
  $\Wcalf2_{{y}_0}\cap\Wcalf3_{{z}}\ne\emptyset$}
  \bigr\}.
\end{align*}
Thus for example, the points in $\Lcalf3_{{x}_0,{y}_0}$ tell us which~$z$ fibers can be used to
link the~$x=x_0$ fiber with the~$y=y_0$ fiber. 
\end{definition}

\begin{definition}
\label{definition:Ccalcurves}
For ${x}_0,{y}_0,{z}_0\in\PP^1$, we define the following curves that
are useful in creating fibral links:
\begin{align*}
  \Ccalf1_{{y}_0,{z}_0}
  &= \bigl\{ ({x},{y},{z})\in(\PP^1)^3 :
  F({x},{y}_0,{z}) = F({x},{y},{z}_0) = 0 \bigr\}, \\
  \Ccalf2_{{x}_0,{z}_0}
  &= \bigl\{ ({x},{y},{z})\in(\PP^1)^3 :
  F({x}_0,{y},{z}) = F({x},{y},{z}_0) = 0 \bigr\}, \\
  \Ccalf3_{{x}_0,{y}_0}
  &= \bigl\{ ({x},{y},{z})\in(\PP^1)^3 :
  F({x}_0,{y},{z}) = F({x},{y}_0,{z}) = 0 \bigr\}.
\end{align*}
We note that the curve $\Ccalf1_{{y}_0,{z}_0}$ is the intersection
in~$(\PP^1)^3$ of a hypersurface of type $(2,0,2)$ and a hypersurface
of type~$(2,2,0)$, and similarly for~$\Ccalf2_{{x}_0,{z}_0}$
and~$\Ccalf3_{{x}_0,{y}_0}$. (See Lemma~\ref{lemma:genusCijk} for an estimate of the genera of these curves.)
\end{definition}

\begin{theorem}[K3 Analogue of {\cite[Proposition~6]{MR3456887}}]
\label{theorem:fiberconnectviathirdfiber}
\hfill\leavevmode\break
Let~$K$ be a field, and let~${x}_0,{y}_0,{z}_0\in\PP^1(K)$. 
\begin{parts}
\Part{(a)}
There are \emph{surjective} maps
\begin{align*}
  \Ccalf1_{{y}_0,{z}_0}(K) \xrightarrow{({x},{y},{z})\mapsto {x}} \Lcalf1_{{y}_0,{z}_0}(K),\\
  \Ccalf2_{{x}_0,{z}_0}(K) \xrightarrow{({x},{y},{z})\mapsto {y}} \Lcalf2_{{x}_0,{z}_0}(K),\\
  \Ccalf3_{{x}_0,{y}_0}(K) \xrightarrow{({x},{y},{z})\mapsto {z}} \Lcalf3_{{x}_0,{y}_0}(K).
\end{align*}
\Part{(b)}
Assume that~$q\ge100$. Then
\[
\Lcalf1_{{y}_0,{z}_0}(\FF_q)\ne\emptyset,\qquad
\Lcalf2_{{x}_0,{z}_0}(\FF_q)\ne\emptyset,\qquad
\Lcalf3_{{x}_0,{y}_0}(\FF_q)\ne\emptyset.
\]
\end{parts}
\end{theorem}
\begin{proof}
(a)\enspace By symmetry, it suffices to prove that the first map is
well-defined and surjective.
Let~$({x},{y},{z})\in\Ccalf1_{{y}_0,{z}_0}(K)$. By definition
of~$\Ccalf1_{{y}_0,{z}_0}$, this means that
\[
F({x},{y}_0,{z}) = F({x},{y},{z}_0) = 0,
\quad\text{and thus}\quad
({x},{y}_0,{z}),\,({x},{y},{z}_0)\in\Wcal(K).
\]
Hence
\[
({x},{y}_0,{z})\in\Wcalf2_{{y}_0}(K)\cap\Wcalf1_{{x}}(K)
\quad\text{and}\quad
({x},{y},{z}_0)\in\Wcalf3_{{z}_0}(K)\cap\Wcalf1_{{x}}(K),
\]
which by definition of~$\Lcalf1_{{y}_0,{z}_0}$ shows
that~${x}\in\Lcalf1_{{y}_0,{z}_0}(K)$. This completes the proof that
the projection map
\begin{equation}
  \label{eqn:pi1Cf1y0z0toLc1y0z0K}
  \pi_1 : \Ccalf1_{{y}_0,{z}_0}(K) \longrightarrow \Lcalf1_{{y}_0,{z}_0}(K)
\end{equation}
is well-defined.
\par
To prove surjectivity, we start with some~${x}\in\Lcalf1_{{y}_0,{z}_0}(K)$. By definition
of~$\Lcalf1_{{y}_0,{z}_0}$, this means that we can find points
\[
({x},{y}_0,{z}_1)\in\Wcalf2_{{y}_0}(K)\cap\Wcalf1_{{x}}(K)
\quad\text{and}\quad
({x},{y}_1,{z}_0)\in\Wcalf3_{{z}_0}(K)\cap\Wcalf1_{{x}}(K).
\]
Then the definition of~$\Ccalf1_{{y}_0,{z}_0}$ tells us that
\[
({x},{y}_1,{z}_1) \in \Ccalf1_{{y}_0,{z}_0}(K).
\]
We have thus constructed a point in~$\Ccalf1_{{y}_0,{z}_0}(K)$ whose
image in~$\Lcalf1_{{y}_0,{z}_0}(K)$ is~$x$, which completes the
proof that the projection map~\eqref{eqn:pi1Cf1y0z0toLc1y0z0K} is
surjective.
\par\noindent(b)\enspace
We use~(a) with~$K=\FF_q$.
Again by symmetry, it suffices to prove the first assertion. And
from the surjectivity of the map in~(a), it suffices to prove
that~$\Ccalf1_{{y}_0,{z}_0}(\FF_q)$ is not empty. 
\par
We let~$\widetilde{\Ccalf1_{{y}_0,{z}_0}}$ be a non-singular model for~$\Ccalf1_{{y}_0,{z}_0}$ (or more generally for any one of its irreducible components if it happens to be reducible), so in particular we have a surjection
\[
\widetilde{\Ccalf1_{{y}_0,{z}_0}}(\FF_q)\longrightarrow\Ccalf1_{{y}_0,{z}_0}(\FF_q).
\]
Then the Weil estimate gives the inequality
\begin{equation}
\label{eqn:widetilC1yzFq}
    \#\widetilde{\Ccalf1_{{y}_0,{z}_0}}(\FF_q) \ge q+1-2\cdot\bigl(\genus\widetilde{\Ccalf1_{{y}_0,{z}_0}}\bigr)\cdot\sqrt{q}.
\end{equation}
In particular, we see that
\begin{equation}
\label{eqn:q1gt2gsqrtq}
q+1>2\cdot\bigl(\genus\widetilde{\Ccalf1_{{y}_0,{z}_0})}\bigr)\cdot\sqrt{q}
\quad\Longrightarrow\quad
\widetilde{\Ccalf1_{{y}_0,{z}_0}}(\FF_q)\ne\emptyset.
\end{equation}
Lemma~\ref{lemma:genusCijk}, whose proof we defer for the moment, says
that the genus of~$\widetilde{\Ccalf1_{{y}_0,{z}_0}}$ is at most~$5$. Hence~\eqref{eqn:widetilC1yzFq} and~\eqref{eqn:q1gt2gsqrtq} imply that~$\Ccalf1_{{y}_0,{z}_0}(\FF_q)$ is non-empty provided \text{$q+1>10\sqrt{q}$}, which is true for all~$q>100$.
\end{proof}

We now prove the genus estimate used in the proof of
Theorem~\ref{theorem:fiberconnectviathirdfiber}.

\begin{lemma}
\label{lemma:genusCijk}
Let~$\Wcal$ be a non-degenerate TIK3 surface.
Then the irreducible components of each of the curves in
Definition~\textup{\ref{definition:Ccalcurves}}
has geometric genus at most~$5$.
\end{lemma}
\begin{proof}
We work over an algebraically closed field.
By symmetry, it suffices to fix~${y}_0,{z}_0\in\PP^1$ and to consider
the curve~$\Ccalf1_{{y}_0,{z}_0}$.
We let~$F$ be the~$(2,2,2)$-form that defines the non-degenerate TIK3
surface~$\Wcal$.  We define a projection map
\[
\pi : \Ccalf1_{{y}_0,{z}_0} \longrightarrow\PP^1,\quad
\pi(x,y,z)=x.
\]
Keeping in mind that~${y}_0$ and~${z}_0$ are fixed, for~${x}_1\in\PP^1$ we have
\[
\pi^{-1}({x}_1) = \bigl\{({y},{z})\in(\PP^1)^2 : F({x}_1,{y}_0,{z})=F({x}_1,{y},{z}_0)=0 \bigr\}.
\]
The equations for~${y}$ and~${z}$ are independent, so we find that
\[
\#\pi^{-1}({x}_1) = \#\bigl\{{z}\in\PP^1 : F({x}_1,{y}_0,{z})=0 \bigr\}
\cdot \#\bigl\{{y}\in\PP^1 : F({x}_1,{y},{z}_0)=0 \bigr\}.
\]
The non-degeneracy assumption tells us that $F({x}_1,{y}_0,{z})$
and~$F({x}_1,{y},{z}_0)$ are not identically~$0$, so they are
non-trivial quadratic forms in, respectively,~${z}$ and~${y}$. As
such, they have either~$1$ or~$2$ roots, and we can determine which is the case
by computing an appropriate discriminant:
\begin{align*}
\#\bigl\{{z}\in\PP^1 : F({x}_1,{y}_0,{z})=0 \bigr\}
&=\begin{cases}
1 &\text{if $\Disc_{z} F({x}_1,{y}_0,{z}) = 0$,} \\
2 &\text{if $\Disc_{z} F({x}_1,{y}_0,{z}) \ne 0$.} \\
\end{cases} \\
\#\bigl\{{y}\in\PP^1 : F({x}_1,{y},{z}_0)=0 \bigr\}
&=\begin{cases}
1 &\text{if $\Disc_{y} F({x}_1,{y},{z}_0) = 0$,} \\
2 &\text{if $\Disc_{y} F({x}_1,{y},{z}_0) \ne 0$.} \\
\end{cases} 
\end{align*}
Combining these estimates yields the following formulas
\[
\begin{array}{|c|c|c|}\hline
  \#\pi^{-1}({x}_1)  & \Disc_{y} F({x}_1,{y},{z}_0) & \Disc_{z} F({x}_1,{y}_0,{z}) \\ \hline\hline
  4 & {}\ne0 & {}\ne0 \\ \hline
  2 & {}=0 & {}\ne0 \\ \hline
  2 & {}\ne0 & {}=0 \\ \hline
  1 & {}=0 & {}=0 \\ \hline  
\end{array}
\]
\par
We next observe that $\Disc_{y}{F}({x},{y},{z}_0)$ is a degree~$4$
form in~${x}$, and thus has at most~$4$ roots in~$\PP^1$ when
considered as a polynomial in~${x}$; and similarly for
$\Disc_{z}{F}({x},{y}_0,{z})$. So there are at most~$8$
points~${x}_1\in\PP^1$ with~$\#\pi^{-1}({x}_1)=2$. Further, each time we
get an~${x}_1$ with~$\#\pi^{-1}({x}_1)=1$, we see that~$2$ of those~$8$
potential values of~${x}_1$ coalesce into~$1$ value. So if we let
\begin{equation}
\label{eqn:ABnumx1piinv}
\begin{aligned}
A &= \#\bigl\{{x}_1\in\PP^1 : \pi^{-1}({x}_1)=2 \bigr\},\\
B &= \#\bigl\{{x}_1\in\PP^1 : \pi^{-1}({x}_1)=1 \bigr\},\\
\end{aligned}
\end{equation}
then we see that
\begin{equation}
  \label{eqn:BAtablegenus}
    \begin{array}{|c||c|c|c|c|c|} \hline
      B & 0 & 1 & 2 & 3 & 4 \\ \hline
      A & {}\le8 & {}\le6 & {}\le4 & {}\le2 & {}=0 \\ \hline
    \end{array}
\end{equation}
\par
We assume for the moment that~$\Ccalf1_{{y}_0,{z}_0}$ is irreducible,\footnote{See Remark~\ref{remark:reducibleBeq0} for examples where $\Ccalf1_{y_0,z_0}$ is reducible.} and we let
\[
\lambda : \widetilde{\Ccalf1_{{y}_0,{z}_0}} \longrightarrow \Ccalf1_{{y}_0,{z}_0}
\]
be a desingularization of~$\Ccalf1_{{y}_0,{z}_0}$, so the geometric genus of~$\Ccalf1_{{y}_0,{z}_0}$ is simply the genus of~$\widetilde{\Ccalf1_{{y}_0,{z}_0}}$. We use the Riemann--Hurwitz genus formula
\[
2\genus\bigl(\widetilde{\Ccalf1_{{y}_0,{z}_0}}\bigr)-2
= -2\deg(\pi\circ\lambda) + \sum_{{x}_1\in\PP^1} \Bigl( \deg(\pi\circ\lambda) - \#(\pi\circ\lambda)^{-1}({x}_1) \Bigr).
\]
Substituting
\[
\deg\pi\circ\lambda = \deg(\pi)\cdot\deg(\lambda) = 4\cdot 1 = 4,
\]
we get
\begin{align*}
  \genus\bigl(\widetilde{\Ccalf1_{{y}_0,{z}_0}}\bigr)
  &= -3 + \frac12 \sum_{\substack{{x}_1\in\PP^1\\\#(\pi\circ\lambda)^{-1}({x}_1)<4\\}}  \Bigl( 4 - \#(\pi\circ\lambda)^{-1}({x}_1) \Bigr) \\
  &\le -3 + \frac12 \sum_{\substack{{x}_1\in\PP^1\\\#\pi^{-1}({x}_1)<4\\}}  \Bigl( 4 - \#\pi^{-1}({x}_1) \Bigr) \\
  &= -3
  + \#\bigl\{{x}_1\in\PP^1 : \#\pi^{-1}({x}_1)=2 \bigr\} \\
  &\phantom{=-3}+ \frac32 \#\bigl\{{x}_1\in\PP^1 : \#\pi^{-1}({x}_1)=1 \bigr\}  \\
  &= -3 + A + \frac32B \quad
  \text{using the notation in \eqref{eqn:ABnumx1piinv},} \\
  &\le 5
  \quad\text{from \eqref{eqn:BAtablegenus}, since the max is at $(A,B)=(8,0)$.}
\end{align*}
\par
Finally, we note that if~$\Ccalf1_{{y}_0,{z}_0}$ is reducible, then the above argument works mutatis mutandis if we replace~$\Ccalf1_{{y}_0,{z}_0}$ with any of its irreducible components and note that now the map~$\pi$ has degree~$1$ or~$2$. This completes the proof of Lemma~\ref{lemma:genusCijk}.
\end{proof}

\begin{remark}
\label{remark:reducibleBeq0}
Let~$\Wcal$ be a TIK3 surface whose equation~$F$ is symmetric in~$y$ and~$z$, i.e., $F(x,y,z)=F(x,z,y)$. Then for any~$\xi\in{K}$ there is a factorization
\[
F(x,\xi,z) - F(x,y,\xi) = F(x,z,\xi) - F(x,y,\xi) = (z-y)L(x,y,z),
\]
wehre~$L(x,y,z)$ has degree~$1$ in~$y$ and~$z$. It follows the curve~$\Ccalf1_{\xi,\xi}$ described in Definition~\ref{definition:Ccalcurves} is reducible, and indeed it is the union of two genus~$1$ curves, each of which is isomorphic to the fibral curve 
\[
\Wcalf3_{\xi}\cong\bigl\{ (x,y)\in\AA^2 : F(x,y,\xi) = 0 \bigr\}
\]
\end{remark}

%%%%%%%%%%%%%%%%%%%%%%%%%%%%%%%%%%%%%%%%%%%%%%%%%%%%%%%%%%%%%%%%%%%%%%
\section{Tri-Involutive Markoff-Type K3 (MK3) Surfaces}
\label{section:mk3surfaces}
%%%%%%%%%%%%%%%%%%%%%%%%%%%%%%%%%%%%%%%%%%%%%%%%%%%%%%%%%%%%%%%%%%%%%%

The Markoff equation~\eqref{eqn:markoffeqnintro} and many of its variants
admit not only the involutions coming from the
projections~$\Mcal\to\AA^2$, they also admit sign-change involutions and coordinate
permutations coming from the symmetry of the Markoff equation.
We give a name to the TIK3 surfaces that have these extra automorphisms.

\begin{definition}
\label{definition:Gcirc}
We let~$\gS_3$, the symmetric group on~$3$ letters,
act on~$(\PP^1)^3$ by permuting the coordinates,  
and we let the group
\begin{equation}
  \label{eqn:defmu231}
  (\bfmu_2^3)_1 := \bigl\{ (\a,\b,\g) : \a,\b,\g\in\bfmu_2~\text{and}~\a\b\g=1 \bigr\}
\end{equation}
act on~$(\PP^1)^3$ via sign changes,
\begin{equation}
  \label{eqn:actionmu231}
  \e_{\a,\b,\g}(x,y,z) = (\a x,\b y,\g z).
\end{equation}
In this way we obtain an embedding\footnote{We remark that   $(\bfmu_2^3)_1\rtimes\gS_3$ is isomorphic to~$\gS_4$, but for our applications the group~$\Gp^\circ$ appears more naturally as the semi-direct product.}
\[
\Gp^\circ :=  (\bfmu_2^3)_1\rtimes\gS_3 \longhookrightarrow \Aut(\PP^1\times\PP^1\times\PP^1).
\]
\end{definition}

\begin{definition} 
\label{definition:MK3surface}
A \emph{Markoff-type K3 \textup{(MK3)} surface}~$\Wcal$ is a~TIK3 surface
whose~$(2,2,2)$-form~\eqref{eqn:222form} is invariant under the action
of~$\Gp^\circ$, i.e., the~$(2,2,2)$-form~$F$ describing~$\Wcal$ satisfies
\begin{align*}
  F(x,y,z)&=F(-x,-y,z)=F(-x,y,-z)=F(x,-y,-z), \\
  F(x,y,z)&=F(z,x,y)=F(y,z,x)=F(x,z,y)=F(y,x,z)=F(z,y,x).
\end{align*}
\end{definition}

\begin{definition}
\label{definition:Gsigma}
Let~$\Wcal$ be an~MK3 surface. We let
\begin{align*}
\Gp^\s &= \langle\s_1,\s_2,\s_3\rangle \subset \Aut(\Wcal),\\
\Gp &= \langle\text{group generated by $\Gp^\s$ and $\Gp^\circ$}\rangle \subset \Aut(\Wcal).
\end{align*}
\end{definition}

We suspect that the full automorphism group of a generic MK3-surface is~$\Gp$; but as we shall see in Remark~\ref{remark:extradeltainvonWk}, some MK3-surfaces admit additional automorphisms.
We start by describing some elementary properties of the group~$\Gp$.

\begin{proposition}
\label{proposition:structureofG}
Let~$\Wcal$ be an MK3-surface, and let~$\Gp^\circ$,~$\Gp^\s$, and~$\Gp$ be the subgroups of~$\Aut(\Wcal)$ described in Definitions~$\ref{definition:Gcirc}$ and~$\ref{definition:Gsigma}$.
\begin{parts}
\Part{(a)}
$\Gp^\s$ is a normal subgroup of $\Gp$.
\Part{(b)}
$\Gp=\Gp^\circ\Gp^\s$.
\end{parts}
\end{proposition}
\begin{proof}
(a)\enspace
Since~$\Gp$ is defined to be the group generated by~$\Gp^\circ$ and~$\Gp^\s$, it suffices to show that~$\Gp^\circ$ is contained in the normalizer of~$\Gp^\s$. We let~$\{i,j,k\}=\{1,2,3\}$, and for the purposes of this proof, we define transpositions and sign changes
\begin{align*}
\t_{ij} &= \text{swap the $i$ and $j$ coordinates}, \\
\e_{ij} &= \text{multiply the $i$ and $j$ coordinates by $-1$}.    
\end{align*}
Since~$\gS_3$ is generated by transpositions and~$(\bfmu_2^3)_1$ is generated by the sign changes, it suffices to check that~$\Gp^\s$ is normalized by the~$\t_{ij}$ and the~$\e_{ij}$. This can be checked by an explicit computation, or alternatively we can use the defining property~$\pi_{ij}\circ\s_k=\pi_{ij}$ of~$\s_k$, where~$\pi_{ij}$ is the projection map; see Definition~\ref{definition:sigmakswapsheets}. Thus momentarily letting~$\t:(\PP^1)^2\to(\PP^1)^2$ be the map that swaps the coordinates and~$\e_i:(\PP^1)^2\to(\PP^1)^2$ be the map that changes the sign of the $i$th coordinate, we compute
\begin{align*}
    \pi_{ij}\circ(\t_{ij}^{-1}\circ\s_k\circ\t_{ij})
    &= \t \circ \pi_{ij} \circ \s_k \circ \t_{ij} 
    = \t \circ \pi_{ij} \circ \t_{ij} 
    = \pi_{ij}, \\
    \pi_{jk}\circ(\t_{ik}^{-1} \circ\s_k\circ\t_{ik})
    &= \t \circ \pi_{ij} \circ \s_k \circ \t_{ik} 
    = \t \circ \pi_{ij} \circ \t_{ik} 
    = \pi_{jk} \\
    \pi_{ij}\circ(\e_{ij}^{-1}\circ\s_k\circ\e_{ij})
    &= \e_{ij} \circ \pi_{ij}\circ\s_k\circ\e_{ij} = \e_{ij}\circ\pi_{ij}\circ\e_{ij} = \e_{ij}^2\circ\pi_{ij}=\pi_{ij}, \\
    \pi_{ij}\circ(\e_{ik}^{-1}\circ\s_k\circ\e_{ik})
    &= \e_i \circ \pi_{ij} \circ \s_k \circ \e_{ik} = \e_i\circ\pi_{ij}\circ\e_{ik} = \e_i^2\circ\pi_{ij}=\pi_{ij}.
\end{align*}
It follows from the definitions of the~$\s_i$ that
\begin{align*}
    \t_{ij}^{-1}\circ\s_k\circ\t_{ij} &= \s_k,
    &\e_{ij}^{-1}\circ\s_k\circ\e_{ij} &= \s_k, \\
    \t_{ik}^{-1} \circ\s_k\circ\t_{ik} &= \s_i, 
    &\e_{ik}^{-1}\circ\s_k\circ\e_{ik} &= \s_k.    
\end{align*}
Hence~$\Gp^\circ$ normalizes~$\Gp^\s$, and indeed,~$(\bfmu_2^3)_1$ is in the centralizer of~$\Gp^\s$.
\par\noindent(b)\enspace
By definition the group~$\Gp$ is generated by~$\Gp^\circ$ and~$\Gp^\s$, and from~(a), we know that~$\Gp^\s$ is a normal subgroup of~$\Gp$. It follows that every element of~$\Gp$ can be written as~$\g\s$ with~$\g\in\Gp^\circ$ and~$\s\in\Gp^\s$. Hence~$\Gp=\Gp^\circ\Gp^\s$.
\end{proof}

\begin{proposition}
\label{proposition:nondegK3Markofftype}
Let~$\Wcal/K$ be a (possibly degenerate)~MK3-surface.%
\begin{parts}
\Part{(a)}
There exist~$a,b,c,d,e\in{K}$ so that the~$(2,2,2)$-form~$F$ that defines~$\Wcal$ has the form
\begin{align}
  \label{equation:Wabcdeeqn}
  F_{a,b,c,d,e}(x,y,z) &= 
  a  x^2y^2z^2
  + b (x^2y^2+x^2z^2+y^2z^2) \notag\\
  &\qquad{}+ c  xyz
  +  d (x^2+y^2+z^2)
  + e  = 0.
\end{align}
\Part{(b)}
Let~$F$ be as in~\textup{(a)}.  Then~$\Wcal$
is a non-degenerate, i.e.,
the projections~$\pi_{ij}:\Wcal\to(\PP^1)^2$ are quasi-finite,
if and only if
\[
  be\ne{d^2}\quad\text{and}\quad ad\ne{b^2}.
\]
\end{parts}
\end{proposition}

\begin{remark}
We can recover the classical (translated) Markoff equation for the surface~$\Mcal_{a,k}$ in Definition~\ref{eqn:markoffeqnintro} as a
special case of an~$F_{a,b,c,d,e}$. Thus~$\Mcal_{a,k}$ is given by the affine equation
\[
F_{0,0,-a,1,-k}(x,y,z) = x^2+y^2+z^2-axyz-k = 0.
\]
We note, however, that the Markoff equation is degenerate,
despite the involutions being well-defined on the affine Markoff
surface~$\Mcal_{a,k}$.  This occurs because the involutions are not
well-defined at some of the points at infinity in the closure
of~$\Mcal_{a,k}$ in~$(\PP^1)^3$.
\end{remark}

\begin{proof}[Proof of $\ref{proposition:nondegK3Markofftype}$]
(a)\enspace
The space of~$\gS_3$-invariant quadratic polynomials
in $\ZZ[x,y,z]$ is spanned by the following~$10$ polynomals:
\begin{align*}
(1) &\quad x^2 y^2 z^2 &
(2) &\quad x y z^2 + x y^2 z + x^2 y z \hidewidth\\
(3) &\quad x y z &
(4) &\quad x^2 y^2 z + x^2 y z^2 + x y^2 z^2 \hidewidth\\
(5) &\quad x^2 + y^2 + z^2 &
(6) &\quad x^2 y^2 + x^2 z^2 + y^2 z^2 \hidewidth\\
(7) &\quad x^2 y + x^2 z + x y^2 + x z^2 + y z^2 + y^2 z \hidewidth\\
(8) &\quad x y + x z + y z &
(9) &\quad x + y + z &
(10) &\quad 1
\end{align*}
Of these, the polynomials that are also invariant for the double-sign
changes in~$(\bfmu_2^3)_1$  are~(1),~(3),~(5),~(6), and~(10).
Hence all~$\bigl((\bfmu_2^3)_1\rtimes\gS_3\bigr)$-invariant~$(2,2,2)$-polynomials have
the form indicated in~(a).
\par\noindent(b)\enspace
By symmetry, it suffices to consider~$\pi_{12}$ and~$\s_{3}$.
The map~$\pi_{12}$ is quasi-finite if and only if the fibers
of the map~$\pi_{12}$ are $0$-dimensional.
Let~$\overline{F}$ be the homogenization of the polynomial in~(a).
Then~$\pi_{12}$ is quasi-finite over the point
\[
\bigl([\a,\b],[\g,\d]\bigr) \in\PP^1\times\PP^1
\]
if and only if
the polynomial $F(\a,\b;\g,\d;X_3,Y_3)$ is not
identically~$0$.  Since 
\[
\Bigl( \text{the $X_3Y_3$ term of $F(\a,\b;\g,\d;X_3,Y_3)$} \Bigr)
=
\a\b\g\d X_3Y_3,
\]
we see that~$\pi_{12}$ is quasi-finite unless~$\a\b\g\d=0$. By
the symmetry of~$F$, it suffices to consider the cases that~$\a=0$
and~$\b=0$.
\par
If~$\a=0$, then
\[
F(0,1;\g,\d;X_3,Y_3)
= (b \g^2 + d \d^2) X_3^2 + ( d  \g^2 + e  \d^2) Y_3^2.
\]
Hence~$\pi_{12}$ is quasi-finite
at~$\bigl([0,1],[\g,\d],[\a_3,\g_3]\bigr)$ unless
\[
b \g^2 + d \d^2 =   d \g^2 + e  \d^2 = 0.
\]
Since~$(\g,\d)\ne(0,0)$, this is possible if and only if~$be=d^2$.
\par
Similarly, if~$\b=0$, we look at
\[
F(1,0;\g,\d;X_3,Y_3)=
(a \g^2 + b \d^2) X_3^2 + (b \g^2 + d \d^2) Y_3^2.
\]
Thus~$\s_{3}$ is well-defined
at~$\bigl([1,0],[\g,\d],[\a_3,\g_3]\bigr)$ unless
\[
a \g^2 + b \d^2 = b \g^2 + d \d^2 = 0.
\]
Since~$(\g,\d)\ne(0,0)$, this is possible if and only if~$ad=b^2$.
This completes the proof that~$\pi_{12}$ is quasi-finite if and only
if~$be\ne{d^2}$ and~$ad\ne{b^2}$.
\end{proof}

%%%%%%%%%%%%%%%%%%%%%%%%%%%%%%%%%%%%%%%%%%%%%%%%%%%%%%%%%%%%%%%%%%%%%%
\section{Connected Fibral Components and the Cage for MK3 Surfaces}
\label{section:cageMK3}
%%%%%%%%%%%%%%%%%%%%%%%%%%%%%%%%%%%%%%%%%%%%%%%%%%%%%%%%%%%%%%%%%%%%%%

For this section we let~$\Wcal$ be an MK3-surface, as described in Definition~\ref{definition:MK3surface}, defined over a finite field~$\FF_q$. We note that the~$\gS_3$-symmetry of~$\Wcal$ implies that for any~$t\in\PP^1(\FF_q)$, the three fibers~$\Wcalf{1}_t(\FF_q)$,~$\Wcalf{2}_t(\FF_q)$ and~$\Wcalf{3}_t(\FF_q)$ have the same orbit structure, so in particular
\begin{multline*}
\text{$\Wcalf{i}_t(\FF_q)\in\ConnFib\bigl(\Wcal(\FF_q)\bigr)$ for some $i\in\{1,2,3\}$} \\
\quad\Longleftrightarrow\quad
\text{$\Wcalf{i}_t(\FF_q)\in\ConnFib\bigl(\Wcal(\FF_q)\bigr)$ for all $i\in\{1,2,3\}$.}
\end{multline*}
Thus the~$\Gp$-connected fibers in~$\Wcal(\FF_q)$ are determined by the projection to~$\PP^1(\FF_q)$ of~$\ConnFib\bigl(\Wcal(\FF_q)\bigr)$ onto any of its coordinates. We denote this set by
\[
    \pi\ConnFib\bigl(\Wcal(\FF_q)\bigr)
=\Bigl\{ t\in\PP^1(\FF_q) : 
\Wcalf{i}_t(\FF_q)\in\ConnFib\bigl(\Wcal(\FF_q)\bigr)\Bigr\}.
\]
Then we have the useful characterization (for MK3-surfaces):
\[
P \in \Cage\bigl(\Wcal(\FF_q)\bigr) 
\;\Longleftrightarrow\;
\text{some coordinate of $P$ is in $\pi\ConnFib\bigl(\Wcal(\FF_q)\bigr)$.}
\]

%%%%%%%%%%%%%%%%%%%%%%%%%%%%%%%%%%%%%%%%%%%%%%%%%%%%%%%%%%%%%%%%%%%%%%
\section{A One Parameter Family of MK3 Surfaces}
\label{section:familyofk3}
%%%%%%%%%%%%%%%%%%%%%%%%%%%%%%%%%%%%%%%%%%%%%%%%%%%%%%%%%%%%%%%%%%%%%%

In the next few sections we study an interesting $1$-parameter family of MK3-surfaces.  We
assume throughout that~$K$ is a field with~$\operatorname{char}(K)\ne2$.

\begin{definition}
\label{definition:Wkx2y2z2kxyz}
For~$k\in{K^*}$ we define~$\Wcal_k$ to be the MK3-surface
\[
\Wcal_k : x^2 + y^2 + z^2 + x^2 y^2 z^2 + k x y z = 0 .
\]
\end{definition}

\begin{remark}
In the notation of Proposition~\ref{proposition:nondegK3Markofftype},
the~$(2,2,2)$-form defining~$\Wcal_k$ has~$(a,b,c,d,e)=(1,0,k,1,0)$.
In particular, we have
\[
be=0\ne1^2=d^2 \quad\text{and}\quad ad=1\ne0^2=b^2,
\]
so Proposition~\ref{proposition:nondegK3Markofftype}(b) tells us
that~$\Wcal_k$ is non-degenerate.  
\end{remark}

\begin{remark}
\label{remark:WkisomWd3kd41}
Let~$\zeta\in{K}$ be an element satisfying~$\zeta^4=1$. Then
there is a $K$-isomorphism
\begin{equation}
  \label{eqn:WkisomWa^3k}
  \Wcal_k \longrightarrow \Wcal_{\zeta^3 k}, \quad (x,y,z) \longmapsto (\zeta x,\zeta y,\zeta z). 
\end{equation}
So we always have an identification~$\Wcal_k(K)\cong\Wcal_{-k}(K)$,
and if~$K$ contains~$i=\sqrt{-1}$, then there are further
identifications~$\Wcal_k(K)\cong\Wcal_{\pm{i}k}(K)$.
\end{remark}

\begin{remark}
The three involutions~\eqref{eqn:skWtoWinvolTIK3} on~$\Wcal_k$ are
given explicitly by
\begin{align*}
  \s_{1}(x,y,z) &= \left(-\frac{kyz}{1+y^2z^2}-x,y,z \right) ,\\
  \s_{2}(x,y,z) &= \left(x,-\frac{kxz}{1+x^2z^2}-y,z \right), \\
  \s_{3}(x,y,z) &= \left(x,y,-\frac{kxy}{1+x^2y^2}-z \right).
\end{align*}
We recall from Section~\ref{section:mk3surfaces} that~$\Gp^\circ$ is the group~$(\bfmu_2^3)_1\rtimes\gS_3$ of order~$24$ sitting in~$\Aut(\Wcal_k)$ composed of sign changes and coordinate permutations, that~$\Gp^\s$ is the normal subgroup of~$\Aut(\Wcal_k)$ generated by~$\s_1,\s_2,\s_3$, and that $\Gp=\Gp^\circ\Gp^\s$ is the subgroup of~$\Aut(\Wcal_k)$ generated by~$\Gp^\circ$ and~$\Gp^\s$.
\end{remark}

\begin{proposition}
\label{proposition:singptsonWk}  
Let~$k\in{K^*}$. The set of singular points of $\Wcal_k$ always contains the~$4$ points
\begin{equation}
\label{eqn:4singpts}
\bigl\{ (0,0,0),\,(0,\infty,\infty),\,(\infty,0,\infty),\,(\infty,\infty,0) \bigr\}.
\end{equation}
The point~$(0,0,0)$ is fixed by~$\Gp$, and the other~$3$ singular points form a~$\Gp$-orbit.\footnote{If we also allow the $\d$-inversion involutions described in  Remark~\ref{remark:extradeltainvonWk}, then the~$4$ singular points form a single orbit.} If $k\notin\{\pm4,\pm4i\}$,  then the set~\eqref{eqn:4singpts} is the full set of singular points of $\Wcal_k$. 
 \par
 For~$k=4$ the set of singular points is
 \begin{multline}
 \label{eqn:ptswithpm1}
 \operatorname{Sing}(\Wcal_4) = 
 \bigl\{ (0,0,0),\,(0,\infty,\infty),\,(\infty,0,\infty),\,(\infty,\infty,0) \\
 (1,1,-1),\, (1,-1,1),\, (-1,1,1),\, (-1,-1,-1)   \bigr\};
 \end{multline}
 and for the other $k\in\{\pm4,\pm4i\}$, the singular points can be found using the isomorphisms described in Remark~$\ref{remark:WkisomWd3kd41}$. The points in~\eqref{eqn:ptswithpm1} with non-zero coordinates form a single~$\Gp$-orbit of size~$4$.
\end{proposition}
\begin{proof}
We let
\begin{equation}
  \label{eqn:singptsWkdefF}
  F(x,y,z) = x^2 + y^2 + z^2 + x^2 y^2 z^2 + k x y z
\end{equation}
be the polynomial defining~$\Wcal_k$, and we use subscripts to denote
partial derivatives. The singular points on this affine piece
of~$\Wcal_k$ are the solutions to
\begin{equation}
  \label{eqn:FFxFyFzeq0}
F = F_x = F_y = F_z = 0.
\end{equation}
The ideal of~$\QQ[x,y,z,k]$ generated by
the four polynomials in~\eqref{eqn:FFxFyFzeq0} contains the following
polynomials:\footnote{Indeed, this is true in the
  ring~$\ZZ[2^{-1},x,y,z,k]$.}
\begin{equation}
  \label{eqn:Wksinglocus}
  \begin{array}{|c|c|c|c|} \hline
  x^2 - y^2 & x(x^4 - 1) & x(2^4x^2-k^2) & x(k^4-2^8) \\ \hline
  x^2 - z^2 & y(y^4 - 1) & y(2^4y^2-k^2) & y(k^4-2^8) \\ \hline
  y^2 - z^2 & z(z^4 - 1) & z(2^4z^2-k^2) & z(k^4-2^8) \\ \hline
  \end{array}
\end{equation}
The point~$(0,0,0)$ is always singular. Since~\eqref{eqn:Wksinglocus}
says that singular points satisfy~$x^2=y^2=z^2$, any other singular
point~$(x,y,z)$ necessarily has~$xyz\ne0$, and
then~\eqref{eqn:Wksinglocus} forces
\[
k^4=2^8,\quad 2^4x^2=2^4y^2=2^4z^2=k^2,\quad\text{and}\quad x^4=y^4=z^4=1.
\]
From $k^4=2^8$, we see that $k\in\{\pm4,\pm4i\}$; and from $x^4=y^4=z^4=1$, we see that $x,y,z\in\{\pm1,\pm{i}\}$. For each of these~$4$ possible values of~$k$, it can be directly checked that the points satisfying $F=F_x=F_y=F_z$ are those given in the table in the statement of the proposition.
\par
It remains to check the points on the
complement in~$(\PP^1)^3$ of the affine piece.  To do
that, we use the fact that~$(0,0,0)$ is the only singular point
of the affine piece of~$\Wcal_k$ that has a coordinate mapped to $\infty$ under the~$\d_{\a,\b,\g}$ inversion maps described in
Remark~\ref{remark:extradeltainvonWk}. 
By symmetry, it suffices to check points~$P$ of the following forms, where~$y$ and~$z$ are non-zero:
\[
\begin{array}{|c|c|c|} \hline
  P & \text{Singular?}  & \text{Why?} \\ \hline\hline
  (\infty,y,z) & \text{No} & \d_{-1,-1,1}(P) = (0,y^{-1},z) \\ \hline
  (\infty,\infty,z) & \text{No} & \d_{-1,-1,1}(P) = (0,0,z) \\ \hline
  (\infty,y,0) & \text{No} & \d_{-1,-1,1}(P) = (0,y^{-1},0) \\ \hline
  (\infty,\infty,0) & \text{Yes} & \d_{-1,-1,1}(P) = (0,0,0) \\ \hline
  (\infty,0,0) & - & {}\notin\Wcal_k \\ \hline
  (\infty,\infty,\infty) & - & {}\notin\Wcal_k \\ \hline
\end{array}
\]
\end{proof}

\begin{remark}[MK3-Surfaces with Extra Involutions]
\label{remark:extradeltainvonWk}  
The family of MK3-surfaces~$\Wcal_k$ admit additional involutions in
which two of~$x,y,z$ are replaced by their multiplicative inverses.\footnote{Note that we're really working in~$\PP^1$, so we formally set~$0^{-1}=\infty$ and $\infty^{-1}=0$.} Thus
analogously to~\eqref{eqn:defmu231} and~\eqref{eqn:actionmu231}, we can define
another action of~$(\bfmu_2^3)_1$ on~$(\PP^1)^3$ via the formula
\begin{equation}
  \label{eqn:actionZ2Z231}
  \d_{\a,\b,\g}(x,y,z) = (x^\a,y^\b,z^\g),\quad\text{where $(\a,\b,\g)\in(\bfmu_2^3)_1$.}
\end{equation}
We observe that the~$\d$ and~$\e$ actions commute (since
$(-1)^{-1}=-1$), so we obtain an embedding
\[
\Gpplus^\circ:= \underbrace{\bigl( (\bfmu_2^3)_1 \times (\bfmu_2^3)_1 \bigr) \rtimes \gS_3}_{\text{\hidewidth We view this as a subgroup of $\Aut\bigl((\PP^1)^3\bigr)$.\hidewidth}}
\longhookrightarrow \Aut(\Wcal_k).
\]
Since the classical Markoff
equation~\eqref{eqn:Makx2y2z2axyzkdef} and general
MK3-surfaces~\eqref{equation:Wabcdeeqn} do not admit these extra
automorphisms, we will not include them when constructing orbits
in~$\Wcal_k$. So for example, the finite orbits
and~$\Gp^\circ$-generators in~$\Wcal_k(\CC)$ that we list in
Table~\ref{table:finiteorbschar0} are~$\Gp$-orbits, as are the finite field orbits in~$\Wcal_k(\FF_p)$ in Appendix~\ref{appendix:finitefieldtables}. There would be
some collapsing of generators and merging of orbits if we also used
the~$\d$-automorphisms. However, the existence of these extra automorphisms can aid in studying the geometry of~$\Wcal_k$, as will be illustrated in the proof of Proposition~\ref{proposition:singptsonWkfiber}.
\par
More generally,  Proposition~\ref{proposition:nondegK3Markofftype} says that 
MK3-surfaces~$\Wcal_{a,b,c,d,e}$ are described by~$(2,2,2)$-forms~$F_{a,b,c,d,e}(x,y,z)$ that depend on~$5$
homogeneous parameters~$[a,b,c,d,e]$.  Then the formula
\begin{multline*}
F_{a,b,c,d,e}(x,y,z)-F_{a,b,c,d,e}(x^{-1},y^{-1},z)x^2y^2 \\
= \Bigl( (a-d)z^2+(b-e) \Bigr) (x^2y^2-1),
\end{multline*}
combined with the~$x,y,z$ symmetry of~$F_{a,b,c,d,e}$,
imply that
\[
\d_{\a,\b,\g} \in \Aut(\Wcal_{a,b,c,d,e})
\quad\Longleftrightarrow\quad
a=d~\text{and}~b=e.
\]
Thus~$\Wcal_k=\Wcal_{1,0,k,1,0}$ corresponds to $a=d=1$ and $b=e=0$.
\end{remark}

\begin{proposition}
\label{proposition:singptsonWkfiber}  
Let~$K$ be a field with $\characteristic(K)\ne2$, let~$k\in{K^*}$, and let \text{$\xi\in\PP^1(K)$}. Then the fiber~$\Wcalf1_{k,\xi}$ is singular if and only if
\[
\xi=0  \quad\text{or}\quad
\xi=\infty  \quad\text{or}\quad
k = \pm 2(\xi\pm\xi^{-1}).
\]
The singular points on the singular fibers are as follows\textup:
\begin{align*}
\Sing\bigl( \Wcalf1_{k,0} \bigr) &= \bigl\{ (0,0,0),\, (0,\infty,\infty) \bigr), \\
\Sing\bigl( \Wcalf1_{k,\infty} \bigr) &= \bigl\{ (\infty,\infty,0),\, (\infty,0,\infty) \bigr), 
\end{align*}
and for all $\xi\notin\{0,\infty\}$ and for all $u\in\{\pm1\}$ and all $v\in\{\pm1,\pm{i}\}$,
\[
\Sing\bigl( \Wcalf1_{u(\xi+v\xi^{-1}),\xi} \bigr)
= \bigl\{ (\xi,v,-uv^3),\,(\xi,-v,uv^3) \bigr\}.
\]
By symmetry, analogous statements are true for~$\Wcalf2_{k,\xi}$ and~$\Wcalf3_{k,\xi}$.
\end{proposition}

\begin{remark}
Let~$\Wcalf{i}_{k,\xi}$ be a fiber of~$\Wcal_k$. Then each of the involutions~$\s_1,\s_2,\s_3$ and each of the automorphisms in~$\Gp^\circ$ defines an isomorphism from~$\Wcalf{i}_{k,\xi}$ to some other (or possibly the same) fiber of~$\Wcal_k$. It follows that the singular points on a fiber are mapped to singular points on a fiber. Hence the set
\[
\bigcup_{i=1}^3 \bigcup_{\xi\in\PP^1} \operatorname{Sing}(\Wcalf{i}_{k,\xi})
\]
of fibral singular points is a finite subset of~$\Wcal_k$ that is~$\Gp$-invariant, so it breaks up into a finite number of finite $\Gp$-orbits. If~$\xi\ne0,\infty$ and~$\xi^4\ne1$, then it will be a~$\Gp$-orbit of size 24; cf.\ Table~\ref{table:finiteorbschar0}.
\end{remark}

\begin{proof}[Proof of Proposition~$\ref{proposition:singptsonWkfiber}$]
As in the proof of Proposition~\ref{proposition:singptsonWk}, we
let~$F$ be the polynomial~\eqref{eqn:singptsWkdefF}
defining~$\Wcal_k$, and we use subscripts to denote partial
derivatives.  The fiber~$\Wcalf1_{k,\xi}$ is singular if and only if
the simultaneous equations
\begin{equation}
  \label{eqn:singlocusW1kx01}
  F(\xi,y,z)=F_y(\xi,y,z)=F_z(\xi,y,z)=0
\end{equation}
have a solution. We compute
\begin{multline*}
\Resultant_y\Bigl(\Resultant_z(F,F_z),\Resultant_z(F_y,F_z)\Bigr) 
= 2^{12} \cdot k^8 \cdot x^{26} 
     \cdot (2 x^2 - k x - 2)^2\\
     \cdot (2 x^2 - k x + 2)^2
     \cdot (2 x^2 + k x - 2)^2
     \cdot (2 x^2 + k x + 2)^2.
\end{multline*}
\par
We first consider the case that~$\xi=0$. Then~\eqref{eqn:singlocusW1kx01} forces $y=z=0$, so the only affine singular point is~$(0,0,0)$. Using the inversion automorphism fixing the $x$-coordinate that is described in Remark~\ref{remark:extradeltainvonWk}, there is an additional singular point~$(0,\infty,\infty)$, so we find that
\[
\Sing( \Wcalf1_{k,0} ) = \bigl\{ (0,0,0),\, (0,\infty,\infty) \bigr\}.
\]
And similarly, using the inversion automorphisms in Remark~\ref{remark:extradeltainvonWk} that replace the~$x$-coordinate with~$x^{-1}$, we see that
\[
\Sing( \Wcalf1_{k,\infty} ) = \bigl\{ (\infty,\infty,0),\, (\infty,0,\infty) \bigr).
\]
\par
We now assume that~$\xi\ne0,\infty$. Then our assumptions that~$\characteristic(K)\ne2$ and~$\Wcalf1_{k,x_0}$ is singular imply that~$\xi$ is a root of one of the polynomials \text{$2x^2\pm{kx}\pm2$}. We will consider the case that
\[
2\xi^2 + k\xi + 2 = 0,
\]
and leave the similar computation for the other three cases to the reader.
Thus we assume that
\[
k = -2(\xi+\xi^{-1}) \quad\text{and}\quad \Wcalf1_{k,\xi}~\text{is singular.}
\]
Substituting the expression for~$k$ into~\eqref{eqn:singlocusW1kx01}, we find that~$(y_0,z_0)$ is a singular point on the fiber~$\Wcalf1_{k,\xi}$ if and only if~$(y_0,z_0)$ satisfy
\begin{align*}
( y^2  z^2 - 2 y z + 1)  \xi^2 - 2 y z  +  y^2 +  z^2 &= 0,\\
( y z^2 - z) \xi^2 -  z +  y &= 0, \\
( y^2 z - y) \xi^2 -  y +  z &= 0.
\end{align*}
Eliminating~$x$ or~$y$ or~$z$ from these three equations, we find that~$(y_0,z_0)$ 
satisfy
\[
y^2-1 = z^2-1 = (y-z)(yz-1) = 0,
\]
and these equations have two solutions,
\[
(y_0,z_0)=(1,1) \quad\text{and}\quad (y_0,z_0)=(-1,-1).
\]
Finally, we substitute $k=-2(\xi+\xi^{-1})$ and $(x,y,z)=(\xi,\pm1,\pm1)$ into~\eqref{eqn:singlocusW1kx01} and verify that~$F$,~$F_y$, and~$F_z$ vanish.
This proves that
\[
\Sing\bigl( \Wcalf1_{-2(\xi+\xi^{-1}),\xi} \bigr)
=
\bigl\{ (\xi,1,1),\,(\xi,-1,-1) \bigr\}\quad\text{for all $\xi\ne0,\infty$,}
\]
which completes the proof of Proposition~\ref{proposition:singptsonWkfiber}. 
\end{proof}

\begin{remark}
For a general TIK3-surface, the three projection maps $\Wcal\to\PP^1$ give~$\Wcal$ three different structures as a surface fibered by genus~$1$ curves, and the corresponding Jacobian variety has a section of infinite order whose translation action on~$\Wcal$ is the~$\s_i$ associated to the projection. For MK3-surfaces, the~$\gS_3$-symmetry implies that the three structures are the same. Using the explicit description of the singular points on~$\Wcal_k$ in Proposition~\ref{proposition:singptsonWk} and the singular fibers of~$\Wcal_k$ in Proposition~\ref{proposition:singptsonWkfiber}, one could compute a N\'eron model for~$\Wcal_k\to\PP^1$ and compute the canonical height of the point on its Jacobian, but we will not do this computation in the present article. 
\end{remark}

\begin{proposition} 
\label{proposition:C1y0z0singular} 
Let~$\Wcal_k$ be the MK3-surface given in
Definition~\textup{\ref{definition:Wkx2y2z2kxyz}}, let~$F$ be the
associated polynomial, let~$y_0,z_0\in\PP^1$, and
let~$\Ccalf1_{y_0,z_0}$ be the curve associated to~$F$ as given in
Definition~\textup{\ref{definition:Ccalcurves}}.
If~$\Ccalf1_{y_0,z_0}$ is singular, then one of the following is
true:
\[
y_0~\text{or}~z_0 = 0~\text{or}~\infty,\quad
y_0^2 = z_0^2,\quad
y_0^2z_0^2 = 1,\quad
y_0~\text{or}~z_0 = \frac{\pm k \pm \sqrt{k^2\pm16}}{4}.
\]
By symmetry, analogous statements are true for~$\Ccalf2_{x_0,z_0}$
and~$\Ccalf3_{x_0,y_0}$.
\end{proposition}

\begin{corollary}
\label{corollary:ptsWstC1C2C3sing}
Let~$k\in\FF_q^*$. Then
\[
\#\left\{
(x_0,y_0,z_0) \in \Wcal_k(\FF_q) :
\begin{tabular}{@{}l@{}}
  one or more of $\Ccalf1_{y_0,z_0}$,\\
  $\Ccalf2_{x_0,z_0}$, $\Ccalf3_{x_0,y_0}$ is singular\\
\end{tabular}
\right\}
\le 144q.
\]  
\end{corollary}

\begin{proof}[Proof of Proposition \textup{\ref{proposition:C1y0z0singular}}]
To ease notation, we let~$b=y_0$ and~$c=z_0$. An affine piece of the
curve~$\Ccalf1_{b,c}$ is given by the equations
\[
F(x,b,z) = F(x,y,c) = 0.
\]
Hence a point~$(x,y,z)\in\Ccalf1_{b,c}$ is a singular point if and only if
\[
\rank\begin{bmatrix}
F_x(x,b,z) & 0 & F_z(x,b,z) \\
F_x(x,y,c) & F_y(x,y,c) & 0 \\
\end{bmatrix}\le 1.
\]
The rank condition and a bit of algebra yields three cases, which we
consider in turn.
\par\noindent
\textbf{Case 1:}  $\boldsymbol{F_z(x,b,z) = F_y(x,y,c) = 0}$.\enspace
In this case we are looking for values of~$b,c,k$ such that the equations
\[
F(x,b,z) = F(x,y,c) = F_z(x,b,z) = F_y(x,y,c) = 0
\]
have a solution~$(x,y,z)\in\AA^3$. Eliminating~$x,y,z$ from these four
equations gives the equation
\[
(b^2-c^2)(b^2c^2-1) = 0.
\]
Hence if there is a singular point, then $c=\pm{b}^{\pm1}$.
\par\noindent
\textbf{Case 2:}  $\boldsymbol{F_x(x,b,z) = F_z(x,b,z) = 0}$.\enspace 
In this case, which is a version of Proposition~\ref{proposition:singptsonWkfiber}, we are looking for values of~$b,c,k$ such that the equations
\[
F(x,b,z) = F(x,y,c) = F_x(x,b,z) = F_z(x,b,z) = 0
\]
have a solution~$(x,y,z)\in\AA^3$. Eliminating~$x,y,z$ from these four
equations gives the equation
\[
b^2(2b^2 - bk - 2)(2b^2 - bk + 2)(2b^2 + bk - 2)(2b^2 + bk + 2) = 0.
\]
Hence if there is a singular point, then
\[
b = 0 \quad\text{or}\quad b = \frac{\pm k \pm \sqrt{k^2\pm16}}{4}.
\]
\par\noindent
\textbf{Case 3:}  $\boldsymbol{F_x(x,y,c) = F_y(x,y,c) = 0}$.\enspace
By symmetry, this is the same as Case~2 with~$y\leftrightarrow{z}$
and~$b\leftrightarrow{c}$.
\end{proof}

\begin{proof}[Proof of Corollary \textup{\ref{corollary:ptsWstC1C2C3sing}}]
It suffices to bound the number of~$(y_0,z_0)\in\PP^1(\FF_q)$ such
that~$\Ccalf1_{y_0,z_0}$ is singular, and then multiply by~$3$ for the
$xyz$-symmetry and also multiply by~$2$ because each~$(y_0,z_0)$ may
yield~$2$ points on~$\Wcal_k$. (This includes some duplicates, so some
improvement is possible.)
\par
According to Proposition~\ref{proposition:C1y0z0singular}, the
singular cases are included in the following table, where again we do
not worry that some points appear more than once:
\[
\begin{array}{|c|c|} \hline
  (y_0, z_0) & \text{\# with $\Ccalf1_{y_0,z_0}$ singular} \\ \hline\hline
  y_0~\text{or}~z_0 = 0~\text{or}~\infty & {}\le 4q \\ \hline
  y_0^2 = z_0^2 \ne 0~\text{or}~\infty  & {}\le 2(q-1) \\ \hline
  y_0^2z_0^2 = 1 & {}\le 2(q-1) \\ \hline
  y_0~\text{or}~z_0 = \frac{\pm k \pm \sqrt{k^2\pm16}}{4} & {}\le 16q \\ \hline
\end{array}
\]
Hence there are at most~$24q$ pairs~$(y_0,z_0)$, and as noted earlier, this must
be multiplied by~$6$ to account for the other cases.
\end{proof}

%%%%%%%%%%%%%%%%%%%%%%%%%%%%%%%%%%%%%%%%%%%%%%%%%%%%%%%%%%%%%%%%%%%%%%
\section{Finite Orbits in \texorpdfstring{\mbox{$\Wcal_k(\CC)$}}{Wk(C)}}
\label{section:finiteorbitsoverC}
%%%%%%%%%%%%%%%%%%%%%%%%%%%%%%%%%%%%%%%%%%%%%%%%%%%%%%%%%%%%%%%%%%%%%%

Table~\ref{table:finiteorbschar0} describes finite $\Gp$-orbits
in~$\Wcal_k(\CC)$. We do not claim that this is the complete
list of possibilities. However, we note that the varied nature of the
finite orbits in the $1$-parameter family~$\Wcal_k$ suggests 
that any description of finite orbits over~$\CC$ on general TIK3-surfaces, or
even on MK3-surfaces, is likely to be quite complicated.

Most of the orbits in Table~\ref{table:finiteorbschar0} were unearthed
by examining small orbits in~$\Wcal_k(\FF_p)$ that appear in
Appendix~\ref{appendix:finitefieldtables} and
looking at specific properties of the points in the orbits.  We explain the process for a number of examples.

\begin{question}[Uniform Boundedness Question]
\label{question:uniformbdness}
For each~$k\in\CC$, we know from~\cite{arxiv2012.01762} that there are only finitely many finite~$\Gp$-orbits in~$\Wcal_k(\CC)$. Is there a bound that is independent of~$k$ for the largest such orbit? More generally, is there such a bound for finite orbits in~$\Wcal(\CC)$ as~$\Wcal$ runs over all MK3-surfaces? And even more generally, how about for all~TIK3-surfaces, although in this case we look at orbits for the group generated by the three involutions~$\s_1,\s_2,\s_3$? 
\end{question}

\begin{remark}
We mention that if we consider $\langle\s_1,\s_2,\s_3\rangle$-orbits, 
then the orbit of size~$144$ in Remark~\ref{remark:orbit144} consist of~$12$ orbits of size~$12$, the orbit of size~$160$ in Remark~\ref{remark:orbit160} consist of~$4$ orbits of size~$40$, and the  orbit of size~$288$ described in Remark~\ref{remark:orbit288} consist of~$12$ orbits of size~$24$. These provide lower bounds for
the putative uniform bounds discussed in Questions~\ref{question:NTIK3bd} and\ref{question:uniformbdness}.
\end{remark}

\begin{definition}[Trivial Orbits]
\label{definition:smallorbitsizes}
As noted in Proposition~\ref{proposition:singptsonWk}, the four singular points in~$\Wcal_k$ form two~$\Gp$-orbits, namely the fixed point
\[
\bigl\{ (0,0,0) \bigr\}
\]
and the orbit of size~$3$,
\[
\bigl\{ (0,\infty,\infty),\,(\infty,0,\infty),\,(\infty,\infty,0) \bigr\}.
\]
We will call these orbits the \emph{trivial orbits} in~$\Wcal_k$, and as
such, we have not included them in the table in Appendix~\ref{appendix:finitefieldtables}.
\end{definition}

\begin{remark}[One-dimensional families of finite orbits in~$\Wcal_k(\CC)$]
Table~\ref{table:finiteorbschar0} contains several examples of one-dimensional families of finite orbits in~$\Wcal_k(\CC)$, and indeed, these families are defined over~$\QQ$ or~$\QQ(i)$.
Ignoring the trivial orbits described in Definition~\ref{definition:smallorbitsizes}, we have the following examples:
\begin{description}
\item[Size 24]
There is a $k\in\QQ(t)$ such that $\Wcal_k\bigl(\QQ(t)\bigr)$ has a~$\Gp$-orbit of size~$24$.
\item[Size 48]
The set~$\Wcal_k\bigl(\QQ(i)\bigr)$ has a~$\Gp$-orbit of size~$48$
\item[Size 192]
There is a $k\in\QQ(t)$ such that $\Wcal_k\bigl(\QQ(t)\bigr)$ has a~$\Gp$-orbit of size~$192$.
\item[Size 288]
There is a curve~$C/\QQ$ of genus~$9$ and an element~$k\in\QQ(C)$ in the function field of~$C$ so that~$\Wcal_k\bigl(\QQ(C)\bigr)$ has a~$\Gp$-orbit of size~$288$.
\end{description}
\end{remark}

\begin{remark}[Orbits of Size 64]
We describe the derivation of the orbit of size~$64$ in
Table~\ref{table:finiteorbschar0}. Experimentally in
Appendix~\ref{appendix:finitefieldtables} we see
orbits of size~$64$ in~$\Wcal_k(\FF_p)$ for various values of~$p$
and~$k$, but the relation between~$p$ and~$k$ is not clear.  Examining
the actual orbits in several of these cases, we found that there was a
single point in~$\Wcal_k(\FF_p)$ of the form~$(\b,\b,\b)$, and that
the point~$(\b,\b,1)$ also appeared in~$\Wcal_k(\FF_p)$. We next
computed
\begin{align*}
  (\b,\b,\b)\in\Wcal_k &\quad\Longleftrightarrow\quad \b^6+k\b^3+3\b^2=0, \\
  (\b,\b,1)\in\Wcal_k &\quad\Longleftrightarrow\quad \b^4+(k+2)\b^2+1=0.
\end{align*}
Eliminating~$k$ and the trivial solutions~$\b\in\{0,1\}$ gives the
equation\footnote{We note that~$\b=0$ gives the contradiction~$1=0$,
  while~$\b=1$ yields~$k=-4$ and an orbit with fewer than~$64$
  elements.}
\[
\b^3+\b^2+\b-1=0.
\]
This gives $k=-(\b+\b^{-1})^2$.  It is then an exercise to compute
the~$\Gp$-orbit of~$(\b,\b,\b)$. It turns out to be the union of
the~$\Gp^\circ$ orbits of the following five points:
\[
\begin{array}{|c||c|c|c|c|c|} \hline
  \text{Point} & (\b,\b,\b) & (\b,\frac1\b,\frac1\b) & (\b,\b,1) & (\frac1\b,\frac1\b,1) & (\b,\frac1\b,1)  \\ \hline
  \text{Size of $\Gp^\circ$-orbit} & 4 & 12 & 12 & 12 & 24 \\ \hline
\end{array}
\]
\end{remark}

\begin{remark}[Orbits of Size 144]
\label{remark:orbit144}
The orbits of size~$144$ in Appendix~\ref{appendix:finitefieldtables} tend to feature points of the form
$(\a,\b,1)$ and $(\a,\b,-\b)$ that satisfy
\[
\s_1(\a,\b,-\b) = (\a,\b,-\b)
\quad\text{and}\quad
\s_3(\a,\b,-\b) = (\a,\b,1).
\]
We assume that~$\a,\b\notin\{0,\infty\}$ and that~$\b\ne-1$, and then
we obtain four conditions on~$k,\a,\b$:
\begin{align*}
  (\a,\b,1)\in\Wcal_k
  &\quad\Longleftrightarrow\quad
  k = -(\a+\a^{-1})(\b+\b^{-1}), \\
  %%%%%%%%%%%%%%%%%%%%%%%%%%%%%%%%%%%%%%%%
  (\a,\b,-\b)\in\Wcal_k
  &\quad\Longleftrightarrow\quad
  \a \b^2 k = \a^2 (\b^4+1) + 2\b^2 , \\
  %%%%%%%%%%%%%%%%%%%%%%%%%%%%%%%%%%%%%%%%
  \s_1(\a,\b,-\b) = (\a,\b,-\b)
  &\quad\Longleftrightarrow\quad
  \a^2 \b^2  (\b^4+1) = 2 \b^2, \\
  %%%%%%%%%%%%%%%%%%%%%%%%%%%%%%%%%%%%%%%%
  \s_3(\a,\b,-\b) = (\a,\b,1)
  &\quad\Longleftrightarrow\quad
  (\b^2 - \b + 1) \a^2 + \b = 0.
\end{align*}
The ideal in~$\ZZ[\a,\b,k]$ generated by these four
relations is also generated  (according to Magma) by the three relations
\[
\a^4 + 4\a^2 - 1 = 0,\quad
k = 4 \a (\a^2+4), \quad
\b^2 + (\a^2+3)\b + 1 = 0.
\]
(We also note that since~$\a\ne0$, we can replace the formula for~$k$ by~$k=4\a^{-1}$.)
\end{remark}

\begin{remark}[Orbits of Size 160]
\label{remark:orbit160}
The orbits of size~$160$ in Appendix~\ref{appendix:finitefieldtables} tend to include a single point of the
form~$(\b,\b,\b)$ having the property that
\begin{equation}
  \label{eqn:s1s3bbb1bstar}
  \s_1\circ\s_3(\b,\b,\b) = (1,\b,*).
\end{equation}
The assumption that~$(\b,\b,\b)\in\Wcal_k$ gives~$k=-(3+\b^4)/\b$, and then
computing~\eqref{eqn:s1s3bbb1bstar} explicitly gives
\[
\s_1\circ\s_3(\b,\b,\b) = \left(\frac{\b^9 + 2 \b^5 + 5 \b}{\b^8 + 6 \b^4 + 1},\b,\frac{2 \b}{\b^4 + 1}\right).
\]
Setting the first coordinate to~$1$ and discarding the trivial solution~$\b=1$ yields the condition
\[
\b^8 + 2 \b^4 - 4 \b^3 - 4 \b^2 - 4 \b + 1.
\]
Setting~$\g=2\b/(\b^4+1)$ for convenience, we find that the union of
the~$\Gp^\circ$-orbits of the following points is an orbit of
size~$160$.
\[
\begin{array}[t]{|c|c|} \hline
  \text{Point} & \text{Size of $\Gp^\circ$-orbit} \\ \hline\hline
  (\b,\b,\b) & 4 \\ \hline
  (\b^{-1},\b^{-1},\b) & 12 \\ \hline
  (\b,\b,\g) & 12 \\ \hline
  (\b^{-1},\b^{-1},\g)  & 12 \\ \hline
  (\b,\b^{-1},\g^{-1}) & 24  \\ \hline
\end{array}
\qquad
\begin{array}[t]{|c|c|} \hline
  \text{Point} & \text{Size of $\Gp^\circ$-orbit} \\ \hline\hline
  (1,\b,\g) & 24 \\ \hline
  (1,\b^{-1},\g) & 24 \\ \hline
  (1,\b,\g^{-1}) & 24 \\ \hline
  (1,\b^{-1},\g^{-1}) & 24 \\ \hline
\end{array}
\]
\end{remark}

\begin{remark}[Orbits of Size 288]
\label{remark:orbit288}  
There is an orbit of size~$288$ in~$\Wcal_{11}(\FF_{47})$ whose points
have coordinates in the following set of values:
\[
\begin{array}{|c||c|c|c|c|} \hline
  & t & -t & t^{-1} & -t^{-1} \\ \hline\hline
  \a & 3 & 44 & 16 & 31 \\ \hline
  \b & 6 & 41 & 8 & 39  \\ \hline
  \g & 11 & 36 & 30 & 17  \\ \hline
  \d & 15 & 32 & 22 & 25  \\ \hline
\end{array}
\]
In particular, we find that
\[
\s_3(3,6,11) = (3,6,15)\quad\text{in $\Wcal_{11}(\FF_{47})$.}
\]
If we now treat~$\a,\b,\g$ as indeterminates and want to require that
\[
(\a,\b,\g)\in\Wcal_k \quad\text{and that}\quad \s_3(\a,\b,\g)=(\a,\b,\d),
\]
then we find that~$k$ and~$\d$ are given by the formulas
\begin{align}
  \label{eqn:kabc}
  k &= -\frac{\a^2+\b^2+\g^2+\a^2\b^2\g^2}{\a\b\g},\\
  \label{eqn:deltaabc}
  \d &= \frac{\a^2+\b^2}{\g(\a^2\b^2+1)}.
\end{align}

Let~$P_1=(3,6,11)\in\Wcal_{11}(\FF_{47})$. Then the~$\Gp$-orbit
of~$P_1$ has size~$288$, while the sub-orbit
for~$\Gp^\s=\langle\s_1,\s_2,\s_3\rangle$ has size~$24$ and is
described in detail in Table~\ref{table:W11F47P3611}. We observe that
the subgroup of~$\Gp^\circ$ leaving the orbit~$\Gp^\s\cdot{P_1}$ 
invariant is
\[
\Stab_{\Gp^\circ}(\Gp^\s\cdot P_1) = \{e,\l\},\quad\text{where}\quad
\l : (x,y,z) \longmapsto (x,-z,-y).
\]
Hence the full~$\Gp$-orbit of~$P_1\in\Wcal_{11}(\FF_{47})$ has order 
\[
\#\Gp\cdot P_1
= \Bigl( \#\Gp^\s\cdot P_1 \Bigr) \cdot \left(\frac{\#\Gp^\circ}{\#\operatorname\Stab_{\Gp^\circ}(\Gp^\s\cdot P_1)}\right)
= 24 \cdot \frac{24}{2} = 288.
\]

Looking at Table~\ref{table:W11F47P3611}, we find many relations
in~$\Wcal_{11}(\FF_{47})$, including for example\footnote{We use the
  convenient notation~$\bfv[j]$ to denote the~$j$th coordinate of the
  vector~$\bfv$.}
\begin{equation}
  \label{eqn:ds1s2d3xx}
  \d = \s_1(\a,\b,\g)[1]^{-1} = -\s_2(\a,\b,\g)[2] = \s_3(\a,\b,\g)[3],
\end{equation}
and
\begin{equation}
  \label{eqn:s2s3abcs1s3negbinv}
  \s_2\circ\s_3(\a,\b,\g) = \s_1\circ\s_3(-\b^{-1},-\g,\a^{-1}).
\end{equation}
If we now view~\eqref{eqn:ds1s2d3xx}
and~\eqref{eqn:s2s3abcs1s3negbinv} as determining conditions on the
indeterminate quanitites~$\a,\b,\g$, we find that~$\a,\b,\g$ must satisfy
certain equations, and restricting to those equations that are satisfied by~$(3,6,11)$ in~$\FF_{47}$, we
find that~$\a,\b,\g$ must satisfy
\begin{align}
  \label{eqn:abgrel1}
  \a^3 \b^2 - \a^2 \b + \a - \b^3 &= 0, \\
  \label{eqn:abgrel2}
  \b^3 \g^3 - \b^2 + \b \g - \g^2 &= 0, \\
  \label{eqn:abgrel3}
  \a^3 \g^2 + \a^2 \g + \a + \g^3 &= 0.
\end{align}
These three relations for~$\a,\b,\g$ define a reducible subset
of~$\AA^3$, and a computation using Magma shows that this set
consists of two pieces. There is a finite set of points defined by
\begin{equation}
  \label{eqn:somepts}
  3\a + \g^3 =     \b + \g =     \g^4 + 3 = 0,
\end{equation}
and there is a geometrically irreducible reduced affine curve in~$\AA^3$ given by the equations
\begin{equation}
  \label{eqn:threeeqns}
  C =
  \left\{(\a,\b,\g) : 
  \begin{aligned}
    \a^2 \b - \a^2 \g + \a \b^2 \g^2 - \a + \b^2 \g - \b \g^2 &= 0 \\
    \a^2 \g^2 - \a \b^2 \g^3 + \a \b + \b \g^3 &= 0 \\
    \b^3 \g^3 - \b^2 + \b \g - \g^2 &= 0
  \end{aligned}
  \right\}
\end{equation}
We discard the points~\eqref{eqn:somepts}, since the orbit collapses
if~$\b=-\g$. A further computation shows that the affine curve~$C$ has
a unique singular point at~$(0,0,0)$ and that it has (geometric)
genus~$9$.

We let~$I$ denote the ideal in~$\QQ[\a,\b,\g]$ generated by the three
polynomials~\eqref{eqn:threeeqns} defining the curve~$C$.  Then for
each of the points~$P_j$ in Table~\ref{table:W11F47P3611},
treating~$\a,\b,\g$ as indeterminates and taking~$k$ and~$\d$
in~$\QQ(\a,\b,\g)$ as specified by~\eqref{eqn:kabc}
and~\eqref{eqn:deltaabc}, we used Magma to check that~$\s_i(P_j)$ is
as specified in Table~\ref{table:W11F47P3611} if we work in the
fraction field of the quotient ring $\QQ[\a,\b,\g]/I$. Hence the~$\Gp^\s$-orbit
of~$(\a,\b,\g)$ has size~$24$ when we work over this ring, and then as noted
earlier, the full~$\Gp$-orbit has size~$288$.
\par
In summary, we have shown that there is an irreducible affine curve $C/\QQ$ of geometric genus~$9$ and an element~$k\in\QQ(C)$ in the function field of~$C$ so that~$\Wcal_k\bigl(\QQ(C)\bigr)$ contains twelve~$\Gp^\s$-orbits of size~$24$ that combine to form one~$\Gp$-orbit of size~$288$.
\par
However, we note that there are points on the curve~$C(\CC)$ for which the orbit collapses.
Thus if we set~$\d$ to be equal to any of~$\a^{-1}$,~$-\b$, or~$\g$, then the~$\Gp^\circ$-orbits of the~$12$ points listed in Table~\ref{table:finiteorbschar0} collapse pairwise, and we obtain a total~$\Gp$-orbit of size~$144$, instead of~$288$. A short computation shows that if we don't allow $\a,\b,\g$ to be in $\{0,\pm1,\pm{i}\}$, then
\[
\d=\a^{-1}\Longrightarrow 3\a^4=-1, \quad
\d=-\b\Longrightarrow \b^4=-3, \quad
\d=\g\Longrightarrow \g^4=-3.
\]
\end{remark}

\begin{table}[t] 
\[
\begin{array}{|c|c|c|c|c|c|c|} \hline
  P & P & \s_1(P) & \s_2(P) & \s_3(P) \\ \hline\hline
P_{1} & (\a,\b,\g) & P_{2} & P_{5} & P_{7} \\ \hline
P_{2} & (\d^{-1},\b,\g) & P_{1} & P_{3} & P_{11} \\ \hline
P_{3} & (\d^{-1},-\a^{-1},\g) & P_{4} & P_{2} & \l P_{11} \\ \hline
P_{4} & (-\b^{-1},-\a^{-1},\g) & P_{3} & P_{6} & P_{10} \\ \hline
P_{5} & (\a,-\d,\g) & P_{6} & P_{1} & \l P_{7} \\ \hline
P_{6} & (-\b^{-1},-\d,\g) & P_{5} & P_{4} & \l P_{10} \\ \hline
P_{7} & (\a,\b,\d) & P_{8} & \l P_{5} & P_{1} \\ \hline
P_{8} & (\g^{-1},\b,\d) & P_{7} & P_{9} & P_{12} \\ \hline
P_{9} & (\g^{-1},-\a^{-1},\d) & P_{10} & P_{8} & \l P_{12} \\ \hline
P_{10} & (-\b^{-1},-\a^{-1},\d) & P_{9} & \l P_{6} & P_{4} \\ \hline
P_{11} & (\d^{-1},\b,\a^{-1}) & P_{12} & \l P_{3} & P_{2} \\ \hline
P_{12} & (\g^{-1},\b,\a^{-1}) & P_{11} & \l P_{9} & P_{8} \\ \hline
\end{array}
\]
\caption{The $\Gp^\s$-orbit of $(\a,\b,\g)=(3,6,11)\in\Wcal_{11}(\FF_{47})$, which we want to lift to a $\Gp^\s$-orbit in characteristic $0$. The map~$\l\in\Gp^\circ$ is $\l(x,y,z)=(x,-z,-y)$.}
\label{table:W11F47P3611}
\end{table}

\begin{remark}[Orbits of Size 288: A Cautionary Tale]
\label{remark:cautionarytale}
We have seen in Remark~\ref{remark:orbit288} that there is an entire $1$-parameter 
family of orbits of size~$288$ in characteristic~$0$. However, there are also exceptional
orbits of size~$288$ in finite characteristic that do not lift.
For example, we consider the orbit of size~$288$ in~$\Wcal_{11}(\FF_{53})$.
This orbit contains many points of the form~$(\a,-\a,1)$ and many points of the
form~$(0,\b,i\b)$. We note that an orbit containing points of this form does not fit
into the family described in Remark~\ref{remark:orbit288}, but this does not preclude it coming
from some other characteristic~$0$ orbit, so we continue analyzing the present example. 
In particular, we see that~$\Wcal_{11}(\FF_{53})$ contains the points
\[
(38,-38,1) \xrightarrow{\;\s_3\;}
(15,38,12) \xrightarrow{\;\s_2\;}
(15,11,12) \xrightarrow{\;\s_1\;}
(0,11,12).
\]
This suggests that we should take a point~$(\a,-\a,1)\in\Wcal_k$ satisfying
\begin{equation}
  \label{eqn:s1s2s3anea10}
  \s_1\circ\s_2\circ\s_3(\a,-\a,1) = (0,\b,i\b).
\end{equation}
The assumption that~$(\a,-\a,1)\in\Wcal_k$ forces~$k=(\a+\a^{-1})^2$,
and the assumption that the first coordinate
in~\eqref{eqn:s1s2s3anea10} is~$0$ forces
\begin{equation}
  \label{eqn:s1s2s3anea10a18}
  \a^{18} - 3 \a^{16} + 12 \a^{14} - 16 \a^{12} + 62 \a^{10} - 38 \a^8 +
  44 \a^6 - 8 \a^4 + 9 \a^2 + 1 = 0.
\end{equation}
We next observe that in~$\Wcal_{11}(\FF_{53})$, the orbit of~$(38,-38,1)$
has a~$\s_3$ fixed point, specifically
\begin{equation}
  \label{eqn:s2s338neg381fixs3}
  \s_2\circ\s_3(38,-38,1) = (15,11,12)\quad\text{is fixed by $\s_3$.}
\end{equation}
So in general we might want to impose the further condition that
\begin{equation}
  \label{eqn:s1s2s3anea10x}
  \s_3\circ\s_2\circ\s_3(\a,-\a,1) = \s_2\circ\s_3(\a,-\a,1)
\end{equation}
to mirror the behavior in~$\Wcal_{11}(\FF_{53})$.  Assuming
that~$\a\ne\pm1$, we find that~\eqref{eqn:s1s2s3anea10x} forces~$\a$
to satisfy
\begin{equation}
  \label{eqn:s1s2s3anea10a12}
  \a^{12} + 2 \a^{10} + 15 \a^8 + 12 \a^6 + 15 \a^4 + 2 \a^2 + 1 = 0.
\end{equation}
However, the conditions~\eqref{eqn:s1s2s3anea10a18}
and~\eqref{eqn:s1s2s3anea10a12} are incompatible in
characteristic~$0$.  Indeed, the resultant of the two polynomials in~\eqref{eqn:s1s2s3anea10a18}
and~\eqref{eqn:s1s2s3anea10a12} is equal to~$2^{80}\cdot53^{2}$, so the fact
that~\eqref{eqn:s2s338neg381fixs3} is true in~$\Wcal_{11}(\FF_{53})$ comes
from our choice of the specific finite field~$\FF_{53}$.
\end{remark}  

\begin{remark}[Orbits of size 256: Another Cautionary Tale]
\label{remark:orbit256}
There is an orbit of size~$256$ in~$\Wcal_{8}(\FF_{53})$ whose points
have coordinates in the following set of values:
\[
\{\pm1,\pm\a^{\pm1},\pm\b^{\pm1},\pm\g^{\pm1}\}
\quad\text{with}\quad
\a=16,\;\b=21,\;\g=39.
\]
In particular, there are points
\begin{align*}
P_1&=(\a,\a,\a)=(16,16,16)\in\Wcal_8(\FF_{53}), \\
P_2&=(\a,\a,\g^{-1})=(16,16,34)\in\Wcal_8(\FF_{53}), \\
P_3&=(1,\a,\b)=(1,16,21)\in\Wcal_8(\FF_{53}), \\
P_4&=(\a,\b,\g)=(16,21,39)\in\Wcal_8(\FF_{53}).
\end{align*}
We first note that
\begin{align}
P_1=(\a,\a,\a)\in \Wcal_k
&\quad\Longrightarrow\quad
k = -\dfrac{\a^4+3}{\a}, 
\notag\\
P_2=(\a,\a,\g^{-1})\in \Wcal_k 
&\quad\Longrightarrow\quad
\a^4+1-2\a\g=0\quad\text{(assuming $P_2\ne P_1$),} 
\label{eqn:abcrel256-1}\\
P_3 = (1,\a,\b) \in \Wcal_k 
&\quad\Longrightarrow\quad
(\a^2 + 1) \b^2 - (\a^4+3) \b + \a^2+1 = 0, 
\label{eqn:abcrel256-2}\\
P_4 = (\a,\b,\g) \in \Wcal_k
&\quad\Longrightarrow\quad
\a^2+\b^2+\g^2+\a^2\b^2\g^2 - (\a^4+3)\b\g = 0.
\label{eqn:abcrel256-3}
\end{align}
This gives three relations on~$\a,\b,\g$. 
We can use the orbit structure of~$\Wcal_8(\FF_{53})$ to generate additional relations such as
\begin{align}
\s_1(16,16,16) &=(39^{-1},16,16)\in\Wcal_8(\FF_{53}) \notag\\
&\quad\Longrightarrow\quad
\s_1(\a,\a,\a)=(\g^{-1},\a,\a) \in \Wcal_k \notag \\
&\quad\Longrightarrow\quad
\a^4-2\a\g+1 = 0,
\label{eqn:abcrel256-4} \\
\s_1(16,21,39)&=(16,21,39)\in\Wcal_8(\FF_{53}) \notag\\
&\quad\Longrightarrow\quad
\s_1(\a,\b,\g)=(\a,\b,\g) \in \Wcal_k \notag\\
&\quad\Longrightarrow\quad
\a^2(\a^4+3)\b^2 - (\a^4-1) = 0.
\label{eqn:abcrel256-5}
\end{align}
The five relations~\eqref{eqn:abcrel256-1}--\eqref{eqn:abcrel256-5} are incompatible in characteristic~$0$, although they do of course have the solution~$(\a,\b,\g)=(16,21,39)$ in~$\FF_{53}$. More precisely, the resultant of the five polynomials~\eqref{eqn:abcrel256-1}--\eqref{eqn:abcrel256-5} is~$9752=2^3\cdot23\cdot53$, and indeed in~$\Wcal_2(\FF_{23})$ we find an orbit of size~$256$ corresponding to~$(\a,\b,\g)=(6,11,18)$. So the orbits of size~$256$ in~$\Wcal_2(\FF_{23})$ and $\Wcal_8(\FF_{53})$ do not lift to characteristic~$0$.
\end{remark}

\begin{remark}[Orbits of Size 384: A Third Cautionary Tale]
\label{remark:orbit384}  
There is a point~$P_1=(22,22,-23)\in\Wcal_{13}(\FF_{71})$. A direct computation shows that~$\#\Gp\cdot{P_1}=384$. 
We let~$(\a,\b,\g,\d)=(22,23,9,44)$, and we consider the six points~$P_1\ldots,P_6\in\Wcal_{13}(\FF_{71})$ described in Table~\ref{table:W13F71Orbit372}. We also let~$\Gpplus^\circ\subset\Aut(\Wcal_k)$ be the subgroup containing~$96$ automorphisms that is described in Remark~\ref{remark:extradeltainvonWk}. Again by direct computation\footnote{Somewhat surprisingly, for this example we find that $\Gp^\s\cdot{P_1}=\Gp\cdot{P_1}=\Gpplus^\circ\Gp^\s\cdot{P_1}$ in $\Wcal_{13}(\FF_{71})$.} we find that~$\Gp\cdot{P_1}\subset\Wcal_{13}(\FF_{71})$ is invariant for~$\Gpplus^\circ$, and that it splits up into six $\Gpplus^\circ$-orbits with orbit representatives~$P_1,\ldots,P_6$ and orbits of size~$48$ or~$96$ as indicated in Table~\ref{table:W13F71Orbit372}.
\par
We now try to lift to characteristic~$0$, so we view~$\a,\b,\g,\d$ as indeterminates. However, it turns out that the six conditions
\[
P_i\in \Wcal_k \quad\text{for}\quad i=1,\ldots,6
\]
are inconsistent in~$\QQ[\a,\b,\g,\d,k]$.
\begin{table}[ht] 
\small
\[
\begin{array}{|c||c|c||c|c|c|} \hline
  \#\Gpplus^\circ P^{\strut} & P & P & \s_1(P) & \s_2(P) & \s_3(P) \\ \hline\hline
48 & P_1 & (\a,\a,-\b) & (\g^{-1},\a,-\b) & (\a,\g^{-1},-\b) & (\a,\a,-\g) \\ \hline
48 & P_2 & (\a,\a,-\g) & (\b^{-1},\a,-\g) & (\a,\b^{-1},-\g) & (\a,\a,-\b) \\ \hline
48 & P_3 & (\b,\b,\g) & (-\a^{-1},\b,\g) & (\b,-\a^{-1},\g) & (\b,\b,\d) \\ \hline
48 & P_4 & (\b,\b,\d) & (-1,\b,\d) & (\b,-1,\d) & (\b,\b,\g) \\ \hline
96 & P_5 & (\a,-\b,\g^{-1}) & (-\b^{-1},-\b,\g^{-1}) & (\a,-\a^{-1},\g^{-1}) & (\a,-\b,\a) \\ \hline
96 & P_6 & (\b,-\d,1) & (\b^{-1},-\d,1) & (\b,-\d^{-1},1) & (\b,-\d,-\b) \\ \hline
\end{array}
\]
\caption{The $\Gp$-orbit of $(\a,\a,-\b)=(22,22,-23)\in\Wcal_{13}(\FF_{71})$, with $\g=9$ and $\d=44$. We want to
  lift it to a $\Gp$-orbit in characteristic $0$. We note that every point in the last three
  columns is in the $\Gpplus^\circ$-orbit of one of $P_1,\ldots,P_6$.}
\label{table:W13F71Orbit372}
\end{table}
\end{remark}

\begin{table}[tp]
\small
\[
\begin{array}{|c|c|c|c|c|c|} \hline
%   \text{orbit size} 
  \begin{tabular}{@{}c@{}} orbit\\ size\\ \end{tabular}
  & k & \text{$\Gp^\circ$-generators} \\
  \hline\hline
  1 & \text{all $k$} & (0,0,0)
  \\ \hline
  3 & \text{all $k$} & (0,\infty,\infty)
  \\ \hline
  4 & k = 4 & (-1,-1,-1) 
%%   \begin{array}[t]{l}\zeta^4=1,\\k=-4\zeta^{-1}\\\end{array} & (\zeta,\zeta,\zeta)
  \\ \hline
  24 & \begin{array}[t]{l}\xi^4\ne1\\ k=-2(\xi+\xi^{-1})\\\end{array} & (\xi,1,1),\,(\xi^{-1},1,1)
  \\ \hline
  48 & \text{all $k$} & (1,i,0),\,(1,i,\infty)
  \\ \hline
  %%%%%
  64 & \begin{array}[t]{l} \b^3+\b^2+\b-1=0  \\ k=-(\b+\b^{-1})^2 \\\end{array} &
  \begin{array}[t]{ll}
    (\b,\b,\b), & (\b,\b,1) \\
    (\b^{-1},\b^{-1},1) & (\b,\b^{-1},\b^{-1}) \\
    (\b,\b^{-1},1) \\
  \end{array}
  \\ \hline
  %%%%%
  96 & \begin{array}[t]{l} \eta^4=-1  \\ k= -2\eta^2 \\\end{array} &
  \begin{array}[t]{ll}
    (\eta,\eta^3,0) & (\eta,\eta^3,\eta^6) \\
    (\eta,\eta^2,\eta^5) & (\eta,\eta^2,\infty) \\
  \end{array}
  \\ \hline  
  %%%%%
  144 & \begin{array}[t]{l} \a^4+4\a^2-1=0 \\ \b^2+(\a^2+3)\b+1=0 \\ 
  \b^4 + 2\b^3 - 2\b^2 + 2\b + 1=0 \\ k=4\a^{-1}  \\ 
  \end{array} &
  \begin{array}[t]{ll}
    (\a,\b,1), & (\a^{-1},\b,1), \\
    (\a,\b^{-1},1), & (\a^{-1},\b^{-1},1), \\
    (\a,\b,-\b), & (\a^{-1},\b^{-1},-\b) \\
  \end{array}
  \\ \hline  
  %%%%%
  160 & \begin{array}[t]{l}
    \b^8 + 2 \b^4 - 4 \b^3 \\
    \qquad{} - 4 \b^2 - 4 \b + 1 = 0 \\
    \g = 2\b/(\b^4+1) \\
    k = -(3+\b^4)/\b \\
  \end{array} &
  \begin{array}[t]{ll}
    (\b,\b,\b) & (1,\b,\g) \\
    (\b^{-1},\b^{-1},\b) & (1,\b^{-1},\g) \\
    (\b,\b,\g) & (1,\b,\g^{-1}) \\
    (\b^{-1},\b^{-1},\g) & (1,\b^{-1},\g^{-1}) \\
    (\b,\b^{-1},\g^{-1}) & \\
  \end{array}
  \\ \hline  
  %%%%%
  192 & \begin{array}[t]{l} \xi^8\ne1\\ k=i(\xi^2-\xi^{-2})\\\end{array} &
  \begin{array}[t]{ll}
    (\xi,i\xi,0),&(\xi,-i\xi,1), \\
    (\xi,i\xi^{-1},1),&(\xi,i\xi^{-1},\infty),\\
    (\xi^{-1},-i\xi,1),&(\xi^{-1},i\xi,\infty),\\
    (\xi^{-1},i\xi^{-1},0),&(\xi^{-1},i\xi^{-1},1)\\
  \end{array}
  \\ \hline  
  %%%%%%%%%%%
  \begin{tabular}[t]{c}
  %% 288, \\ or 144 \\ if $3\a^4=-1$ \\ or $\b^4=-3$ \\ or $\g^4=-3$ \\
  288 \\ or \\ 144$^*$ \\
  \end{tabular}
  &
  \begin{array}[t]{l}
    \a^2 \b - \a^2 \g + \a \b^2 \g^2 \\
    \qquad - \a + \b^2 \g - \b \g^2 = 0 \\
    \a^2 \g^2 - \a \b^2 \g^3 + \a \b + \b \g^3 = 0 \\
    \b^3 \g^3 - \b^2 + \b \g - \g^2 = 0\\
    \d = \dfrac{\a^2+\b^2}{\g(\a^2\b^2+1)}\\
    k = -\dfrac{\a^2+\b^2+\g^2+\a^2\b^2\g^2}{\a\b\g}\\
  \end{array}
  &
  \begin{array}[t]{ll}
    (\a,\b,\g) &   (\d^{-1},\b,\g) \\
    (\d^{-1},-\a^{-1},\g) &   (-\b^{-1},-\a^{-1},\g) \\
    (\a,\b,\d) &   (\g^{-1},\b,\d) \\
    (\g^{-1},-\a^{-1},\d) &   (-\b^{-1},-\a^{-1},\d) \\
    (\a,-\g,\d) &   (-\b^{-1},-\g,\d) \\
    (\d^{-1},\b,\a^{-1}) &   (\g^{-1},\b,\a^{-1}) \\
    \multicolumn{2}{c}{\hrulefill} \\
    \multicolumn{2}{l}{\text{\small $^*$Orbit size 144 if $3\a^4=-1$}} \\
    \multicolumn{2}{l}{\text{\small \phantom{$^*$}or $\b^4=-3$ or $\g^4=-3$}} \\
  \end{array}
  \\ \hline
\end{array}
\]
\caption{Finite $\Gp$-orbits in $\Wcal_k(\CC)$, where in each case
  we list only one of~$\Wcal_{\pm{k}}$ and~$\Wcal_{\pm{ik}}$;
  cf.\ Remark~\ref{remark:WkisomWd3kd41}.}
\label{table:finiteorbschar0}
\end{table}

%%%%%%%%%%%%%%%%%%%%%%%%%%%%%%%%%%%%%%%%%%%%%%%%%%%%%%%%%%%%%%%%%%%%%%
\section{Full Orbits in \texorpdfstring{$\Wcal_k(\FF_p)$}{Wk(Fp)}}
\label{section:totalorbits}
%%%%%%%%%%%%%%%%%%%%%%%%%%%%%%%%%%%%%%%%%%%%%%%%%%%%%%%%%%%%%%%%%%%%%%

In this section we consider total orbits in~$\Wcal_k(\FF_p)$.
Such orbits are necessarily finite.  In
Appendix~\ref{appendix:finitefieldtables} we
list the orbit structure for each~$3\le{p}\le113$. We first did these computations with 
a custom program that we wrote in PARI-GP~\cite{PARI2}. This program used a straightforward algorithm
to compute the points in~$\Wcal_k(\FF_p)$, and then a hash table to optimize finding and checking off points in
orbits. This program allowed us to compute the components of~$\Wcal_k(\FF_p)$ for~$p\le79$. We then 
reprogrammed the problem in Magma~\cite{MR1484478}. This allowed us to double-check the PARI-GP program, and
ultimately to extend the computations to larger primes. Our first Magma implementation used the permutation group package
in Magma and was a bit slower than PARI-GP. When we replaced the Magma permutation group package with the Magma graph theory package, the computations were roughly~$10$ times faster. This implementation used a Magma function that computes points on projective subvarieties of~$(\PP^1)^3(\FF_p)$. When we replaced this with a Magma function that computes points on affine subvarieties of~$\AA^3(\FF_p)$ and filled in the few extra points on~$\Wcal_k(\FF_p)$ lying at infinity, we picked up roughly another order of magnitude in speed. To give an idea of the resources used, we note that the program computed the orbits in~$\Wcal_k(\FF_{113})$ for~$29$ values of~$k$ in roughly~$31$~minutes on a MacBook Pro (2021) using an Apple M1 Pro chip.
\par
In view of the isomorphisms provided by Remark~\ref{remark:WkisomWd3kd41}, for
$p\equiv3\pmodintext4$ we compute the orbit structure
of~$\Wcal_k(\FF_p)$ for only one of~$\pm{k}\in\FF_p^*$; and for
$p\equiv1\pmodintext4$, we compute the orbit structure
of~$\Wcal_k(\FF_p)$ for only one of~$\pm{k},\pm{ik}\in\FF_p^*$,
where~$i=\sqrt{-1}\in\FF_p$.  In the tables in Appendix~\ref{appendix:finitefieldtables}, we have also omitted the trivial orbits of size $1$ and $3$ described in Definition~\ref{definition:smallorbitsizes}.
\par
Reducing the characteristic~$0$ orbits in Table~\ref{table:finiteorbschar0} modulo~$p$ yields some of the small characteristic~$p$ orbits in Appendix~\ref{appendix:finitefieldtables}. In particular,
Table~\ref{table:288char0tocharp} lists the characteristic~$p$ orbits of sizes~$144$,~$160$ and~$288$ for $p\le79$ that come from characteristic~$0$.

\begin{table}[ht]
\[
\begin{array}[t]{|c|c||c|c||c|} \hline
p & k & \a & \b & \text{Orbit size}  \\ \hline\hline
11 & 1 & 4 & 5 & 144 \\ \hline
19 & 8 & 11 & 4 & 144 \\ \hline
29 & 1 & 4 & 18 & 144 \\ \hline
29 & 11 & 3 & 2 & 144 \\ \hline
31 & 2 & 2 & 3 & 144 \\ \hline
59 & 9 & 7 & 21 & 144 \\ \hline
71 & 34 & 21 & 59 & 144 \\ \hline
79 & 6 & 27 & 63 & 144 \\ \hline
\multicolumn{5}{c}{\hidewidth
\parbox{.48\hsize}{\centering Orbits of size 144: Remark~\ref{remark:orbit144}}
\hidewidth}
\end{array}
\qquad\qquad
\begin{array}[t]{|c|c||c|c||c|} \hline
p & k & \b & \g & \text{Orbit size}  \\ \hline\hline
19 & 2 & 6 & 10 & 160 \\ \hline
23 & 5 & 20 & 19 & 160   \\ \hline
31 & 6 & 22 & 8 & 160  \\ \hline
41 & 1 & 25 & 35 & 160  \\ \hline
41 & 4 & 31 & 34 & 160  \\ \hline
59 & 8 & 36 & 38 & 160  \\ \hline
67 & 27 & 11 & 49 & 160  \\ \hline
73 & 18 & 9 & 16 & 160  \\ \hline
\multicolumn{5}{c}{\hidewidth
\parbox{.48\hsize}{\centering Orbits of size 160: Remark~\ref{remark:orbit160}}
\hidewidth}
\end{array}
\]
\[
\begin{array}{|c|c||c|c|c||c||c|} \hline
p & k & \a & \b & \g & \text{Orbit size} & \\ \hline\hline
19 & 9 & 7 & 2 & 3 & 144 & \b^4=-3 \\ \hline
23 & 4 & 10 & 8 & 9 & 288 & \\ \hline
43 & 2 & 28 & 13 & 14 & 144 & 3\a^4=-1 \\  \hline
47 & 11 & 3 & 6 & 11 & 288 & \\ \hline
59 & 23 & 13 & 33 & 8 & 288 & \\ \hline
61 & 15 & 4 & 7 & 18 & 288 & \\ \hline
67 & 31 & 5 & 30 & 12 &  144 & 3\a^4=-1 \\ \hline
71 & 13 & 10 & 44 & 16 & 288 & \\ \hline
79 & 35 & 36 & 8 & 59 & 288 & \\ \hline
79 & 36 & 12 & 19 & 51 & 288 & \\ \hline
\multicolumn{7}{c}{\hidewidth
\parbox{.9\hsize}{\centering Orbits of sizes 144 and 288: Remark~\ref{remark:orbit288}}
\hidewidth}
\end{array}
\]
\caption{$\Wcal(\FF_p)$ orbits of sizes 144, 160 and 288  in Tables~\ref{table:orbitsizeWk-1}--\ref{table:orbitsizeWk-4} coming from  $\Wcal(\Qbar)$ orbits in Table~\ref{table:finiteorbschar0}.}
\label{table:288char0tocharp}
\end{table}

%%%%%%%%%%%%%%%%%%%%%%%%%%%%%%%%%%%%%%%%%%%%%%%%%%%%%%%%%%%%%%%%%%%%%%
\section{Fibral Orbits in  \texorpdfstring{$\Wcal_k(\FF_p)$}{Wk(Fp)}}
\label{section:fibralorbits}
%%%%%%%%%%%%%%%%%%%%%%%%%%%%%%%%%%%%%%%%%%%%%%%%%%%%%%%%%%%%%%%%%%%%%%
We let
\[
\Gp = \langle\s_1,\s_2,\s_3,\t_{12},\t_{13},\t_{23},\e_{12},\e_{13},\e_{23}
\rangle \subset \Aut(\Wcal_k).
\]
For~$x_0,y_0,z_0\in{K}$, we denote the fibers of~$\Wcal_k(K)$ as usual by
\begin{align*}
\Wcalf1_{k,x_0} &= \bigl\{ (x_0,y,z) \in \Wcal_k(K) \bigr\},\\
\Wcalf2_{k,y_0} &= \bigl\{ (x,y_0,z) \in \Wcal_k(K) \bigr\},\\
\Wcalf3_{k,z_0} &= \bigl\{ (x,y,z_0) \in \Wcal_k(K) \bigr\}.
\end{align*}

The $\Gp$-fibral automorphism group of the fiber~$\Wcalf1_{k,x_0}$ is generated by
the two involutions~$\s_2$ and~$\s_3$ that fix~$x_0$, the
transposition~$\tau_{23}$ that swaps the~$y$ and~$z$ coordinates, and
the map~$\e_{23}$ that changes the sign of~$y$ and~$z$; and similarly
for the other fibers. Thus\footnote{We have listed more generators
  than needed.  For example,
  $\s_{3}=\tau_{23}\circ\s_{2}\circ\tau_{23}$, so
  $\Aut\bigl(\Wcalf1_{x_0}\bigr)=\langle \s_{2},\tau_{23},\e_{23}
  \rangle$, and similarly for the others fibers.}
\begin{align*}
\Gp^{(1)}_{x_0}  &= \langle \s_{2},\s_{3},\tau_{23},\e_{23} \rangle \subset  \Aut\bigl(\Wcalf1_{x_0}\bigr) ,\\
\Gp^{(2)}_{y_0} &= \langle \s_{1},\s_{3},\tau_{13},\e_{13} \rangle \subset \Aut\bigl(\Wcalf2_{y_0}\bigr),\\
\Gp^{(3)}_{z_0} &= \langle \s_{1},\s_{2},\tau_{12},\e_{12} \rangle\subset\Aut\bigl(\Wcalf3_{z_0}\bigr).
\end{align*}
We recall that since~$\Wcal_k$ is an MK3-surface, there is a set of points
\[
\pi\ConnFib\bigl(\Wcal_k(\FF_q)\bigr) \subset \PP^1(\FF_q)
\]
such that
\begin{multline*}
    t\in\pi\ConnFib\bigl(\Wcal_k(\FF_q)\bigr) 
\quad\Longleftrightarrow\quad\\
\text{$\Wcalf{i}_t(\FF_q)\in\Cage\bigl(\Wcal_k(\FF_q)\bigr)$ for one (equivalently all) $i\in\{1,2,3\}$.}
\end{multline*}

\begin{example}
\label{example:W1overF53cagedisconnected}
We consider the surface~$\Wcal_1$ over the finite field~$\FF_{53}$.
The set~$\Wcal_1(\FF_{53})$ has six $\Gp$-orbits of sizes,
respectively,~$1$,~$3$,~$24$,~$24$,~$48$ and~$3456$.  We compute the number of components on the various fibers, and when we do so, we find that
\begin{equation}
  \label{eqn:Wi1t0F53fibcomp}
  \pi\ConnFib\bigl(\Wcal_1(\FF_{53})\bigr)
  = \{
  \pm 2, \pm 4, \pm 6, \pm 13, \pm 20, \pm 24, \pm 26
  \}.
\end{equation}
Next, for each~$t$ in~$\pi\ConnFib\bigl(\Wcal_1(\FF_{53})\bigr)$, we
would like to know which of the coordinates
in~$\pi\ConnFib\bigl(\Wcal_1(\FF_{53})\bigr)$ appear as the coordinate of
some point in the (connected) fiber~$\Wcalf{i}_{t}(\FF_{53})$.  In
general, if~$S$ is any set of points in~$(\PP^1)^3$, we define
\[
\Flatten(S) =
\text{the set of all coordinates of all points in $S$}.
\]
Then we may compute the connectivity of the cage
of~$\Wcal_1(\FF_{53})$ using the data in the following table.
\[
\begin{array}{|c|c|} \hline
  t
  & \Flatten\bigl( \Wcalf1_{1,t}(\FF_{53})\bigr) \cap \pi\ConnFib\bigl(\Wcal_1(\FF_{53})\bigr)  
  \\ \hline\hline
  \pm2 & \{\pm 6, \pm 20\} \\ \hline
  \pm4 & \{\pm 24\} \\ \hline
  \pm6 & \{\pm 2, \pm 20, \pm 26\} \\ \hline
  \pm13 & \{\pm 24\} \\ \hline
  \pm20 & \{\pm 2, \pm 6, \pm 20, \pm 26\} \\ \hline
  \pm24 & \{\pm 4, \pm 13, \pm 24\} \\ \hline
  \pm26 & \{\pm 6, \pm 20\} \\ \hline
\end{array}
\]
Thus the cage in the big component of~$\Wcal_1(\FF_{53})$ is not
connected. It consists of the following two pieces, which are also
illustrated in Figure~\ref{figure:cageW1F53}:
\[
\bigcup_{t\in\{\pm2,\pm6,\pm20,\pm26\}}
\bigcup_{i\in\{1,2,3\}} \Wcalf{i}_{1,t}
\quad\text{and}\quad
\bigcup_{t\in\{\pm4,\pm13,\pm24\}}
\bigcup_{i\in\{1,2,3\}} \Wcalf{i}_{1,t}
\]
\end{example}

\begin{figure}
  \begin{picture}(300,120)(0,-60)
    \thicklines
    \put(0,0){\line(1,0){100}}
    \put(50,-50){\line(0,1){100}}
    \put(25,0){\line(1,1){50}}
    \put(25,0){\line(-1,-1){25}}
    \put(50,-25){\line(1,1){50}}
    \put(50,-25){\line(-1,-1){25}}
    \put(150,0){\line(1,0){100}}
    \put(180,-50){\line(0,1){100}}
    \put(220,-50){\line(0,1){100}}    
    \put(0,2){\makebox(0,0)[bl]{\CircleNum{20}}}
    \put(48,50){\makebox(0,0)[rt]{\CircleNum{6}}}
    \put(77,48){\makebox(0,0)[lt]{\CircleNum{26}}}
    \put(103,19){\makebox(0,0)[t]{\CircleNum{2}}}
    \put(150,2){\makebox(0,0)[bl]{\CircleNum{24}}}
    \put(178,50){\makebox(0,0)[tr]{\CircleNum{4}}}
    \put(218,50){\makebox(0,0)[tr]{\CircleNum{13}}}
   \end{picture}
  \caption{The two connected components of the cage of $\Wcal_1(\FF_{53})$,
    where the segment labeled~$\CircleNum{t}$ denotes the union of the six
    connected fibers $\cup_{i=1,2,3}\cup_{\e=\pm1}
    \Wcalf{i}_{1,\e{t}}(\FF_{53})$}
  \label{figure:cageW1F53}
\end{figure}

%% ----------------------------------------------------------------------
%% PARI command: K3FibralTable(1,5,41)
\begin{table}[p]
\tiny
\[
\begin{array}[t]{|c||*{11}{c|}} \hline
 t_0 \backslash p  & 5 & 7 & 11 & 13 & 17 & 19 & 23 & 29 & 31 & 37 & 41\\ \hline \hline
\infty & 2 & 1 & 1 & 4 & 6 & 1 & 1 & 8 & 1 & 10 & 12 \\ \hline
0 & 3 & 2 & 2 & 5 & 6 & 2 & 2 & 9 & 2 & 11 & 12 \\ \hline
1 & 2 & 1 & 1 & 2 & 2 & 2 & 2 & 3 & 3 & 4 & 3 \\ \hline
2 & 1 & 1 & 1 & 2 & 3 & 1 & 1 & 1 & 1 & 2 & 3 \\ \hline
3 & 1 & 1 & 1 & 2 & 2 & 0 & 1 & 2 & 1 & 3 & 1 \\ \hline
4 & 2 & 1 & 1 & 2 & 4 & 1 & 1 & 2 & 1 & 6 & 2 \\ \hline
5 &  & 1 & 1 & 2 & 3 & 1 & 1 & 1 & 1 & 4 & 2 \\ \hline
6 &  & 1 & 1 & 1 & 2 & 0 & 1 & 2 & 1 & 3 & 2 \\ \hline
7 &  &  & 1 & 1 & 2 & 1 & 1 & 3 & 1 & 1 & 1 \\ \hline
8 &  &  & 1 & 2 & 2 & 1 & 1 & 2 & 1 & 2 & 1 \\ \hline
9 &  &  & 1 & 2 & 2 & 1 & 1 & 2 & 1 & 4 & 4 \\ \hline
10 &  &  & 1 & 2 & 2 & 1 & 1 & 1 & 1 & 3 & 2 \\ \hline
11 &  &  &  & 2 & 2 & 1 & 2 & 2 & 2 & 2 & 1 \\ \hline
12 &  &  &  & 2 & 3 & 1 & 2 & 2 & 1 & 3 & 1 \\ \hline
13 &  &  &  &  & 4 & 0 & 1 & 2 & 1 & 3 & 4 \\ \hline
14 &  &  &  &  & 2 & 1 & 1 & 1 & 1 & 3 & 1 \\ \hline
15 &  &  &  &  & 3 & 1 & 1 & 1 & 1 & 2 & 2 \\ \hline
16 &  &  &  &  & 2 & 0 & 1 & 2 & 1 & 1 & 1 \\ \hline
17 &  &  &  &  &  & 1 & 1 & 2 & 1 & 3 & 1 \\ \hline
18 &  &  &  &  &  & 2 & 1 & 2 & 1 & 1 & 1 \\ \hline
19 &  &  &  &  &  &  & 1 & 1 & 1 & 1 & 6 \\ \hline
20 &  &  &  &  &  &  & 1 & 2 & 2 & 3 & 2 \\ \hline
21 &  &  &  &  &  &  & 1 & 2 & 1 & 1 & 2 \\ \hline
22 &  &  &  &  &  &  & 2 & 3 & 1 & 2 & 6 \\ \hline
23 &  &  &  &  &  &  &  & 2 & 1 & 3 & 1 \\ \hline
24 &  &  &  &  &  &  &  & 1 & 1 & 3 & 1 \\ \hline
25 &  &  &  &  &  &  &  & 2 & 1 & 3 & 1 \\ \hline
26 &  &  &  &  &  &  &  & 2 & 1 & 2 & 2 \\ \hline
27 &  &  &  &  &  &  &  & 1 & 1 & 3 & 1 \\ \hline
28 &  &  &  &  &  &  &  & 3 & 1 & 4 & 4 \\ \hline
29 &  &  &  &  &  &  &  &  & 1 & 2 & 1 \\ \hline
30 &  &  &  &  &  &  &  &  & 3 & 1 & 1 \\ \hline
31 &  &  &  &  &  &  &  &  &  & 3 & 2 \\ \hline
32 &  &  &  &  &  &  &  &  &  & 4 & 4 \\ \hline
33 &  &  &  &  &  &  &  &  &  & 6 & 1 \\ \hline
34 &  &  &  &  &  &  &  &  &  & 3 & 1 \\ \hline
35 &  &  &  &  &  &  &  &  &  & 2 & 2 \\ \hline
36 &  &  &  &  &  &  &  &  &  & 4 & 2 \\ \hline
37 &  &  &  &  &  &  &  &  &  &  & 2 \\ \hline
38 &  &  &  &  &  &  &  &  &  &  & 1 \\ \hline
39 &  &  &  &  &  &  &  &  &  &  & 3 \\ \hline
40 &  &  &  &  &  &  &  &  &  &  & 3 \\ \hline
\end{array}
\]
\caption{\# of fibral $\Aut(\Wcalf{i}_{1,t_0})$-orbits in $\Wcal_{1}(\FF_p)$ for $i=1,2,3$}
\label{table:fiberorbitsizesW1}
\end{table}
%%%%%%%%%%%%%%%%%%%%%%%%%%%%%%%%%%%%%%%%%%%%%%%%%%%%%%%%%%%%%%%%%%%%%%

\clearpage

%%%%%%%%%%%%%%%%%%%%%%%%%%%%%%%%%%%%%%%%%%%%%%%%%%%%%%%%%%%%%%%%%%%%%%
\section{The curious case of \texorpdfstring{$\Wcal_4(\FF_p)$}{W4Fp} with
\texorpdfstring{$p\equiv1\pmodintext{8}$}{p=1mod8} }
\label{appendix:W4Fpeq1mod8}
We close with the curious case of~$\Wcal_4(\FF_p)$, which seems to consistently
have two large orbits when~\text{$p\equiv1\pmodintext8$}. We remark that
the classical affine surface~$\Mcal_{1,4}$, which is known as the Cayley surface,
also has an unusual $\FF_p$-orbit structure due to the fact that it admits a double cover
by~$(\GG_m)^2$ in which the involutes~$\s_1,\s_2,\s_3$ become monomial maps; see for example~\cite{arxiv1812.07275}.
There are analogous~MK3 surfaces in which~$(\GG_m)^2$ is replaced by~$E^2$, but the fibers of such surfaces
are all isomorphic curves, while the $j$-invariants of the fibers of~$\Wcal_4$ vary, so~$\Wcal_4$ does not
appear to be an MK3 analogue of the Cayley surface. In any case, we list in Table~\ref{table:W4peq1mod8}
the sizes of the components of~$\Wcal(\FF_p)$ for all primes~$p\le113$ satisfying~$p\equiv1\pmodintext{8}$.

\begin{table}[ht]
\[
\begin{array}[t]{|c|c|c|c|} \hline 
     p  & \text{small orbits} & \text{two largest orbits}  \\ \hline\hline
    17 &    4, 16, 24, 48^2 & 64, 288 \\
    41 &    4, 24, 40, 48, 72, 120, 160, 192^3, 216 & 288, 576 \\
    73 &    4, 24, 40, 48, 120, 160, 192, 288^2 & 1920, 2976 \\
    89 &    4, 24, 48, 160^2, 192^2, 288^2 & 3264, 4512 \\
    97 &    4, 24, 48, 192, 960 & 3840, 5408 \\
   113 &    4, 24, 48 & 6656, 7488 \\ 
    \hline
\end{array}
\]
\caption{Orbit sizes in $\Wcal_4(\FF_p)$ for $p\equiv1\pmodintext{8}$}
\label{table:W4peq1mod8}
\end{table}

%%%%%%%%%%%%%%%%%%%%%%%%%%%%%%%%%%%%%%%%%%%%%%%%%%%%%%%%%%%%%%%%%%%%%%

\clearpage

%% \begin{thebibliography}{99}
%% \itemsep=\smallskipamount
%% \end{thebibliography}

%% \bibliographystyle{plain}
%% \bibliography{TIK3}

\clearpage

%%%%%%%%%%%%%%%%%%%%%%%%%%%%%%%%%%%%%%%%%%%%%%%%%%%%%%%%%%%%%%%%%%%%%%
\appendix
%%%%%%%%%%%%%%%%%%%%%%%%%%%%%%%%%%%%%%%%%%%%%%%%%%%%%%%%%%%%%%%%%%%%%%

%%%%%%%%%%%%%%%%%%%%%%%%%%%%%%%%%%%%%%%%%%%%%%%%%%%%%%%%%%%%%%%%%%%%%%
\section{Orbits of \texorpdfstring{$\Wcal_k$}{Wk} over finite fields} 
\label{appendix:finitefieldtables}
%%%%%%%%%%%%%%%%%%%%%%%%%%%%%%%%%%%%%%%%%%%%%%%%%%%%%%%%%%%%%%%%%%%%%%
This appendix contains tables listing the orbit sizes for~$\Wcal_k(\FF_p)$.

\begin{table}[htbp]
  \small
\[
\begin{array}[t]{|c|c|c|} \hline
  p & k & \text{orbit sizes} \\ \hline\hline
  3 & 1 & 4 \\ \hline
  \hline
  5 & 1 & 4,48 \\ \hline
  \hline
  7 & 1 & 64 \\ \hline
  7 & 2 & 24 \\ \hline
  7 & 3 & 4 \\ \hline
  \hline
  11 & 1 & 144 \\ \hline
  11 & 2 & 64 \\ \hline
  11 & 3 & 24 \\ \hline
  11 & 4 & 4,128 \\ \hline
  11 & 5 & 24,64 \\ \hline
  \hline
  13 & 1 & 24,48,192 \\ \hline
  13 & 2 & 24,40,48,64,120 \\ \hline
  13 & 4 & 4,48,192 \\ \hline
  \hline
  17 & 1 & 4,16,24,48^{2},64,288 \\ \hline
  17 & 2 & 48,96,192 \\ \hline
  17 & 3 & 24,48,384 \\ \hline
  17 & 6 & 24,48,160,192 \\ \hline
  \hline
  19 & 1 & 24,160 \\ \hline
  19 & 2 & 24,160 \\ \hline
  19 & 3 & 320 \\ \hline
  19 & 4 & 4,320 \\ \hline
  19 & 5 & 24,288 \\ \hline
  19 & 6 & 24,288 \\ \hline
  19 & 7 & 432 \\ \hline
  19 & 8 & 288 \\ \hline
  19 & 9 & 48,64,144^{2} \\ \hline
  \hline
  23 & 1 & 24,448 \\ \hline
  23 & 2 & 256,352 \\ \hline
  23 & 3 & 24,336 \\ \hline
  23 & 4 & 4,96,288 \\ \hline
  23 & 5 & 24,112,160 \\ \hline
  23 & 6 & 448 \\ \hline
  23 & 7 & 576 \\ \hline
  23 & 8 & 24,448 \\ \hline
  23 & 9 & 608 \\ \hline
  23 & 10 & 448 \\ \hline
  23 & 11 & 24,384 \\ \hline
  \hline
\end{array}
\qquad 
\begin{array}[t]{|c|c|c|} \hline
  p & k & \text{orbit sizes} \\ \hline\hline
  29 & 1 & 40,48,120,144,192,352 \\ \hline
  29 & 2 & 24,48,352,672 \\ \hline
  29 & 3 & 24^{2},48,1152 \\ \hline
  29 & 4 & 4,48,192^{2},288^{2} \\ \hline
  29 & 6 & 24^{2},48,1184 \\ \hline
  29 & 8 & 24,48,64,96,288,576 \\ \hline
  29 & 11 & 48,144,192^{2},384 \\ \hline
  \hline
  31 & 1 & 24,800 \\ \hline
  31 & 2 & 24,144,544 \\ \hline
  31 & 3 & 896 \\ \hline
  31 & 4 & 4,768 \\ \hline
  31 & 5 & 24,688 \\ \hline
  31 & 6 & 24,160,256,384 \\ \hline
  31 & 7 & 24,864 \\ \hline
  31 & 8 & 864 \\ \hline
  31 & 9 & 864 \\ \hline
  31 & 10 & 1024 \\ \hline
  31 & 11 & 1056 \\ \hline
  31 & 12 & 24,624 \\ \hline
  31 & 13 & 1120 \\ \hline
  31 & 14 & 24,800 \\ \hline
  31 & 15 & 1024 \\ \hline
  \hline
  37 & 1 & 36^{2},48,72^{2},160,192, \\ & & \qquad 216,288,384 \\ \hline
  37 & 2 & 24,48,72,216,576,672 \\ \hline
  37 & 3 & 24^{2},48,768,1056 \\ \hline
  37 & 4 & 4,48,192,384,960 \\ \hline
  37 & 5 & 24^{2},48,1792 \\ \hline
  37 & 8 & 24,48,480,1152 \\ \hline
  37 & 9 & 24,48,160,192,1312 \\ \hline
  37 & 10 & 24,48,1664 \\ \hline
  37 & 15 & 48,160,192^{2},288,624 \\ \hline
  \hline
\end{array}
\]
\caption{Non-trivial orbits in $\Wcal_k(\FF_p)$; cf.\ Definition~\ref{definition:smallorbitsizes}}
\label{table:orbitsizeWk-1}
\end{table}

\begin{table}[htbp]
  \small
\[
\begin{array}[t]{|c|c|c|} \hline
  p & k & \text{orbit sizes} \\ \hline\hline
  41 & 1 & 48,64,160,1632 \\ \hline
  41 & 2 & 24,40,48,96,120,192,1536 \\ \hline
  41 & 3 & 24,48,192,1824 \\ \hline
  41 & 4 & 4,24,40,48,72,120,160, \\ & & \qquad 192^{3},216,288,576 \\ \hline
  41 & 6 & 16,24,48^{2},192,1632 \\ \hline
  41 & 7 & 24,48,192,1792 \\ \hline
  41 & 8 & 24,48,192,1792 \\ \hline
  41 & 11 & 24,48,384,1600 \\ \hline
  41 & 12 & 24^{2},48,2160 \\ \hline
  41 & 16 & 48,96,192,1440 \\ \hline
  \hline
  43 & 1 & 1728 \\ \hline
  43 & 2 & 24,48,144,1536 \\ \hline
  43 & 3 & 24,1536 \\ \hline
  43 & 4 & 4,1856 \\ \hline
  43 & 5 & 24,1408 \\ \hline
  43 & 6 & 1632 \\ \hline
  43 & 7 & 1936 \\ \hline
  43 & 8 & 1968 \\ \hline
  43 & 9 & 1760 \\ \hline
  43 & 10 & 24,64,1600 \\ \hline
  43 & 11 & 1936 \\ \hline
  43 & 12 & 256,1504 \\ \hline
  43 & 13 & 24,1408 \\ \hline
  43 & 14 & 1728 \\ \hline
  43 & 15 & 2032 \\ \hline
  43 & 16 & 24,1408 \\ \hline
  43 & 17 & 24,384,1024 \\ \hline
  43 & 18 & 1968 \\ \hline
  43 & 19 & 24,1664 \\ \hline
  43 & 20 & 24,256,1408 \\ \hline
  43 & 21 & 24,1728 \\ \hline
  \hline
\end{array}
\quad
\begin{array}[t]{|c|c|c|} \hline
   p & k & \text{orbit sizes} \\ \hline\hline
  47 & 1 & 24,1712 \\ \hline
  47 & 2 & 2304 \\ \hline
  47 & 3 & 2112 \\ \hline
  47 & 4 & 4,1920 \\ \hline
  47 & 5 & 24,2080 \\ \hline
  47 & 6 & 2336 \\ \hline
  47 & 7 & 64,2016 \\ \hline
  47 & 8 & 24,2080 \\ \hline
  47 & 9 & 24,1776 \\ \hline
  47 & 10 & 24,2080 \\ \hline
  47 & 11 & 64,96,160,288,1728 \\ \hline
  47 & 12 & 24,64,2016 \\ \hline
  47 & 13 & 24,2080 \\ \hline
  47 & 14 & 1984 \\ \hline
  47 & 15 & 24,1776 \\ \hline
  47 & 16 & 864,1216 \\ \hline
  47 & 17 & 2304 \\ \hline
  47 & 18 & 2336 \\ \hline
  47 & 19 & 24,1712 \\ \hline
  47 & 20 & 24,2016 \\ \hline
  47 & 21 & 24,1776 \\ \hline
  47 & 22 & 2400 \\ \hline
  47 & 23 & 1984 \\ \hline
  \hline
  53 & 1 & 24^{2},48,3456 \\ \hline
  53 & 2 & 48,192,2736 \\ \hline
  53 & 3 & 24^{2},48,192,3360 \\ \hline
  53 & 4 & 4,48,3072 \\ \hline
  53 & 5 & 24,48,64,3168 \\ \hline
  53 & 6 & 24,48,192,3040 \\ \hline
  53 & 8 & 48,64,192,256,336,2016 \\ \hline
  53 & 10 & 24,48,192,3072 \\ \hline
  53 & 11 & 24,48,64,192,288,2688 \\ \hline
  53 & 13 & 24,48,192,288,2752 \\ \hline
  53 & 15 & 24,48,192,2944 \\ \hline
  53 & 17 & 24,48,192,3040 \\ \hline
  53 & 22 & 24,48,192^{2},2752 \\ \hline
  \hline
\end{array}  
\]
\caption{Non-trivial orbits in $\Wcal_k(\FF_p)$; cf.\ Definition~\ref{definition:smallorbitsizes}}
\label{table:orbitsizeWk-2}
\end{table}

\begin{table}[htbp]
\small
\[
\begin{array}[t]{|c|c|c|} \hline
  p & k & \text{orbit sizes} \\ \hline\hline
  59 & 1 & 3232 \\ \hline
  59 & 2 & 3328 \\ \hline
  59 & 3 & 3360 \\ \hline
  59 & 4 & 4,3392 \\ \hline
  59 & 5 & 24,2880 \\ \hline
  59 & 6 & 24,3264 \\ \hline
  59 & 7 & 3696 \\ \hline
  59 & 8 & 24,160,2848 \\ \hline
  59 & 9 & 144,160,3328 \\ \hline
  59 & 10 & 24,3008 \\ \hline
  59 & 11 & 24,2880 \\ \hline
  59 & 12 & 3792 \\ \hline
  59 & 13 & 24,3328 \\ \hline
  59 & 14 & 24,2880 \\ \hline
  59 & 15 & 160,3072 \\ \hline
  59 & 16 & 24,3008 \\ \hline
  59 & 17 & 3600 \\ \hline
  59 & 18 & 3232 \\ \hline
  59 & 19 & 3632 \\ \hline
  59 & 20 & 3328 \\ \hline
  59 & 21 & 24,3264 \\ \hline
  59 & 22 & 3232 \\ \hline
  59 & 23 & 24,96,288,2944 \\ \hline
  59 & 24 & 24,3328 \\ \hline
  59 & 25 & 24,2880 \\ \hline
  59 & 26 & 3632 \\ \hline
  59 & 27 & 24,3328 \\ \hline
  59 & 28 & 24,3136 \\ \hline
  59 & 29 & 3696 \\ \hline
  \hline
  61 & 1 & 24,48,4224 \\ \hline
  61 & 2 & 24^{2},48,4512 \\ \hline
  61 & 3 & 24,48,192,256,384,3424 \\ \hline
  61 & 4 & 4,48,192,384,3456 \\ \hline
  61 & 5 & 24^{2},48,4480 \\ \hline
  61 & 7 & 24,48,192,4032 \\ \hline
  61 & 8 & 24^{2},48,192,4288 \\ \hline
  61 & 9 & 24^{2},48,192^{2},4192 \\ \hline
  61 & 10 & 36^{2},48,72,192,288,3168 \\ \hline
\end{array}
\qquad  
\begin{array}[t]{|c|c|c|} \hline
  p & k & \text{orbit sizes} \\ \hline\hline  
  61 & 13 & 48,64,544,3248 \\ \hline
  61 & 14 & 24,48,352,3904 \\ \hline
  61 & 15 & 24,48,96,288^{3},3264 \\ \hline
  61 & 19 & 48,192^{2},288,3184 \\ \hline
  61 & 20 & 48,288,3568 \\ \hline
  61 & 25 & 24,48,192,3936 \\ \hline
  \hline
  67 & 1 & 4320 \\ \hline
  67 & 2 & 24,4256 \\ \hline
  67 & 3 & 24,3808 \\ \hline
  67 & 4 & 4,4544 \\ \hline
  67 & 5 & 24,4256 \\ \hline
  67 & 6 & 4656 \\ \hline
  67 & 7 & 24,3936 \\ \hline
  67 & 8 & 4624 \\ \hline
  67 & 9 & 24,4320 \\ \hline
  67 & 10 & 24,3808 \\ \hline
  67 & 11 & 4720 \\ \hline
  67 & 12 & 4352 \\ \hline
  67 & 13 & 24,4128 \\ \hline
  67 & 14 & 4624 \\ \hline
  67 & 15 & 4352 \\ \hline
  67 & 16 & 24,3936 \\ \hline
  67 & 17 & 4224 \\ \hline
  67 & 18 & 24,4256 \\ \hline
  67 & 19 & 24,4256 \\ \hline
  67 & 20 & 24,3936 \\ \hline
  67 & 21 & 24,3808 \\ \hline
  67 & 22 & 4720 \\ \hline
  67 & 23 & 4320 \\ \hline
  67 & 24 & 24,3808 \\ \hline
  67 & 25 & 24,4128 \\ \hline
  67 & 26 & 480,3840 \\ \hline
  67 & 27 & 96,160,288,4080 \\ \hline
  67 & 28 & 288,4528 \\ \hline
  67 & 29 & 24,4320 \\ \hline
  67 & 30 & 4624 \\ \hline
  67 & 31 & 48,144,4032 \\ \hline
  67 & 32 & 4352 \\ \hline
  67 & 33 & 24,3808 \\ \hline
  \hline
\end{array}
\]  
\caption{Non-trivial orbits in $\Wcal_k(\FF_p)$; cf.\ Definition~\ref{definition:smallorbitsizes}}
\label{table:orbitsizeWk-3}
\end{table}

\begin{table}[htbp]
\footnotesize
\[
\begin{array}[t]{|c|c|c|} \hline
  p & k & \text{orbit sizes} \\ \hline\hline
  71 & 1 & 5280 \\ \hline
  71 & 2 & 4768 \\ \hline
  71 & 3 & 24,4560 \\ \hline
  71 & 4 & 4,4608 \\ \hline
  71 & 5 & 24,4800 \\ \hline
  71 & 6 & 24,4864 \\ \hline
  71 & 7 & 5376 \\ \hline
  71 & 8 & 24,4368 \\ \hline
  71 & 9 & 5184 \\ \hline
  71 & 10 & 4864 \\ \hline
  71 & 11 & 5280 \\ \hline
  71 & 12 & 24,4304 \\ \hline
  71 & 13 & 96,288,384,4096 \\ \hline
  71 & 14 & 24,4864 \\ \hline
  71 & 15 & 5216 \\ \hline
  71 & 16 & 24,4800 \\ \hline
  71 & 17 & 24,4864 \\ \hline
  71 & 18 & 24,4672 \\ \hline
  71 & 19 & 5184 \\ \hline
  71 & 20 & 24,4864 \\ \hline
  71 & 21 & 5216 \\ \hline
  71 & 22 & 4864 \\ \hline
  71 & 23 & 24,4368 \\ \hline
  71 & 24 & 4864 \\ \hline
  71 & 25 & 4768 \\ \hline
  71 & 26 & 5216 \\ \hline
  71 & 27 & 24,4672 \\ \hline
  71 & 28 & 24,4304 \\ \hline
  71 & 29 & 4864 \\ \hline
  71 & 30 & 24,4304 \\ \hline
  71 & 31 & 4864 \\ \hline
  71 & 32 & 5216 \\ \hline
  71 & 33 & 24,4368 \\ \hline
  71 & 34 & 24,144,4224 \\ \hline
  71 & 35 & 24,4800 \\ \hline
  \hline
  73 & 1 & 48,192,5248 \\ \hline
  73 & 2 & 24,48,96,5760 \\ \hline
  73 & 3 & 24,48,64,5920 \\ \hline
  73 & 4 & 4,24,40,48,120,160, \\
  & & \qquad 192,288^{2},1920,2976 \\ \hline
  73 & 5 & 24^{2},48,6448 \\ \hline
  73 & 6 & 48,192,5376 \\ \hline
  73 & 7 & 24,48,5952 \\ \hline
  73 & 9 & 24^{2},48,6288 \\ \hline
  73 & 10 & 48,192,5248 \\ \hline
  73 & 12 & 24,48,192,5792 \\ \hline
\end{array}\qquad
\begin{array}[t]{|c|c|c|} \hline
  p & k & \text{orbit sizes} \\ \hline\hline
  73 & 13 & 48,192,672,4576 \\ \hline
  73 & 15 & 48,192,544,4704 \\ \hline
  73 & 17 & 24,48,192,5760 \\ \hline
  73 & 18 & 24^{2},48,160,192,6000 \\ \hline
  73 & 20 & 16,24,48^{2},192,5728 \\ \hline
  73 & 23 & 24,48,5856 \\ \hline
  73 & 26 & 24^{2},48,6256 \\ \hline
  73 & 31 & 24,48,192,5792 \\ \hline
  \hline  
  %% It took 20 minutes to compute p = 79 on a MacPro with 2.9 GHz Dual-Core Intel Core i5
  79 & 1 & 24,5856 \\ \hline
  79 & 2 & 24,5424 \\ \hline
  79 & 3 & 24,5488 \\ \hline
  79 & 4 & 4,5760 \\ \hline
  79 & 5 & 24,6048 \\ \hline
  79 & 6 & 24,144,5344 \\ \hline
  79 & 7 & 5952 \\ \hline
  79 & 8 & 5792 \\ \hline
  79 & 9 & 24,5488 \\ \hline
  79 & 10 & 24,5984 \\ \hline
  79 & 11 & 24,5984 \\ \hline
  79 & 12 & 24,5424 \\ \hline
  79 & 13 & 6432 \\ \hline
  79 & 14 & 24,6048 \\ \hline
  79 & 15 & 24,5488 \\ \hline
  79 & 16 & 6400 \\ \hline
  79 & 17 & 24,5984 \\ \hline
  79 & 18 & 6592 \\ \hline
  79 & 19 & 6400 \\ \hline
  79 & 20 & 6048 \\ \hline
  79 & 21 & 5952 \\ \hline
  79 & 22 & 24,5488 \\ \hline
  79 & 23 & 6496 \\ \hline
  79 & 24 & 6496 \\ \hline
  79 & 25 & 6048 \\ \hline
  79 & 26 & 6432 \\ \hline
  79 & 27 & 24,5984 \\ \hline
  79 & 28 & 6080 \\ \hline
  79 & 29 & 5792 \\ \hline
  79 & 30 & 6496 \\ \hline
  79 & 31 & 24,6048 \\ \hline
  79 & 32 & 5952 \\ \hline
  79 & 33 & 24,5984 \\ \hline
  79 & 34 & 6592 \\ \hline
  79 & 35 & 96,288,6112 \\ \hline
  79 & 36 & 24,96,288,5664 \\ \hline
  79 & 37 & 24,5680 \\ \hline
  79 & 38 & 5952 \\ \hline
  79 & 39 & 24,64,5616 \\ \hline
  \hline
\end{array}
\]
\caption{Non-trivial orbits in $\Wcal_k(\FF_p)$; cf.\ Definition~\ref{definition:smallorbitsizes}}
\label{table:orbitsizeWk-4}
\end{table}

\begin{table}[htbp]
\footnotesize
\[
\begin{array}[t]{|c|c|c|} \hline
  q &    k & \text{orbit sizes} \\ \hline\hline  
   83 &    1 & 24, 96, 288, 5664 \\ \hline 
   83 &    2 & 7248 \\ \hline 
   83 &    3 & 6720 \\ \hline 
   83 &    4 & 4, 7040 \\ \hline 
   83 &    5 & 24, 6176 \\ \hline 
   83 &    6 & 7088 \\ \hline 
   83 &    7 & 24, 6048 \\ \hline 
   83 &    8 & 24, 6496 \\ \hline 
   83 &    9 & 24, 6496 \\ \hline 
   83 &   10 & 24, 6176 \\ \hline 
   83 &   11 & 7248 \\ \hline 
   83 &   12 & 6720 \\ \hline 
   83 &   13 & 24, 6624 \\ \hline 
   83 &   14 & 7056 \\ \hline 
   83 &   15 & 6688 \\ \hline 
   83 &   16 & 6432 \\ \hline 
   83 &   17 & 7088 \\ \hline 
   83 &   18 & 24, 6688 \\ \hline 
   83 &   19 & 7152 \\ \hline 
   83 &   20 & 6688 \\ \hline 
   83 &   21 & 24, 6688 \\ \hline 
   83 &   22 & 7088 \\ \hline 
   83 &   23 & 7088 \\ \hline 
   83 &   24 & 6592 \\ \hline 
   83 &   25 & 24, 6496 \\ \hline 
   83 &   26 & 6592 \\ \hline 
   83 &   27 & 24, 6048 \\ \hline 
   83 &   28 & 24, 96, 288, 6304 \\ \hline 
   83 &   29 & 24, 6048 \\ \hline 
   83 &   30 & 6688 \\ \hline 
   83 &   31 & 6688 \\ \hline 
   83 &   32 & 24, 6176 \\ \hline 
   83 &   33 & 24, 6176 \\ \hline 
   83 &   34 & 24, 6176 \\ \hline 
   83 &   35 & 7056 \\ \hline 
   83 &   36 & 7088 \\ \hline 
   83 &   37 & 24, 6624 \\ \hline 
   83 &   38 & 24, 6048 \\ \hline 
   83 &   39 & 24, 64, 6624 \\ \hline 
   83 &   40 & 24, 6496 \\ \hline 
   83 &   41 & 6688 \\ \hline 
\hline
\end{array}
\qquad
\begin{array}[t]{|c|c|c|} \hline
  q &    k & \text{orbit sizes} \\ \hline\hline  
   89 &    1 & 24, 48, 192^2, 8320 \\ \hline 
   89 &    2 & 24, 48, 96, 192, 8320 \\ \hline 
   89 &    3 & 24, 48, 96, 192, 288^2, 7872 \\ \hline 
   89 &    4 & 4, 24, 48, 160^2, 192^2, \\
   & & \qquad 288^2, 3264, 4512 \\ \hline 
   89 &    5 & 24, 48, 8608 \\ \hline 
   89 &    6 & 24, 48, 192, 8416 \\ \hline 
   89 &    7 & 48, 192, 288, 7584 \\ \hline 
   89 &    9 & 24, 48, 8448 \\ \hline 
   89 &   10 & 24, 48, 8448 \\ \hline 
   89 &   11 & 24, 48, 192, 8512 \\ \hline 
   89 &   12 & 24^2, 48, 9264 \\ \hline 
   89 &   14 & 24^2, 48, 9072 \\ \hline 
   89 &   15 & 16, 48^2, 8128 \\ \hline 
   89 &   17 & 48, 8192 \\ \hline 
   89 &   19 & 24^2, 48, 144, 192^2, 8640 \\ \hline 
   89 &   20 & 48, 192, 7872 \\ \hline 
   89 &   22 & 24, 48, 8608 \\ \hline 
   89 &   25 & 24, 48, 8736 \\ \hline 
   89 &   27 & 24, 48, 8704 \\ \hline 
   89 &   30 & 40, 48, 120, 8032 \\ \hline 
   89 &   33 & 24, 48, 8704 \\ \hline 
   89 &   38 & 24^2, 48, 144, 192, 8768 \\ \hline    
   \hline
\end{array}
\]
\caption{Non-trivial orbits in $\Wcal_k(\FF_p)$; cf.\ Definition~\ref{definition:smallorbitsizes}}
\label{table:orbitsizeWk-5}
\end{table}

\begin{table}[htbp]
\footnotesize
\[
\begin{array}[t]{|c|c|c|} \hline 
    q &    k & \text{orbit sizes} \\ \hline\hline
   97 &    1 & 48, 192, 9504 \\ \hline 
   97 &    2 & 24, 48, 96, 672, 9408 \\ \hline 
   97 &    3 & 16, 24, 48^2, 160, 10080 \\ \hline 
   97 &    4 & 4, 24, 48, 192, 960, 3840, 5408 \\ \hline 
   97 &    5 & 24, 48, 10304 \\ \hline 
   97 &    6 & 48, 192, 9376 \\ \hline 
   97 &    7 & 24^2, 48, 10672 \\ \hline 
   97 &    8 & 24, 48, 10304 \\ \hline 
   97 &   10 & 48, 192, 9376 \\ \hline 
   97 &   11 & 24, 40, 48, 120, 9856 \\ \hline 
   97 &   12 & 24, 48, 10304 \\ \hline 
   97 &   14 & 24, 48, 192, 10080 \\ \hline 
   97 &   15 & 24, 48, 10304 \\ \hline 
   97 &   16 & 48, 9696 \\ \hline 
   97 &   19 & 24^2, 48, 10864 \\ \hline 
   97 &   20 & 24, 48, 192^2, 9792 \\ \hline 
   97 &   21 & 48, 9696 \\ \hline 
   97 &   24 & 24, 48, 192, 10080 \\ \hline 
   97 &   25 & 24, 48, 192, 10080 \\ \hline 
   97 &   28 & 24^2, 48, 192, 10576 \\ \hline 
   97 &   29 & 24^2, 48, 192, 10512 \\ \hline 
   97 &   33 & 24, 48, 192, 10080 \\ \hline 
   97 &   37 & 24, 48, 96, 288^2, 9344 \\ \hline 
   97 &   42 & 24, 48, 192, 9824 \\ \hline 
\hline
\end{array}
\qquad
\begin{array}[t]{|c|c|c|} \hline 
    q &    k & \text{orbit sizes} \\ \hline\hline
  101 &    1 & 24, 48, 192, 10912 \\ \hline 
  101 &    2 & 24, 48, 11104 \\ \hline 
  101 &    3 & 24, 48, 192, 10944 \\ \hline 
  101 &    4 & 4, 48, 192^2, 288^2, 9792 \\ \hline 
  101 &    5 & 24, 48, 192, 10912 \\ \hline 
  101 &    6 & 24^2, 48, 11552 \\ \hline 
  101 &    7 & 24^2, 48, 11712 \\ \hline 
  101 &    8 & 24, 48, 192^2, 10464 \\ \hline 
  101 &    9 & 24^2, 48, 192, 11360 \\ \hline 
  101 &   12 & 48, 60^2, 120, 192^2, 9728 \\ \hline 
  101 &   13 & 24^2, 48, 192, 11328 \\ \hline 
  101 &   14 & 24, 48, 352, 10656 \\ \hline 
  101 &   15 & 48, 10608 \\ \hline 
  101 &   16 & 24, 48, 160, 192, 10656 \\ \hline 
  101 &   17 & 24, 48, 11104 \\ \hline 
  101 &   18 & 24^2, 48, 11552 \\ \hline 
  101 &   23 & 48, 10352 \\ \hline 
  101 &   24 & 24, 48, 11008 \\ \hline 
  101 &   25 & 48, 64, 96^2, 144, 192, 288^2, 9184 \\ \hline 
  101 &   26 & 24, 48, 11104 \\ \hline 
  101 &   27 & 24, 40, 48, 120, 192, 480, 10272 \\ \hline 
  101 &   34 & 48, 144, 192, 10080 \\ \hline 
  101 &   35 & 48, 10416 \\ \hline 
  101 &   36 & 24^2, 40, 48, 120, 192^2, 11296 \\ \hline 
  101 &   45 & 24, 48, 192, 10816 \\ \hline 
\hline
\end{array}
\]
\caption{Non-trivial orbits in $\Wcal_k(\FF_p)$; cf.\ Definition~\ref{definition:smallorbitsizes}}
\label{table:orbitsizeWk-6}
\end{table}

\begin{table}[htbp]
\footnotesize
\[
\begin{array}[t]{|c|c|c|} \hline 
    q &    k & \text{orbit sizes} \\ \hline\hline
  103 &    1 & 10112 \\ \hline 
  103 &    2 & 24, 10304 \\ \hline 
  103 &    3 & 10400 \\ \hline 
  103 &    4 & 4, 9984 \\ \hline 
  103 &    5 & 24, 9616 \\ \hline 
  103 &    6 & 10368 \\ \hline 
  103 &    7 & 24, 10176 \\ \hline 
  103 &    8 & 11136 \\ \hline 
  103 &    9 & 10400 \\ \hline 
  103 &   10 & 10272 \\ \hline 
  103 &   11 & 24, 9616 \\ \hline 
  103 &   12 & 24, 9984 \\ \hline 
  103 &   13 & 24, 9552 \\ \hline 
  103 &   14 & 10848 \\ \hline 
  103 &   15 & 96, 288, 10464 \\ \hline 
  103 &   16 & 24, 9552 \\ \hline 
  103 &   17 & 11008 \\ \hline 
  103 &   18 & 10816 \\ \hline 
  103 &   19 & 24, 9808 \\ \hline 
  103 &   20 & 64, 10048 \\ \hline 
  103 &   21 & 24, 10368 \\ \hline 
  103 &   22 & 24, 10368 \\ \hline 
  103 &   23 & 10368 \\ \hline 
  103 &   24 & 10848 \\ \hline 
  103 &   25 & 10400 \\ \hline 
  \end{array}
\qquad
\begin{array}[t]{|c|c|c|} \hline 
    q &    k & \text{orbit sizes} \\ \hline\hline
103 &   26 & 10912 \\ \hline 
  103 &   27 & 11008 \\ \hline 
  103 &   28 & 10400 \\ \hline 
  103 &   29 & 24, 9616 \\ \hline 
  103 &   30 & 24, 9616 \\ \hline 
  103 &   31 & 24, 9616 \\ \hline 
  103 &   32 & 24, 10176 \\ \hline 
  103 &   33 & 11008 \\ \hline 
  103 &   34 & 24, 10176 \\ \hline 
  103 &   35 & 24, 64, 10240 \\ \hline 
  103 &   36 & 10112 \\ \hline 
  103 &   37 & 24, 9616 \\ \hline 
  103 &   38 & 10112 \\ \hline 
  103 &   39 & 24, 10304 \\ \hline 
  103 &   40 & 64, 10944 \\ \hline 
  103 &   41 & 24, 9808 \\ \hline 
  103 &   42 & 24, 9808 \\ \hline 
  103 &   43 & 24, 10368 \\ \hline 
  103 &   44 & 24, 9808 \\ \hline 
  103 &   45 & 24, 10304 \\ \hline 
  103 &   46 & 10272 \\ \hline 
  103 &   47 & 24, 10304 \\ \hline 
  103 &   48 & 10848 \\ \hline 
  103 &   49 & 24, 10304 \\ \hline 
  103 &   50 & 10816 \\ \hline 
  103 &   51 & 10912 \\ \hline 
\hline
\end{array}
\]
\caption{Non-trivial orbits in $\Wcal_k(\FF_p)$; cf.\ Definition~\ref{definition:smallorbitsizes}}
\label{table:orbitsizeWk-7}
\end{table}

\begin{table}[htbp]
\footnotesize
\[
\begin{array}[t]{|c|c|c|} \hline 
    q &    k & \text{orbit sizes} \\ \hline\hline
  107 &    1 & 24, 11136 \\ \hline 
  107 &    2 & 11696 \\ \hline 
  107 &    3 & 24, 10752 \\ \hline 
  107 &    4 & 4, 11264 \\ \hline 
  107 &    5 & 24, 10368 \\ \hline 
  107 &    6 & 11104 \\ \hline 
  107 &    7 & 24, 11008 \\ \hline 
  107 &    8 & 24, 96, 288, 10624 \\ \hline 
  107 &    9 & 96, 288, 11280 \\ \hline 
  107 &   10 & 11104 \\ \hline 
  107 &   11 & 24, 10496 \\ \hline 
  107 &   12 & 11232 \\ \hline 
  107 &   13 & 11696 \\ \hline 
  107 &   14 & 11200 \\ \hline 
  107 &   15 & 24, 10368 \\ \hline 
  107 &   16 & 11696 \\ \hline 
  107 &   17 & 11696 \\ \hline 
  107 &   18 & 10944 \\ \hline 
  107 &   19 & 11104 \\ \hline 
  107 &   20 & 11760 \\ \hline 
  107 &   21 & 11104 \\ \hline 
  107 &   22 & 24, 11136 \\ \hline 
  107 &   23 & 24, 10368 \\ \hline 
  107 &   24 & 24, 64, 10432 \\ \hline 
  107 &   25 & 11664 \\ \hline 
  107 &   26 & 11664 \\ \hline 
  \end{array}
\quad
\begin{array}[t]{|c|c|c|} \hline 
    q &    k & \text{orbit sizes} \\ \hline\hline
  107 &   27 & 11760 \\ \hline 
  107 &   28 & 24, 10816 \\ \hline 
  107 &   29 & 24, 10368 \\ \hline 
  107 &   30 & 11984 \\ \hline 
  107 &   31 & 24, 10496 \\ \hline 
  107 &   32 & 11856 \\ \hline 
  107 &   33 & 24, 10496 \\ \hline 
  107 &   34 & 11200 \\ \hline 
  107 &   35 & 11104 \\ \hline 
  107 &   36 & 11984 \\ \hline 
  107 &   37 & 24, 10368 \\ \hline 
  107 &   38 & 24, 10496 \\ \hline 
  107 &   39 & 11200 \\ \hline 
  107 &   40 & 24, 10496 \\ \hline 
  107 &   41 & 11200 \\ \hline 
  107 &   42 & 24, 11136 \\ \hline 
  107 &   43 & 24, 11008 \\ \hline 
  107 &   44 & 24, 11008 \\ \hline 
  107 &   45 & 24, 11136 \\ \hline 
  107 &   46 & 10944 \\ \hline 
  107 &   47 & 24, 10496 \\ \hline 
  107 &   48 & 24, 288, 10912 \\ \hline 
  107 &   49 & 11984 \\ \hline 
  107 &   50 & 24, 96, 288, 10816 \\ \hline 
  107 &   51 & 11200 \\ \hline 
  107 &   52 & 24, 11200 \\ \hline 
  107 &   53 & 24, 11200 \\ \hline 
\hline
\end{array}
\]
\caption{Non-trivial orbits in $\Wcal_k(\FF_p)$; cf.\ Definition~\ref{definition:smallorbitsizes}}
\label{table:orbitsizeWk-8}
\end{table}

\begin{table}[htbp]
\footnotesize
\[
\begin{array}[t]{|c|c|c|} \hline 
    q &    k & \text{orbit sizes} \\ \hline\hline
  109 &    1 & 24, 48, 12864 \\ \hline 
  109 &    2 & 24^2, 48, 13408 \\ \hline 
  109 &    3 & 24^2, 48, 13632 \\ \hline 
  109 &    4 & 4, 48, 192, 12288 \\ \hline 
  109 &    5 & 24, 48, 192^2, 12224 \\ \hline 
  109 &    6 & 24, 48, 12768 \\ \hline 
  109 &    7 & 48, 12112 \\ \hline 
  109 &    8 & 24^2, 48, 192, 13312 \\ \hline 
  109 &    9 & 24, 48, 12864 \\ \hline 
  109 &   11 & 24^2, 48, 13504 \\ \hline 
  109 &   12 & 48, 192, 288, 11568 \\ \hline 
  109 &   14 & 24, 48, 12768 \\ \hline 
  109 &   15 & 24, 48, 192, 12576 \\ \hline 
  109 &   16 & 24, 48, 192, 12416 \\ \hline 
  109 &   18 & 48, 192^2, 11920 \\ \hline 
  109 &   19 & 48, 12304 \\ \hline 
  109 &   21 & 24, 48, 192^3, 12032 \\ \hline 
  109 &   22 & 24, 48, 160, 12736 \\ \hline 
  109 &   24 & 24, 48, 12864 \\ \hline 
  109 &   25 & 24, 48, 64, 96, 192, 288, 11968 \\ \hline 
  109 &   28 & 24, 48, 192, 12704 \\ \hline 
  109 &   31 & 24^2, 48, 192, 480, 12672 \\ \hline 
  109 &   32 & 24, 48, 192, 12416 \\ \hline 
  109 &   35 & 48, 192, 11920 \\ \hline 
  109 &   38 & 24, 48, 192, 12576 \\ \hline 
  109 &   41 & 48, 12304 \\ \hline 
  109 &   48 & 24^2, 48, 13408 \\ \hline 
\hline
\end{array}
\qquad
\begin{array}[t]{|c|c|c|} \hline 
    q &    k & \text{orbit sizes} \\ \hline\hline
  113 &    1 & 24, 48, 13792 \\ \hline 
  113 &    2 & 48, 96, 192, 12672 \\ \hline 
  113 &    3 & 24^2, 48, 14256 \\ \hline 
  113 &    4 & 4, 24, 48, 6656, 7488 \\ \hline 
  113 &    5 & 24^2, 40, 48, 120, 192, 480, 13456 \\ \hline 
  113 &    6 & 24, 48, 192, 13344 \\ \hline 
  113 &    7 & 48, 13088 \\ \hline 
  113 &    9 & 24, 48, 192, 13504 \\ \hline 
  113 &   10 & 48, 288, 12800 \\ \hline 
  113 &   11 & 24, 48, 160, 192^2, 13152 \\ \hline 
  113 &   12 & 24, 48, 13824 \\ \hline 
  113 &   13 & 24, 48, 192, 13344 \\ \hline 
  113 &   14 & 24, 48, 192, 256, 13344 \\ \hline 
  113 &   17 & 48, 12960 \\ \hline 
  113 &   18 & 24^2, 48, 14288 \\ \hline 
  113 &   19 & 48, 192, 12768 \\ \hline 
  113 &   20 & 24, 48, 192, 288, 13312 \\ \hline 
  113 &   21 & 40, 48, 120, 192^2, 480, 12064 \\ \hline 
  113 &   25 & 16, 24, 48^2, 13728 \\ \hline 
  113 &   26 & 24^2, 48, 192, 14160 \\ \hline 
  113 &   27 & 48, 13088 \\ \hline 
  113 &   28 & 24, 48, 96, 192, 13408 \\ \hline 
  113 &   33 & 24^2, 48, 14256 \\ \hline 
  113 &   34 & 48, 192^3, 12768 \\ \hline 
  113 &   35 & 24, 48, 96, 192, 288^2, 12832 \\ \hline 
  113 &   41 & 24^2, 48, 14448 \\ \hline 
  113 &   42 & 24^2, 48, 14288 \\ \hline 
  113 &   49 & 24, 48, 13824 \\ \hline 
\hline
\end{array}
\]
\caption{Non-trivial orbits in $\Wcal_k(\FF_p)$; cf.\ Definition~\ref{definition:smallorbitsizes}}
\label{table:orbitsizeWk-9}
\end{table}

\clearpage

\end{document}